\documentclass[11pt]{article}

\usepackage[backend=biber,style=alphabetic,sorting=nyt]{biblatex}

\begingroup
\usepackage{authblk}
\usepackage{amsmath,amsthm,amsfonts,amssymb,mathrsfs,epsfig,bm}
\usepackage[colorinlistoftodos]{todonotes}
\usepackage{enumitem}
\setlist{itemsep=2pt,topsep=3pt,leftmargin=\parindent}

\numberwithin{equation}{section}

\usepackage[hidelinks=true]{hyperref}
\usepackage{graphicx}
\usepackage{marginnote}
\usepackage[utf8]{inputenc}

\newtheorem{theorem}{Theorem}[section]
\newtheorem{lemma}[theorem]{Lemma}
\newtheorem{conjecture}[theorem]{Conjecture}
\newtheorem{corollary}[theorem]{Corollary}
\newtheorem{claim}[theorem]{Claim}
\newtheorem{proposition}[theorem]{Proposition}

\newtheorem{remark}[theorem]{Remark}
\theoremstyle{remark}

\newtheorem{example}[theorem]{Example}

\def\N{\mathbb{N}}
\def\Z{\mathbb{Z}}

\def\R{\mathbb{R}}

\def\P{\mathbb{P}}
\def\E{\mathbb{E}}
\def\M{\mathcal{M}}

\def\BB{\mathscr{B}}

\def\LL{\mathscr{L}}

\renewcommand{\phi}{\varphi}
\renewcommand{\epsilon}{\varepsilon}

\newcommand{\1}{{\text{\Large $\mathfrak 1$}}}
\newcommand{\ind}{\1}
\newcommand{\dd}{\mathrm{d}}

\newcommand{\var}{\operatorname{Var}}

\renewcommand{\d}{\text{\rm\,d}}

\definecolor{mygray}{gray}{0.9}
\definecolor{deeppink}{RGB}{255,20,147}
\definecolor{mygreen}{rgb}{0.05, 0.576, 0.03}
\definecolor{myred}{rgb}{0.768, 0.09, 0.09}

\long\def\symbolfootnote[#1]#2{\begingroup
\def\thefootnote{\fnsymbol{footnote}}\footnote[#1]{#2}\endgroup}

\newcommand{\mm}{\sigma^2}
\newcommand{\q}{\overline{q}}
\newcommand{\gep}{\varepsilon}
\begin{document}

\pagenumbering{arabic}
\title{Collective vs.\ individual behaviour for sums of i.i.d.\ random variables: appearance of the one-big-jump phenomenon}

\author{Quentin Berger\footnote{Sorbonne Universit\'e and Universit\'e Paris Cit\'e, CNRS, Laboratoire de Probabilit\'es, Statistique et Modélisation, F-75005 Paris, France.
Email: \texttt{quentin.berger@sorbonne-universite.fr}} \footnote{DMA, École Normale Supérieure, Université PSL, CNRS, 75005 Paris, France.},
Matthias Birkner\footnote{Johannes-Gutenberg-Universit\"{a}t Mainz, Institut f\"{u}r Mathematik.
Staudingerweg 9, 55099 Mainz, Germany. Email: \texttt{birkner@mathematik.uni-mainz.de}} {} and
Linglong Yuan\footnote{University of Liverpool,
Department of Mathematical Sciences.
Peach Street, L69 7ZL, Liverpool, UK. Email: \texttt{linglong.yuan@liverpool.ac.uk}}
}
\maketitle

\begin{abstract}
\noindent
This article studies large and local large deviations for sums of i.i.d.\ real-valued random variables in the domain of attraction of an $\alpha$-stable law, $\alpha\in (0,2]$, with emphasis on the case $\alpha=2$. There are two different scenarios: either the deviation is realised via a collective behaviour with all summands contributing to the deviation (a Gaussian scenario), or a single summand is atypically large and contributes to the deviation (a one-big-jump scenario). Such results are known when $\alpha \in (0,2)$ (large deviations always follow a one big-jump scenario) or when the random variables admit a moment of order $2+\delta$ for some $\delta>0$. We extend these results, including in particular the case where the right tail is regularly varying with index $-2$ (treating cases with infinite variance in the domain of attraction of the normal law). We identify the threshold for the transition between the Gaussian and the one-big-jump regimes; it is slightly larger when considering local large deviations compared to integral large deviations. Additionally, we complement our results by describing  the behaviour of the sum and of the largest summand conditionally on a (local) large deviation, for any $\alpha\in (0,2]$, both in the Gaussian and in the one-big-jump regimes.
As an application, we show how our results can be used in the study of condensation phenomenon in the zero-range process at the critical density, extending the range of parameters previously considered in the literature.
\end{abstract}


\section{Introduction}
\label{intro}

Let $\xi$ be a real-valued 
random variable and let $\xi_1,\xi_2,\ldots$ be independent and identically distributed (i.i.d.) copies drawn from the distribution of $\xi$. We denote, for $x\in \R$, 
\[
F(x):=\P(\xi\leq x), \qquad \overline F(x):=1-F(x)\,.
\]
The general theme of the present paper is to study the interplay between 
\[
S_n:=\sum_{i=1}^n\xi_i \quad \text{and} \quad M_n:=\max\{\xi_1,\xi_2,\ldots,\xi_n\},
\]
when one of them is atypically large.

We assume that $\xi$ is in the domain of attraction of an $\alpha$-stable law, with $\alpha\in (0,2]$, \textit{i.e.}\ that there exist sequences $(a_n)_{n\geq 1}, (b_n)_{n\geq 1}$ such that $(S_n)_{n\geq 1}$
satisfies
\begin{equation}
\label{attract}
\frac{S_n-b_n}{a_n}\xrightarrow[n\to\infty]{d} \mathcal S_{\alpha},
\end{equation} 
where $\mathcal S_{\alpha}$ is an $\alpha$-stable random variable and $\xrightarrow[]{d}$ denotes convergence in distribution.
From Feller~\cite[Ch.~XVII.5, Thm.~2]{F71}, a necessary and sufficient condition is the following:
\begin{itemize}
    \item If $\alpha \in (0,2)$: there is a slowly varying function $L(\cdot)$ such that 
\begin{equation}
\label{tails}
  \overline F(x) \sim p L(x)x^{-\alpha} \,, \quad F(-x) \sim q L(x) x^{-\alpha} \quad \text{ as } x\to\infty \,,
\end{equation}
with $p,q\geq 0$, $p+q=1$; if $p=0$, we interpret \eqref{tails} as $\overline F(x) = o(L(x) x^{-\alpha})$ and similarly if $q=0$.
Note that in this case we have that $\overline F(x)+F(-x) \sim L(x) x^{-\alpha}$ as $x\to\infty$.
    \item If $\alpha =2$: $\xi$ has a finite expectation, that we denote by $\mu=\E[\xi]$, and the truncated variance
    \begin{equation}
    \label{m2}
    \mm(x) := \E\big[(\xi-\mu)^2\ind_{\{|\xi-\mu|\leq x\}}\big] \, , \quad x\geq 0
    \end{equation} 
    is slowly varying as $x\to\infty$. Note that this contains the case of a finite second moment, \textit{i.e.}\ $\lim_{x\to\infty}\mm(x) = \E[(\xi-\mu)^2] =:\mm$.
\end{itemize}
In the case $\alpha=2$, let us recall the fact that  (see \cite[Ch.~IX.8, Eq.~(8.5)]{F71})
\begin{equation}\label{f<sigma}
\overline F(x)+F(-x)=o(x^{-2}\sigma^2(x)), \quad \text{ as } x\to\infty.
\end{equation}
The normalising sequence $(a_n)_{n\geq 1}$ verifies, as $n\to\infty$
\begin{equation}
\label{def:an}
 L(a_n) a_n^{-\alpha} \sim n^{-1} \ \ \text{ if } \alpha\in (0,2); \qquad   \mm(a_n) a_n^{-2} \sim n^{-1}  \ \ \text{ if } \alpha= 2\,,
\end{equation}
and the centering sequence $(b_n)_{n\geq 1}$ is given by 
\begin{equation}\label{bn}
b_n=0 \   \text{ if } \alpha\in(0,1);\quad b_n=n\E\left[\xi\ind_{\{|\xi|\leq a_n\}}\right] \ \text{ if } \alpha=1; \quad
b_n = n\E[\xi] \ \text{ if } \alpha\in (1,2]\,.
\end{equation}
Note that we can also replace \eqref{def:an} by 
$L(a_n)a_n^{-\alpha}\sim an^{-1}, \mm(a_n) a_n^{-2} \sim an^{-1}$ in the different cases for any $a>0$. The value of $a$ only changes the law of $\mathcal S_{\alpha}$ by a dilation factor.

\subsection*{The one-big-jump phenomenon}
The ``one-big-jump phenomenon'' asserts that the large deviation event of having $S_n$ unusually large
may be realised essentially thanks to exactly one of the variables $\xi_1,\ldots, \xi_n$ being very large.
This is captured in the large deviation behaviour of the probability 
\begin{equation}
\label{bigjump-probab}
\P(S_n - b_n\geq x_n) \sim \P(S_n-b_n\geq x_n, M_n\geq x_n) \sim \P(M_n\geq x_n) \quad \text{as } n\to\infty \,,
\end{equation}
where $x_n$ is a sequence going to infinity. We also refer to Theorem~\ref{thm:bigjump} and Corollary~\ref{cor:objnot=2} below for a more detailed interpretation of \eqref{bigjump-probab}.
Finding conditions on the distribution of~$\xi$ and on the sequence $(x_n)_{n\geq 1}$ for~\eqref{bigjump-probab} to hold has a long story, in particular within the class of subexponential distributions\footnote{Recall that a distribution is called \textit{subexponential} if for any $y\in \mathbb R$ we have $\lim_{x\to\infty} \overline F(x+y) /\overline F(x) =1$  and $\P(S_2\geq x)\sim 2\overline F(x)$ as $x\to\infty$. See for example \cite{FKZ11} for further information and background.}, see e.g.~\cite{CH89,N79,MN98} or \cite{DDS08} for a more recent reference.
The one-big-jump phenomenon occurs in a large deviation regime, \textit{i.e.}\ when $\P(S_n-b_n\geq x_n)\to 0$, which is equivalent to 
$\lim_{n\to\infty}\frac{x_n}{a_n} =\infty$, so we will focus on this regime henceforth.

\begin{example}
\label{ex:first}
As a first example, consider a \textit{centered} random variable $\xi$, \textit{i.e.}\ $\mu=\E[\xi]=0$, with finite variance $\sigma^2 = \var(\xi)$, that verifies
\begin{equation}
\label{powertail}
\overline F(x) := \P(\xi >x) \sim  L(x) x^{-\beta} \quad \text{as } x\to\infty \,,
\end{equation}
for some $\beta \in (2,\infty)$ and some slowly varying function $L(\cdot)$. 
In this case, Nagaev~\cite[Thm.~1.9]{N79} shows that if in addition $\E[|\xi|^{2+\delta}]<\infty$ for some $\delta>0$, then for any sequence $(x_n)_{n\geq 1}$ such that $x_n\geq \sqrt{n}$, we have 
\begin{equation}
\label{NagaevLD}
\P(S_n\geq x_n) = (1+o(1)) \overline\Phi(\tfrac{x_n}{\sigma \sqrt{n}})  +  (1+o(1)) n \overline F(x_n),\quad n\to\infty  
\end{equation}
where $\overline\Phi$ is the tail probability function of the standard normal distribution.
Let us set 
\[
\gamma_n := \frac{x_n^2}{2\sigma^2 n} - \Big(\frac{\beta}{2}-1\Big) \log n - \frac12 (\beta-1) \log\log n  + \log L(\sqrt{n\log n})\,,
\]
so that in particular $x_n^2 \sim \sigma^2  (\beta -2)n\log n$ if $\gamma_n = o(\log n)$ (which is the most interesting case).
Then, using the asymptotics $\overline\Phi(t) \sim \frac{1}{t\sqrt{2\pi}}e^{-t^2/2}$ as $t\to\infty$, a straightforward (but tedious) calculation gives that if $\lim_{n\to\infty} \gamma_n = \gamma_{\infty} \in [-\infty,\infty]$, we obtain
\[
\lim_{n\to\infty} \frac{\P(S_n\geq x_n)}{n \overline F(x_n)} = 1 + c_{\beta,\sigma} e^{-\gamma_{\infty}}, \quad \text{ with } c_{\beta,\sigma} = \frac{\sigma^{\beta}}{\sqrt{2\pi}} (\beta-2)^{(\beta-1)/2} \,.
\]
In particular, the one-big-jump phenomenon~\eqref{bigjump-probab} holds if and only if $\lim_{n\to\infty} \gamma_n =\infty$. We will discuss how the condition that $\E[|\xi|^{2+\delta}]<\infty$ for some $\delta>0$ (which in this example, in view of \eqref{powertail}, is a condition on the left tail) and also the regular variation condition\ \eqref{powertail} for the right tail can be weakened, see~\cite[Thm.~6]{R90} or Theorem~\ref{thm:rozo6} below.
\qed
\end{example}
Local versions of~\eqref{bigjump-probab}, known as local large
deviations in the one-big-jump regime, have also been studied,
see~\cite{B19,CD19,DDS08,D89,D97}.  One usually needs to make stronger
assumptions on the distribution: with reference to
Example~\ref{ex:first}, a local version of~\eqref{powertail} has been
considered:  
$\P(\xi = x) \sim \beta L(x) x^{-(1+\beta)}$ with $\beta>2$ as $x\to\infty$, see
e.g.~\cite{B19,D89,D97} or Example~\ref{ex:second} below. 
\smallskip

Observe in Example~\ref{ex:first} that for distributions in the domain
of attraction of the normal law which have a suitably heavy right
tail, depending on the magnitude $x_n \gg \sqrt{n}$ of the deviation that
one studies, collective or one-big-jump behaviour can occur; in fact,
it is possible that both effects contribute on an equal
footing. Moreover, the changeover regime can be characterised
precisely.  The main new results of this paper, presented in the
following Section~\ref{sect:normal}, show that this situation (for both
integral and local large deviations) occurs for a wide class of (necessarily
subexponential) distributions in the domain of attraction of the
normal law. Specifically, we sharpen the profile of a
result by Rozovskii\ \cite{R90}, which shows how the Gaussian term and
the one-big-jump term emerge in the large deviation probability
$\P(S_n-b_n\geq x_n)$ in the case $\alpha=2$. 
We (partly) re-prove this under
weaker assumptions in a concise way and extract sharp conditions for the
changeover between collective and one-big-jump behaviour from it.
We also establish the same result in the local setting, generalising some results of Doney~\cite{D89, D01}.
Furthermore, we make the behaviour of the individual
summands in the different cases explicit, corroborating the intuition
behind the computations.

For completeness, we discuss in
Section~\ref{sect:alpha<2} the case of attraction to a stable law with
index $\alpha\in(0,2)$. This is well understood and in some sense
``cleaner'': Large deviations are always realised by the one-big-jump
behaviour, \textit{i.e.}\ the latter occurs if and only if
$x_n/a_n \to \infty$ (for the sufficiency of the criterion see
\cite{CH89,N79,B19} and references therein; we complement this by
proving necessity). We will comment on applications of our results,
especially to the zero-range process, in Section~\ref{sec:maincomments}.

\section{Main results: One-big-jump vs collective behaviour in the domain of attraction of the normal law}
\label{sect:normal}

Consider the case where $\xi$ is in the domain of attraction of the normal law, that is $\alpha=2$.
Recall from~\eqref{m2} that it corresponds to the truncated variance $\mm(x)=\E[(\xi-\mu)^2 \ind_{\{|\xi-\mu|\leq x\}}]$ defined in~\eqref{m2} being a slowly varying function.

This case is interesting since there is an interplay between the one-big-jump scenario and the Gaussian scenario, as seen in Example~\ref{ex:first}.
We will need to make some assumption on the right tail of the distribution, which in particular implies subexponentiality (but is weaker than having a regularly varying tail), see~\eqref{def:intermediateregvar} and the comments and examples below it.

Before we turn to our new results, we recall and discuss a result by Rozovskii~\cite{R90}, which provides some sharp large deviation estimates in that case; the assumption required is weaker than $\E[|\xi|^{2+\delta}] <\infty$ assumed in~\cite[Thm.~1.9]{N79}.
Afterwards, in Section~\ref{sec:largedev} (for integral large deviations) and Section~\ref{sec:locallargedev} (for local large deviations), we provide additional (and extended) results that help us understand the transition from the Gaussian to the one-big-jump regime, especially in the case $\E[\xi^2]=\infty$, for which we say the distribution of~$\xi$ is in the {\it non-normal domain of attraction to the normal law}. In Sections~\ref{subsect:condlaws}, \ref{subscet:condlocallaw} and \ref{sect:condlawofmax}, we then deduce asymptotics of the conditional law of the summands given these large deviations.

\subsection{Rozovskii's theorem and a few comments}

We now state the main part of Theorem~6 in Rozovskii's \cite{R90} which is to our knowledge one of the sharpest results so far for large deviations of random walks in the domain of attraction of the normal law; it is particularly interesting in the non-normal domain of attraction to the normal law.
Let us introduce 
\begin{equation}
\label{defq}
    q(x) := \frac{x^2}{\mm(x)} \overline F(x) \,.
\end{equation}
As noticed above in~\eqref{f<sigma}, we necessarily have that $\lim_{x\to\infty} q(x)= 0$. The idea in \cite{R90} is not to put conditions on $\overline F(x)$, but rather on the function $q(x)$. 
Rozovskii assumes that there is some $c>0$ such that $x^{c}q(x)$ is asymptotically equivalent to a non-decreasing function (in Rozovskii's notation, $x^cq(x)\uparrow$).
We prove in Section~\ref{sec:asympincreasing} that it is equivalent to having that
\begin{equation}
\label{cond:F0}
\exists\ c>0 \ \text{ s.t. }  x^c \overline F(x) \text{ is equivalent to a non-decreasing function as } x\to\infty.
\end{equation}

\begin{theorem}[Main part of Thm.~6 in \cite{R90}]\label{thm:rozo6}
Let $\alpha=2$ and $(a_n)_{n\geq 1}$ be a normalising sequence as in~\eqref{def:an}, and assume \eqref{cond:F0}.
Define $\omega_n:=\omega_n(a_n)=a_n/\sqrt{|\log q(a_n)|}$.
Then the following two relations are equivalent:
\begin{equation}\label{barvn}
  \sup_{x\geq a_n} \Big| \frac{\P(S_n-b_n\geq x)}{\overline\Phi(x/a_n)+n\overline F(x)} -1  \Big| \xrightarrow[n\to\infty]{} 0 \,,
\end{equation}
    \begin{equation}\label{cond}
    nF(-\omega_n)+\left|\frac{n }{\omega_n^2} \mm(\omega_n) -\frac{a_n^2}{\omega_n^2}\right| \xrightarrow[n\to\infty]{} 0 \,.
\end{equation}
\end{theorem}

\noindent
 
Note that the condition~\eqref{cond} gives a condition on the left tail of the distribution as well.

\paragraph{Some comments on Rozovkii's theorem.} We now make several comments: we refer to Appendix~\ref{sec:discussR} for more discussions (and for some details on the following claims).
\begin{description}[leftmargin=0pt]

\item[1.] The convergences in~\eqref{barvn} and \eqref{cond} are sensitive to the choice of the normalising sequence $(a_n)_{n\geq 1}$.
For instance, to make sure that $\overline \Phi(x_n/a_n) \sim \overline \Phi(x_n/a_n')$ if $x_n/a_n \to\infty$, it is not enough to have $a_n\sim a_n'$. 
Indeed, using the standard asymptotics $\overline \Phi(t) \sim \frac{1}{t\sqrt{2\pi}} e^{-\frac12 t^2}$ as $t\to\infty$, one needs to have $(\frac{x_n}{a_n'})^2 - (\frac{x_n}{a_n})^2=o(1)$,  or equivalently $ a_n'=a_n (1+o(a_n^2/x_n^2))$. 
Hence, there are cases where~\eqref{barvn} and \eqref{cond} will hold for some choices of the normalising sequence $a_n$ but not for other ones, see Examples~\ref{ex:lefttail} and~\ref{ex:lefttail2} in the Appendix. 

  \item[2.] As implied by the above comment, the condition \eqref{cond} is less demanding for some choices of $a_n$ than others.  In view of~\eqref{def:an}, a natural choice of $a_n$ would be to define it by 
\begin{equation}
\label{an}
 a_n^2=n\overline \sigma^2( a_n) \,,
\end{equation}
with $\overline\sigma^2(x) := \E[(|\xi-\mu| \wedge x)^2 ]$; note that such a definition is always possible since $\overline\sigma^2$ is increasing and continuous.
Let us mention that we have replaced here $\sigma^2(\cdot)$ by $\overline \sigma^2(\cdot)$  because $\sigma^2(\cdot)$ is not necessarily continuous, but $\overline \sigma^2(y)\sim\sigma^2(y)$ as $y\to\infty$ (see \eqref{ddbar}).

With the choice~\eqref{an}, one can verify that $a_n$ is a normalising sequence satisfying~\eqref{def:an} and we show in~Proposition~\ref{thm:bn8equiv} (see also the comment below \cite[Thm.~3b]{rozovskii1994probabilities}) that the condition~\eqref{cond} is then equivalent to
\begin{equation}
\label{final}
\sigma^2\left(x\sqrt{|\log q(x)|}\right)- \sigma^2(x)=o\left(\frac{ \sigma^2(x)}{|\log q(x)|}\right),\quad \text{ as } x\to\infty \,.
\end{equation}

\item[3.] As stated in~\cite[Thm.~6]{R90} (or the comment below \cite[Thm.~3b]{rozovskii1994probabilities}), if $\mathrm{Var}(\xi)=\sigma^2<
\infty$ and if one chooses $a_n=\sigma\sqrt{n}$, the condition~\eqref{cond} is equivalent to
\begin{equation}\label{sqrtn}
\sigma^2 - \mm(x)  = \E[(\xi-\mu)^2 \ind_{\{|\xi-\mu|>x\}}]= o\Big(\frac{1}{\log q(x)}\Big),\quad  \text{ as } x\to\infty \,.
\end{equation}
We give in Examples~\ref{ex:lefttail}-\ref{ex:lefttail2} instances where $\mathrm{Var}(\xi)=\sigma^2<\infty$ and \eqref{sqrtn} is satisfied (\textit{i.e.}\ \eqref{barvn} holds for $a_n=\sigma\sqrt{n}$), but Nagaev's criterion fails, \textit{i.e.}\ $\E[|\xi|^{2+\delta}] =\infty$ for all $\delta>0$. On the other hand, if Nagaev's conditions hold, then \eqref{sqrtn} is true, meaning that Rozovskii's theorem implies Nagaev's result.

Let us stress that, with respect to comment 2 above, in the case $\mathrm{Var}(\xi)<+\infty$, the condition~\eqref{final} is strictly weaker than~\eqref{sqrtn}; hence the choice $a_n$ in~\eqref{an} is better, see Examples~\ref{ex:lefttail} and~\ref{ex:lefttail2}.
We also provide  an example where $\E[\xi^2]=\infty$ (hence~\eqref{sqrtn} fails) and  \eqref{final} is verified (hence \eqref{cond} holds) for $a_n$ defined by \eqref{an}, see Example~\ref{ex:central}.

\item[4.] Finally, there exist distributions such that, for any normalising sequence $(a_n)_{n\geq 1}$ satisfying~\eqref{def:an}, the condition~\eqref{cond} does not hold (hence~\eqref{barvn} fails), see Example~\ref{condnottrue}.
\end{description}

\smallskip 

It is useful to keep the following example in mind, which generalises Example~\ref{ex:first}  by including the case $\beta=2$.

\begin{example}\label{example:generalised}
Consider a \emph{centered} random variable $\xi$, with a right tail that verifies
\begin{equation}
\label{tailright}
\overline F(x) := \P(\xi>x) \sim  L(x) x^{-\beta} \quad \text{as } x\to\infty \,,
\end{equation}
for some $\beta\in [2,\infty)$ and some slowly varying function $L(\cdot)$. Note that compared to 
Example~\ref{ex:first}, the case $\beta=2$ is included.
Let us stress that in general, it might not be easy to verify Rozovskii's condition~\eqref{cond}, but here a simple sufficient condition ---~with the normalising sequence $(a_n)_{n\geq 1}$ given by~\eqref{an}~--- is that the left tail verifies $F(-x) \sim c \overline F(x)$ for some $c\geq 0$ (by convention $F(-x)=o(\overline F(x))$ if $c=0$); we refer to Example~\ref{ex:central} in the Appendix for details.

Now, if Rozovskii's condition~\eqref{cond} holds (note that $q(x)\sim x^{2-\beta} L(x)/\mm(x)$ so $q(x)$ is slowly varying if $\beta=2$), we can check that setting\footnote{Note that $x_n^2 \sim 2a_n^2 |\log q(a_n)|$ if $\gamma_n=o(|\log q(a_n)|)$, which is the most interesting case.}
\begin{equation}\label{ourgamma}
\gamma_n = \frac{x_n^2}{2 a_n^2} - |\log q(a_n)| - \frac{1}{2} (\beta-1) \log |\log q(a_n)| - \log \Big( \frac{L(a_n)}{L(a_n \sqrt{|\log q(a_n)|})}\Big) \,,
\end{equation}
then applying~\eqref{barvn} we get that if $\lim_{n\to\infty} \gamma_n = \gamma_{\infty} \in [-\infty,\infty]$ we have
\[
\lim_{n\to\infty} \frac{\P(S_n\geq x_n)}{n \overline F(x_n)} = 1 + c_{\beta} e^{-\gamma_{\infty}}, \quad \text{ with } c_{\beta} = \frac{1}{\sqrt{2\pi}}\, 2^{(\beta-1)/2} \,.
\]
Hence, the one-big-jump phenomenon~\eqref{bigjump-probab} occurs if and only if $\lim_{n\to\infty} \gamma_n =\infty$.
Note that this example extends and sharpens \cite[Thm.~1.2]{Mog08}.
\end{example}

\paragraph{Further comments on the right-tail assumption.}

We stress that~\eqref{cond:F0} does not require the full force of regular variation for the right tail $\overline{F}$.
Let us introduce the following generalisations of regular variation, called \textit{extended} and \textit{intermediate} regular variation\footnote{For the sake of completeness, let us mention that there are other generalisations of regular variation: we refer to \cite{BGT89} for an overview.}, introduced respectively by Matuszewska~\cite{Matu64} and Cline~\cite{Cline94}:

\smallskip
(i) A function $f$ is called \emph{extended} regularly varying if for some real numbers $c,d$ (called upper and lower Matuszewska indices) 
    \begin{equation}
    \label{def:extendedregvar}
      \lambda^d \leq \liminf_{x\to\infty} \frac{f(\lambda x)}{f(x)} \leq \limsup_{x\to\infty} \frac{f(\lambda x)}{f(x)} \leq \lambda^c \qquad \text{ for all }\lambda \ge 1  \,.  
    \end{equation}
    In fact, it is known, see \cite[Thm.~2.0.7]{BGT89}, that the bounds hold locally uniformly in the sense that for any $\Lambda>1$,
    $(1+o(1))\lambda^d \le f(\lambda x)/f(x) \le (1+o(1)) \lambda^c$ as $x \to \infty$ uniformly in
    $1 \le \lambda \le \Lambda$.

\smallskip
(ii) A function $f$ is called \emph{intermediate} regularly varying if
    \begin{equation}
        \label{def:intermediateregvar}
    \lim_{\lambda\downarrow 1} \liminf_{x\to\infty} \frac{f(\lambda x)}{f(x)} = \lim_{\lambda\downarrow 1}\limsup_{x\to\infty} \frac{f(\lambda x)}{f(x)} =1 \,.
    \end{equation}

\smallskip
In Appendix~\ref{sec:discussR}, see Claim~\ref{claim:equivalentcondF}, we prove that the condition~\eqref{cond:F0} is equivalent to the fact that $\overline{F}$ is \textit{extended} regularly varying.
In the following we will mostly work assuming the slightly weaker condition that $\overline F$ is \textit{intermediate} regularly varying. We stress that both conditions imply that the law of $\xi$ is subexponential, \textit{i.e.}\ for any $y\in \mathbb R$ we have $\lim_{x\to\infty} \overline F(x+y) /\overline F(x) =1$.
We refer to Section~\ref{sec:asympincreasing} for further comments.

After a careful review of~\cite{R90}, we believe that Rozovskii's Theorem~\ref{thm:rozo6} still holds assuming intermediate regular variation,  in place of extended regular variation. In fact, in a subsequent work~\cite{rozovskii1994probabilities}, Rozovskii considers a slightly different assumption that allows lighter tails, going somewhat beyond intermediate regular variation, see Theorem~3b there. However~\cite{R90,rozovskii1994probabilities} are very intricate so we are not confident enough to make a definite claim.

\begin{example}
\label{example2}
A classical example where $\overline F$ is intermediate regularly varying but not regularly varying at $\infty$ is the following:
\[
\overline F(x) \sim x^{-\gamma +\sin(\log \log x) }  \,,
\]
where $\gamma$ has to be larger than $\sqrt{2}$ in order for the r.h.s.\ to be non-increasing in $x$ (computing the derivative, a factor $\sin(\log \log x)+\cos(\log \log x) - \gamma$ appears).
Note that the upper and lower Matuszewska indices are then $-\gamma+1$ and $-\gamma-1$ respectively.
Let us stress that for~$\overline{F}$ to be in the domain of attraction of a normal law, one also needs to have $\gamma\geq 3$.\qed
\end{example}

\subsection{First set of results: Large deviations and conditional laws}
\label{sec:largedev}

We now collect a few results on large deviation probabilities for $S_n$. Our main assumption, aside from the fact that $\xi$ is in the domain of attraction of a normal law, will be that $\overline{F}$ is intermediate regularly varying, \textit{i.e.}\ unless otherwise specified we assume that $\overline{F}$ satisfies~\eqref{def:intermediateregvar}.
We will point out explicitly when we need the stronger condition~\eqref{def:extendedregvar} of extended regular variation.
In all the following, we assume that $\alpha=2$ and we let $(a_n)_{n\geq 1}$ be a normalising sequence as in~\eqref{def:an}.

\subsubsection{Large deviations}

Let us consider a function $r: (0,\infty) \to \mathbb R_+$ that satisfies the asymptotic relation $\frac{r(t)}{\mm(r(t))} \sim t$ as $t\to\infty$.
Such a function exists because $\sigma^2(x)$ is slowly varying at infinity and thus we have $\lim_{t\to\infty}r(t)=\infty$. 

Then for any positive sequence $(x_n)_{n\geq 1}$ we define
\begin{equation}\label{def:rrn}
r_n := r\Big(\frac{n}{x_n}\Big).
\end{equation}
In fact, we will mostly focus on the case where $\lim_{n\to\infty} \frac{x_n}{n} =0$, in which case $r_n$ verifies 
\begin{equation}
\label{def:rn}
\lim_{n\to\infty} r_n =\infty\,;\qquad \frac{r_n}{\mm(r_n)} \sim \frac{n}{x_n} \,.
\end{equation}
In order to better explain the estimates of $\P(S_n-b_n\geq x_n)$, we make the following decomposition: 
\begin{equation}\label{eq:decomp}
\P(S_n -b_n\geq x_n) = \P(S_n -b_n \geq x_n , M_n \leq r_n ) + \P(S_n -b_n \geq x_n , M_n > r_n ).
\end{equation}

\paragraph{Analysis of the first term on the r.h.s.\ of~\eqref{eq:decomp}.}
Let us introduce, for $u\geq 0$, the tilted (and truncated) distribution $\P_u = \P_u^{(r_n)}$ defined by
\begin{equation}
\label{def:Pu}
    \frac{\dd \P_u}{\dd  \P} (x) = \frac{1}{M(u)} e^{u(x-\mu)}\ind_{(-\infty,r_n]}(x)  \quad \text{ with }
    M(u) =M_{r_n}(u) :=  \E\left[e^{u  (\xi-\mu)} \ind_{\{\xi\leq r_n \}}\right] \,.
\end{equation}
Define also, for $u\geq 0$, \begin{equation}\label{eqn:mu}m(u) =\frac{M'(u)}{M(u)} = \E_u[\xi-\mu]\end{equation}
and its inverse ($u\mapsto m(u)$ is increasing since $m'(u)= \mathbb{V}\mathrm{ar}_u(\xi)>0$ for any $u\geq 0$)
\begin{equation}\label{eqn:lambdat}\lambda(t) := m^{-1}(t)\end{equation} 
for $t\geq  \E_0[\xi-\mu]$ (note that $\E_0[\xi-\mu]<0$), so that $\E_{\lambda(t)}[\xi] =\mu+t$.
Finally, let 
\begin{equation}
    \label{entropy}
    H(t) = H_{r_n}(t) := -\log M(\lambda(t)) + t \lambda(t) \,,
\end{equation}
which is the relative entropy of $\P_{\lambda(t)}$ w.r.t.\ to $\P$. Note that we allow a small abuse of notation here since $\P_u$ will later refer both to the law of a single variable and to the law of~$n$ i.i.d.\ copies of that variable.
With this notation we can state the following result.
\begin{proposition}
\label{prop:tilting}
Let $\alpha=2$ and $(a_n)_{n\geq 1}$ be a normalising sequence as in~\eqref{def:an}. Let $(x_n)_{n\geq 1}$ be a sequence such that $x_n\geq a_n$ and $\lim_{n\to\infty}  \frac{x_n}{n} =0$, and let~$r_n$ be as in \eqref{def:rrn}. Then
\begin{equation}\label{eq:tilting1}
\P(S_n -b_n \geq x_n , M_n \leq  r_n ) \sim \frac{1}{\sqrt{2\pi}}  \frac{a_n}{x_n} \sqrt{\frac{\mm(r_n)}{\mm(a_n)}}\, e^{- n H(\frac{x_n}{n})} \quad \text{ as } n\to\infty\,.
\end{equation}\end{proposition}

\begin{remark}\label{rem:varrn=an}
Let us stress that if $a_n\leq x_n\leq C a_n\sqrt{|\log q(a_n)|}$ for some constant $C>0$, then assuming that $\overline{F}$ is intermediate regularly varying (see~\eqref{def:intermediateregvar}) we get that $\mm(r_n)\sim \mm (a_n)$, see Remark~\ref{a-2a} in the Appendix. 
Note also that~\eqref{eq:tilting1} falls into the scope of the central limit theorem when $x_n=O(a_n)$.
\end{remark}

We also provide the following lemma that estimates the relative entropy $nH(\frac{x_n}{n})$.

\begin{lemma}
\label{lem:entropy}
Under the assumptions of Proposition~\ref{prop:tilting}
we have
\[
nH\left(\frac{x_n}{n} \right) = (1+o(1)) \frac12 \frac{x_n^2}{n\mm(r_n)}  = (1+o(1)) \frac12 \frac{x_n^2}{a_n^2} \frac{\mm(a_n)}{\mm(r_n)} .
\]
\end{lemma}

\begin{remark}
\label{rem:rozov}
Note that if $\frac{x_n}{a_n} =O(1)$, then $nH(\frac{x_n}{n}) = (1+o(1)) \frac{x_n^2}{2a_n^2} =  \frac{x_n^2}{2a_n^2}+o(1)$ (also using that $\sigma^2(r_n)\sim\sigma^2(a_n)$, see Remark~\ref{rem:varrn=an}). 
Rozovskii's theorem actually tells that, if $\overline{F}$ is extended regularly varying, see \eqref{def:extendedregvar},  then condition~\eqref{cond} is a criterion to ensure that $n H(\frac{x_n}{n}) = \frac{x_n^2}{2a_n^2} +o(1)$ uniformly for $a_n\leq x_n\leq C a_n \sqrt{|\log q(a_n)|}$; the r.h.s.\ of~\eqref{eq:tilting1} is then asymptotically equivalent to $\overline \Phi(\frac{x_n}{a_n})$. Note that the threshold $a_n \sqrt{|\log q(a_n)|}$ appears when comparing $\overline \Phi(\frac{x_n}{a_n})$ with $n\overline{F}(x_n)$: if $x_n \geq C a_n \sqrt{|\log q(a_n)|}$ with $C>\sqrt{2}$ then~\eqref{eq:tilting1} becomes negligible compared to $n\overline{F}(x_n)$, see Remark~\ref{rem:threshold} in the Appendix.
\end{remark}

\paragraph{Analysis of the second term on the r.h.s.\ of~\eqref{eq:decomp}.}
For the remaining term in~\eqref{eq:decomp}, we have the following result.
\begin{proposition}
\label{prop:decomp}
Let $\alpha=2$ and $(a_n)_{n\geq 1}$ be a normalising sequence as in~\eqref{def:an}.
Assume that $\overline{F}$ is intermediate regularly varying, see~\eqref{def:intermediateregvar}.
If $\lim_{n\to\infty} \frac{x_n}{a_n} =\infty$ and $\lim_{n\to\infty}  \frac{x_n}{n} =0$, then letting~$r_n$ be as in \eqref{def:rrn}, we have
\[
\P(S_n -b_n\geq x_n, M_n >r_n) \sim \P(S_n-b_n\geq x_n, M_n\geq x_n) \,,
\]
or equivalently
\begin{equation}\label{<rn>x}
\P(S_n -b_n\geq x_n) = \P(S_n -b_n \geq x_n , M_n \leq r_n ) + (1+o(1)) n \overline F(x_n) \,.
\end{equation}
If $\liminf_{n\to\infty} \frac{x_n}{n} >0$, then 
$\P(S_n -b_n \geq x_n) \sim \P(S_n -b_n \geq x_n, M_n\geq x_n)\sim n \overline F(x_n)$.
\end{proposition}
The statement with $\liminf_{n\to\infty} \frac{x_n}{n} >0$ is standard, see e.g.~\cite{DDS08}. We include it in the proposition for the sake of completeness.

\begin{remark}
  Denisov, Dieker and Shneer\ \cite{DDS08} give conditions for the
  appearance of the one-big-jump phenomenon in the general context of
  subexponential distributions.  Since we work here under much more restricted assumptions, we are able to
   give sharper conditions for the minimal size of $x_n$ to be in
  the one-big-jump regime for specific examples:
  \cite{DDS08} consider in Section~8.2 centered $\xi$'s with $\E[\xi^2]=1$ and 
  certain assumptions on  $\overline{F}$.
  If we assume, as in \cite[Thm.~8.1]{DDS08}, that $\overline{F}$ is in fact intermediate regularly varying and also that~\eqref{sqrtn} holds (which is in this case equivalent to
  \eqref{cond}), we can combine Proposition~\ref{prop:decomp}, Proposition~\ref{prop:tilting} and
  Remark~\ref{rem:rozov} (see also Remark~\ref{rem:threshold} in the Appendix) to see that
  \[
    \limsup_{n\to\infty} \sup_{y \ge c a_n \sqrt{|\log q(a_n)|}} \Big| \frac{\P(S_n > y)}{n \overline{F}(y)} - 1 \Big|
    \; \begin{cases}
      = 0, & \text{if } c > \sqrt{2}, \\
      >0, & \text{if } 0 \le c < \sqrt{2}.
    \end{cases}
  \]
  Using the nomenclature from \cite{DDS08}, we see that in this case if we set
  $x_n = c a_n \sqrt{|\log q(a_n)|}$ with $c>\sqrt{2}$ and define $r_n$  via \eqref{def:rn}, then
  $(r_n)$ is a ``truncation sequence'' and 
  $(x_n)$ is a ``$(r_n)$-small-steps sequence''
  (we can use $\sqrt{n}$ as the ``natural-scale sequence'' and any sequence $(I_n)$ with $\sqrt{n}/I_n \to 0$
  as ``insensitivity sequence'').

  If in this situation we impose that $\overline{F}$ is in fact regularly varying, we can read off the
  precise behaviour from Example~\ref{example:generalised}: A given sequence $(x_n)$ is a
  small-steps sequence in the sense of \cite{DDS08} if and only if the sequence $(\gamma_n)$ defined in
  \eqref{ourgamma} diverges. An analogous situation occurs in the local case, see Example~\ref{example:local}. 
\end{remark}

\subsubsection{Conditional laws on large deviation events}
\label{subsect:condlaws}

Using Propositions~\ref{prop:tilting} and \ref{prop:decomp} (and their proofs), we can make the one-big-jump phenomenon more precise: 
The intuition behind \eqref{bigjump-probab} is of course that only one random variable absorbs the large deviation, without effectively changing the distribution of the $n-1$ others. The following results make the fact that the remaining random variables are ``left untouched'' explicit by showing that their distribution is close to that of $n-1$ i.i.d.\ random variables with the same law as $\xi$.

For $y=(y_1,\dots,y_n) \in \R^n$, we define $R(y) \in \R^{n-1}$,  the vector obtained by removing the $i$-th coordinate from $y$, where $i$ is the index for which the maximum $\max_{1 \le j \le n}|y_j|$ is attained 
\textit{i.e.}\ $|y_i| = \max_{1 \le j \le n}|y_j| =: M_n(y)$. In case of a tie for the maximum, we take away the variable with the index $\min\{i: y_i=M_n(y)\}$. In fact, other choices are possible and will lead to the same limit behaviour as in Theorem~\ref{thm:bigjump} and in Proposition \ref{prop:othersnormalMnl} below. In the following, we write $\mathscr{L}(X)$ to denote the law of a random variable (or a vector of random variables) $X$.

\begin{theorem}[One-big-jump phenomenon]\label{thm:bigjump}
Let $\alpha=2$ and $(a_n)_{n\geq 1}$ be a normalising sequence as in~\eqref{def:an}. 
Assume that $\overline{F}$ is intermediate regularly varying, see~\eqref{def:intermediateregvar}. If $\lim_{n\to\infty} \frac{x_n}{a_n} =\infty$ holds, and then letting~$r_n$ be as in \eqref{def:rrn}, we have
\begin{equation}\label{dtv:big}
 \lim_{n\to\infty} d_{\mathrm{TV}}\Big( \mathscr{L}\big( R(\xi_1,\dots,\xi_n) \, \big| \, S_n -b_n \geq  x_n, M_n>r_n\big), \,
    \big(\mathscr{L}(\xi)\big)^{\otimes (n-1)} \Big) = 0,
    \end{equation}
    \begin{equation}\label{dtv:gaussian}
 \lim_{n\to\infty} d_{\mathrm{TV}}\Big( \mathscr{L}\big( R(\xi_1,\dots,\xi_n) \, \big| \, S_n -b_n \geq  x_n, M_n\leq r_n\big), \,
    \big(\mathscr{L}(\xi)\big)^{\otimes (n-1)} \Big) = 1,
  \end{equation}
  where $d_{\mathrm{TV}}$ denotes the total variation distance.

\noindent  
As a consequence, if in addition $(x_n)$ satisfies $\lim_{n\to\infty}\frac{n\overline F(x_n)}{\P(S_n\geq x_n)}=s\in [0,1]$, 
  then we have 
 \begin{align}
    \label{normalSnlarge}
    \lim_{n\to\infty} d_{\mathrm{TV}}\Big( \mathscr{L}\big( R(\xi_1,\dots,\xi_n) \, \big| \, S_n -b_n \geq  x_n \big), \,
    \big(\mathscr{L}(\xi)\big)^{\otimes (n-1)} \Big) = 1-s \,.
  \end{align}  
If we assume further that $\overline{F}$ is extended regularly varying (see~\eqref{def:extendedregvar}) and that \eqref{cond} holds, then  $s=\lim_{n\to\infty}\frac{n\overline F(x_n)}{n\overline F(x_n)+\overline \Phi(\frac{x_n}{a_n})}$.
\end{theorem}

\noindent
For the value of $s$ in the context of Example~\ref{example:generalised}, recalling the definition~\eqref{ourgamma} of $\gamma_n$,  we have $s= \frac{1}{1+c_\beta e^{-\gamma_\infty}} \in [0,1]$, with $\gamma_{\infty} =\lim_{n\to\infty} \gamma_n \in [-\infty,\infty]$ and $c_{\beta}$ as in Example~\ref{example:generalised}.

\begin{remark}
  \label{rem:TV>0}
If $\lim_{n\to\infty}\frac{x_n}{a_n}\in[-\infty,\infty)$, then it is natural to expect that
  \begin{equation*}
    \lim_{n\to\infty} d_{\mathrm{TV}}\Big( \mathscr{L}\big( R(\xi_1,\dots,\xi_n) \, \big| \, S_n -  b_n \geq x_n \big), \,
    \big(\mathscr{L}(\xi)\big)^{\otimes (n-1)} \Big) >0\,,
  \end{equation*}
since conditioning on a typical fluctuation will affect all the summands.
This is analogous to the situation of Corollary~\ref{cor:objnot=2} below in the case $\alpha \in (0,2)$, but we do not go into further detail here.
\end{remark}

In the course of the proofs of  Propositions~\ref{prop:tilting} and \ref{prop:decomp}, we also obtain information on the overshoot $S_n-b_n-x_n$, conditioned on the large deviation event $\{S_n-b_n\geq x_n\}$.

\begin{corollary}
  \label{cor:overshoot}
 Let $\alpha=2$ and $(a_n)_{n\geq 1}$ be a normalising sequence as in~\eqref{def:an}, and assume that $\lim_{n\to\infty}\frac{x_n}{a_n} =\infty$.  Then we have
\begin{enumerate}
    \item If $\lim_{n\to\infty}\frac{x_n}{n}=0$, then conditionally on having no big-jump ($M_n\leq r_n$),
    \begin{equation}\label{overshoot normal}
    \mathscr{L}\Big( \frac{S_n-b_n-x_n}{\textstyle\sqrt{x_n/r_n}} \, \Big| \, S_n -b_n \geq  x_n, M_n \le r_n \Big)
    \xrightarrow[n\to\infty]{w} \mathrm{Exp}(1) \,
  \end{equation}
  where $\xrightarrow{w}$ denotes weak convergence; 
    \item If $\overline{F}$ is intermediate regularly varying, see~\eqref{def:intermediateregvar}, then conditionally on having a big-jump ($M_n>r_n$), for any sequence $x'_n \ge x_n$ we have, as $n\to\infty$,
  \begin{equation}\label{overshoot big}
    \P\big(S_n - b_n \ge x_n' \mid S_n-b_n \geq x_n , M_n > r_n \big) \sim \frac{\overline{F}(x_n')}{\overline{F}(x_n)} .
  \end{equation}
\end{enumerate}
\end{corollary}

\noindent
In particular, if $\overline{F}$ varies regularly with index $-\beta$ for some $\beta \ge 2$, the last part gives that, for $y \in [1,\infty)$
\begin{equation}\label{overshoot pareto}
   \lim_{n\to\infty} \P\Big( \frac{S_n-b_n}{x_n} \geq y \, \Big| \, S_n -b_n \geq  x_n, M_n > r_n \Big)
    = y^{-\beta}\,,
\end{equation}
\textit{i.e.} conditioned on $\{ S_n -b_n \geq  x_n, M_n > r_n \}$, the overshoot behaves asymptotically like~$x_n$ times a Pareto random variable.
In particular,  
in the regime where $s=\lim_{n\to\infty}\frac{n\overline F(x_n)}{\P(S_n-b_n\geq x_n)} \in (0,1)$, the conditional law of the overshoot will be a non-trivial mixture of a scaled exponential and of a (differently) scaled Pareto distribution.

\subsection{Second set of results: Local large deviations and conditional laws}
\label{sec:locallargedev}

In this section, we study local versions of the large deviations, namely we obtain estimates on the probabilities of the type $\P(S_n - b_n \in  [x_n, x_n+\Delta))$ in the regime $\lim_{n\to\infty}\frac{x_n}{a_n}=\infty$.
To simplify the statements, we will assume that $\xi$ is integer valued and that $(S_n)_{n\geq 1}$ is aperiodic (the latter is true if $\overline F$ is intermediate regularly varying). The case that $\xi$ has support in $a + h \Z$ for some $h>0$ and $a \in \R$ is completely analogous, just notationally a little more cumbersome.

To obtain sharp results, we need some local condition on the distribution of $\xi$. Our main assumption is that 
$\P(\xi=x)$ is intermediate regularly varying at infinity (which implies that $\overline F(x)$ is also intermediate regularly varying), see~\eqref{def:intermediateregvar}, and that $\P(\xi=x)$ is ``almost monotone'', in the sense that
\begin{equation}
    \label{cond:taillocal}
   \limsup_{x\to\infty} \Big( \frac{\sup_{y\geq x} \P(\xi=y)}{\P(\xi=x)} \Big) <+\infty \,. 
\end{equation}
Note that~\eqref{cond:taillocal} is not implied by the intermediate regular variation of $\P(\xi=x)$ (but would be guaranteed if $\P(\xi=x)$ were regularly varying).

\begin{example}\label{ex:tobefollowed}
An important class of examples that we consider is when 
$\P(\xi=x)$ is regularly varying at infinity:
there exists a slowly varying function $L(\cdot)$ and some $\beta \geq 2$ such that
\begin{equation}
\label{eq:localtail}
    \P(\xi = x) \sim \beta L(x) x^{-(1+\beta)} \qquad \text{ as } x\to\infty \,.
\end{equation}
In particular, we have that $\overline F(x)\sim L(x)x^{-\beta}$. Doney~\cite{D01} proves that if in addition to~\eqref{eq:localtail} we have $\E[|\xi|^{2+\delta}]<\infty$ for some $\delta>0$ (implying $\beta>2$),  then analogously to Nagaev's result~\eqref{NagaevLD} for the integral case, we have
\begin{equation}
\label{localDoney}
 \P(S_n - \lfloor b_n \rfloor =x_n) = (1+o(1))\frac{1}{\sigma\sqrt{n}} g\left(\frac{x_n}{\sigma \sqrt{n}}\right) + (1+o(1)) n \P(\xi =x_n) \,,   
\end{equation}
where $g(t)=\frac{1}{\sqrt{2\pi}} e^{-\frac{t^2}{2}}$ is the standard normal density and $\sigma^2 = \mathrm{Var}(\xi)$.
Let us mention that \cite[Thm.~2.2]{Mog08} gives the same result, assuming $\beta>2$ in~\eqref{eq:localtail} and with the moment condition $\E[\xi^2 \ind_{\{|\xi|>x\}}] = o(1/\log x)$ as $x\to\infty$, analogously
to~\eqref{sqrtn}.
\end{example}

\begin{example} 
\label{ex:second}
Continuing the discussion in Example \ref{ex:tobefollowed}, assume that \eqref{eq:localtail} holds with $\beta>2$ and that $\E[\xi^2 \ind_{\{|\xi|>x\}}] = o(1/\log x)$ as $x\to\infty$.
Then since~\eqref{localDoney} holds, it is a standard calculation to obtain that letting
\[
\tilde\gamma_n := \frac{x_n^2}{2\sigma^2 n} - \Big(\frac{\beta}{2}-1\Big) \log n - \frac12 (\beta+1) \log\log n  + \log L(\sqrt{n\log n})
\]
with $\sigma^2=\mathrm{Var}(\xi)$,
then if $\lim_{n\to\infty} \tilde \gamma_n = \tilde \gamma_{\infty} \in [-\infty,\infty]$, we obtain
\[
\lim_{n\to\infty} \frac{\P(S_n -\lfloor b_n \rfloor = x_n)}{n \P(\xi=x_n)} = 1 + \tilde c_{\beta,\sigma} e^{-\tilde \gamma_{\infty}}, \quad \text{ with } \tilde c_{\beta,\sigma} = \frac{\sigma^{\beta}}{\sqrt{2\pi}}  \beta^{-1} (\beta-2)^{(\beta+1)/2} \,.
\]
Hence the one-big-jump phenomenon occurs if and only if $\lim_{n\to\infty}\tilde\gamma_n =\infty$.
Let us stress that, compared with Example~\ref{ex:first}, the definition of $\tilde \gamma_n$ differs from that of $\gamma_n$ in the constant in front of $\log \log n$.
In particular, we might have $\gamma_n\to \infty$ while $\tilde \gamma_n\to-\infty$. In other words, there is a regime where the one-big-jump phenomenon occurs for the large deviations but not for the local large deviations. \qed
\end{example}

Our goal is to extend the asymptotics~\eqref{localDoney} of \cite{D01,Mog08}  to include the non-normal domain 
of attraction to the normal law: 
in particular, in~\eqref{eq:localtail} we include the case $\beta=2$.

\subsubsection{Local large deviations}

As for the (integral) large deviation case, we split the local large deviation probability according to the threshold $r_n$ defined in~\eqref{def:rrn}:
\[
 \P(S_n - \lfloor b_n \rfloor = x_n) 
 = \P( S_n - \lfloor b_n \rfloor = x_n, M_n \leq  r_n) +\P( S_n - \lfloor b_n \rfloor = x_n, M_n >  r_n) \,.
\]
We analyse the two terms through two separate results respectively. 

\begin{proposition}
\label{prop:tiltinglocal}
Let $\alpha=2$ and let $(a_n)_{n\geq 1}$ be a normalising sequence as in~\eqref{def:an}; assume also that $(S_n)_{n\geq 0}$ is aperiodic.
If $(x_n)_{n\geq 1}$ is a sequence such that $x_n\geq a_n$ and $\lim_{n\to\infty} \frac{x_n}{n} = 0$,
then letting $r_n$ be as in~\eqref{def:rrn}, we have as $n\to\infty$
\begin{equation}\label{eq:tilting1local}
\P( S_n - \lfloor b_n \rfloor = x_n, M_n \leq  r_n)  \sim  \frac{1}{\sqrt{2\pi n \mm(r_n)}}  e^{- n H(\frac{x_n}{n})} \,.
\end{equation}
\end{proposition}

\begin{proposition}
\label{prop:decomplocal}
Let $\alpha=2$, let $(a_n)_{n\geq 1}$ be a normalising sequence as in~\eqref{def:an}. Assume that $\P(\xi=x)$ is intermediate regularly varying (see~\eqref{def:intermediateregvar}) and that~\eqref{cond:taillocal} holds. 
If $\lim_{n\to\infty} \frac{x_n}{a_n} =\infty$ and $\lim_{n\to\infty} \frac{x_n}{n} = 0$,
then letting $r_n$ be as in~\eqref{def:rrn}, we have for any $\gep>0$
\begin{equation}
\label{eq:proplocal}
\P(S_n - \lfloor b_n \rfloor = x_n, M_n>r_n) \sim \P\left(S_n - \lfloor b_n \rfloor  = x_n ,  | M_n - x_n |  \leq  \epsilon x_n\right)\,,
\end{equation}
and also
\begin{equation*}
\P(S_n - \lfloor b_n \rfloor = x_n) = \P(S_n - \lfloor b_n \rfloor  = x_n , M_n \leq r_n ) + (1+o(1))n\P(\xi=x_n)\,.
\end{equation*}

\noindent
If $\liminf_{n\to\infty} \frac{x_n}{n} >0$, then we have 
\[
\P(S_n - \lfloor b_n \rfloor = x_n)\sim  \P\left(S_n - \lfloor b_n \rfloor  = x_n ,  | M_n - x_n |  \leq  \epsilon x_n\right)\sim  n\P(\xi=x_n)\,.
\]
\end{proposition}

\noindent
The statement with $\liminf_{n\to\infty} \frac{x_n}{n} >0$ is standard, see e.g.~\cite[Prop.~1]{AL09}. We 
include it in our proposition for the sake of completeness.

Let us mention that Remark~\ref{rem:rozov} also applies here in the local case:
If additionally to the assumptions in Proposition~\ref{prop:decomplocal} condition~\eqref{cond} holds, then thanks to Lemma~\ref{lem:entropy} we have $n H\big(\frac{x_n}{n}\big) = \frac{x_n^2}{2a_n^2} +o(1)$  uniformly for $a_n\leq x_n \leq C a_n \sqrt{|\log q(a_n)|}$.
Therefore, if $\lim_{n\to\infty} \frac{x_n}{a_n}=\infty$ and $x_n\leq C a_n \sqrt{|\log q(a_n)|}$, thanks to Propositions\ \ref{prop:tiltinglocal} and \ref{prop:decomplocal}  we obtain
\begin{equation}
\label{eq:localrozov}
    \P(S_n - \lfloor b_n \rfloor = x_n) = (1+o(1)) \frac{1}{a_n} g \left( \frac{x_n}{a_n} \right) + (1+o(1))n\P(\xi=x_n) \,,
\end{equation}
where we recall $g(\cdot)$ is the standard normal density (and we have used that $\sqrt{n\mm(r_n)}\sim \sqrt{n\mm(a_n)}\sim a_n$ in~\eqref{eq:tilting1local}, see Remark~\ref{rem:varrn=an}).

\begin{example}\label{example:local} 
Let $\xi$ be an integer-valued random variable that satisfies~\eqref{eq:localtail}. We also assume that Rozovskii's condition~\eqref{cond} holds, see Example~\ref{ex:central} for a simple sufficient condition.
Note that compared to Example~\ref{ex:second}, we include the case $\beta=2$.
If we set 
\begin{equation}\label{ourgammalocal}
\tilde\gamma_n = \frac{x_n^2}{2 a_n^2} - |\log q(a_n)| - \frac{1}{2} (\beta+1) \log |\log q(a_n)| - \log \Big( \frac{L(a_n)}{L(a_n \sqrt{|\log q(a_n)|})}\Big) \,,
\end{equation}
then if $\lim_{n\to\infty} \tilde\gamma_n = \tilde\gamma_{\infty} \in [-\infty,\infty]$, thanks to~\eqref{eq:localrozov} we obtain that
\[
\lim_{n\to\infty} \frac{\P(S_n-\lfloor b_n \rfloor = x_n)}{n \P(\xi=x_n)} = 1 + \tilde c_{\beta} e^{-\tilde\gamma_{\infty}}, \quad \text{ with } \tilde c_{\beta} = \frac{1}{\sqrt{2\pi}}\, \beta^{-1} 2^{(\beta+1)/2} \,.
\]
Hence the one-big-jump phenomenon occurs if and only if $\lim_{n\to\infty}\tilde\gamma_n =\infty$.
We stress again that, compared with Example~\ref{example:generalised}, the definition~\eqref{ourgammalocal} of $\tilde \gamma_n$ differs from that~\eqref{ourgamma} of~$\gamma_n$ in the prefactor of $\log|\log q(a_n)|$; in particular we might have $\gamma_n\to\infty$ while $\tilde \gamma_n \to -\infty$.
\qed
\end{example}

\begin{remark}
Let us mention that even without the local assumption that $\P(\xi=x)$ is (intermediate) regularly varying, one can still obtain useful estimates on the local large deviation probabilities.
For instance, assuming that $\alpha=2$ and letting $(a_n)_{n\geq 1}$ be a normalising sequence as in~\eqref{def:an}, 
\cite[Thm.~1.1]{MT22} gives the following bound (with optimal decay rate), in the case where the left and right tails $F(-x), \overline{F}(x)$ are regularly varying: uniformly for $x_n\geq a_n$
\begin{equation}
\label{eq:boundMT}
   \P(S_n - \lfloor b_n \rfloor =x_n) \leq \frac{C}{a_n} \times \frac{n \mm(x_n)}{x_n^2}\,. 
\end{equation}
If one has that $\overline{F}(x) \leq L(x) x^{-\beta}$ for some $\beta\geq 2$ and some slowly varying function $L(\cdot)$, then \cite[Thm.~2.1]{B19-} gives that there are positive constants $c_1,C_1$ such that uniformly in $x_n\geq a_n$,
\[
\P(S_n - \lfloor b_n \rfloor =x_n) 
\leq \frac{C_1}{a_n} \Big(  e^{- c_1 x_n^2/a_n^2} +  n L(x_n) x_n^{-\beta} \Big).
\]
This is obviously not as sharp as \eqref{localDoney} (in particular the constant $c_1$ in the exponential is not optimal) but it is sharper than~\eqref{eq:boundMT} when $x_n\gg a_n$.
The above display requires very little assumption and can be useful in several situations.
\end{remark}

\subsubsection{Conditional laws on local large deviation events}
\label{subscet:condlocallaw}

Recall the definition of the map $y\mapsto R(y)$ for $y=(y_1,\ldots, y_n)\in \mathbb R^n$
from Section~\ref{subsect:condlaws}. 

\begin{theorem}[Local one-big-jump phenomenon]\label{thm:bigjumplocal}
Let $\alpha=2$, let $(a_n)_{n\geq 1}$ be a normalising sequence as in~\eqref{def:an}, and assume that $\P(\xi=x)$ is intermediate regularly varying (see~\eqref{def:intermediateregvar}).
Let $(x_n)_{n\geq 1}$ be a sequence such that $\lim_{n\to\infty} \frac{x_n}{a_n} =\infty$ and let $r_n$ be as defined in~\eqref{def:rrn}.
Then we have 
\[
\begin{split}
 &\lim_{n\to\infty} d_{\mathrm{TV}}\Big( \mathscr{L}\big( R(\xi_1,\dots,\xi_n) \, \big| \, S_n -\lfloor b_n \rfloor =  x_n, M_n>r_n\big), \,
    \big(\mathscr{L}(\xi)\big)^{\otimes (n-1)} \Big) = 0 \,,\\
& \lim_{n\to\infty} d_{\mathrm{TV}}\Big( \mathscr{L}\big( R(\xi_1,\dots,\xi_n) \, \big| \, S_n -\lfloor b_n \rfloor =  x_n, M_n\leq r_n\big), \,
    \big(\mathscr{L}(\xi)\big)^{\otimes (n-1)} \Big) = 1 \,.
\end{split}
\]
As a consequence, if in addition~\eqref{cond:taillocal} holds and $(x_n)$ satisfies $\lim_{n\to\infty}\frac{n\P(\xi=x_n)}{\P(S_n-\lfloor b_n\rfloor=x_n)}=s\in [0,1]$, then we have
 \begin{align}
    \label{normalSnlargelocal}
    &\lim_{n\to\infty} d_{\mathrm{TV}}\Big( \mathscr{L}\big( R(\xi_1,\dots,\xi_n) \, \big| \, S_n - \lfloor b_n \rfloor =  x_n \big), \,
    \big(\mathscr{L}(\xi)\big)^{\otimes (n-1)} \Big) = 1-s. 
  \end{align}
\end{theorem}

\noindent
For the value of $s$ in the context of Example~\ref{example:local}, recalling the definition~\eqref{ourgammalocal} of $\tilde \gamma_n$, we have $s= \frac{1}{1+\tilde c_\beta e^{-\tilde \gamma_\infty}} \in [0,1]$, with
$\tilde \gamma_{\infty} =\lim_{n\to\infty} \tilde \gamma_n \in [-\infty,\infty]$ and $\tilde c_{\beta}$ as in Example~\ref{example:local}.

\begin{remark}
Similarly to Remark~\ref{rem:TV>0},
when we have $\lim_{n\to\infty}\frac{x_n}{a_n}\in[-\infty,\infty)$, it is natural to expect the following (the proof is omitted)
\begin{equation*}
\label{necessitynormalSnlargelocal}
\lim_{n\to\infty} d_{\mathrm{TV}}\Big( \mathscr{L}\big( R(\xi_1,\dots,\xi_n) \, \big| \, S_n - \lfloor b_n \rfloor = x_n \big), \,
\big(\mathscr{L}(\xi)\big)^{\otimes (n-1)} \Big) >0.
\end{equation*}
\end{remark}

\subsection{Conditional law of the maximum on a large deviation event}
\label{sect:condlawofmax}

To complement the previous results, we study the law of the maximum summand conditioned on the sum being large.
We will prove only the first corollary as the second one can be proved following similar arguments.

\begin{corollary}
\label{cor:extendzero1}
Let $\alpha=2$, let $(a_n)_{n\geq 1}$ be a normalising sequence as in~\eqref{def:an} and
assume that $\overline{F}$ is intermediate regularly varying (see~\eqref{def:intermediateregvar}). If $\lim_{n\to\infty} \frac{x_n}{a_n} =\infty$
and if $\lim_{n\to\infty}\frac{n \overline{F}(x_n)}{\P(S_n-b_n\geq x_n)}=s\in [0,1]$, then we have
\[
\LL\left(\frac{M_n}{S_n-b_n}\, \Big|\, S_n-b_n \geq  x_n\right) \xrightarrow[n\to\infty]{w} (1-s)\delta_0 +  s \delta_1\,.
\]
\end{corollary}

\begin{corollary}
\label{cor:extendzero}
Let $\alpha=2$ and $(a_n)_{n\geq 1}$ be a normalising sequence as in~\eqref{def:an} and assume that $\P(\xi=x)$ is intermediate regularly varying (see~\eqref{def:intermediateregvar}) and that~\eqref{cond:taillocal} holds. Assume $\lim_{n\to\infty}\frac{x_n}{a_n} =\infty$. Then, if $\lim_{n\to\infty}\frac{n \P(\xi=x_n)}{\P(S_n-\lfloor b_n \rfloor= x_n)}=s\in [0,1]$, we have
\[
\LL\left(\frac{M_n}{x_n}\, \Big|\, S_n-\lfloor b_n \rfloor = x_n\right) \xrightarrow[n\to\infty]{w} (1-s)\delta_0 +  s \delta_1\,.
\]
\end{corollary}

\noindent
From the above two corollaries, a natural question is to understand the distribution of (rescaled versions of) $M_n$ or $M_n - (S_n-b_n)$ conditionally on a large deviation event (either local or integral).
We discuss this in Section~\ref{sec:comments} later.

\subsection{Organisation of the rest of the paper}
 
Let us give a brief overview on how the rest of the paper is organised:
\begin{itemize}
\item In Section~\ref{sect:alpha<2}, we study 
the summands in the domain of attraction of a stable law with 
index $\alpha < 2$.
\item In Section~\ref{sec:maincomments}, we discuss some applications and possible extensions of our results.
\item In Section~\ref{sec:proofnormal}, we prove the main large deviation results of Section~\ref{sec:largedev} above. After some preliminary estimates, we prove Proposition~\ref{prop:tilting} and then  Proposition~\ref{prop:decomp}.
  Finally, we give proofs for Theorem~\ref{thm:bigjump}, Corollary~\ref{cor:overshoot} and Corollary~\ref{cor:extendzero1}.
\item In Section~\ref{sec:LLD}, we prove the local large deviation results of Section~\ref{sec:locallargedev}.
  We start with the proof of Proposition~\ref{prop:tiltinglocal}, then we prove Proposition~\ref{prop:decomplocal} and finally Theorem~\ref{thm:bigjumplocal}.
\item Section~\ref{sec:conditional} is dedicated to the proof of Proposition~\ref{prop:othersnormalMnl}, which is carried out after having constructed the regular conditional distribution of $R(\xi_1,\ldots, \xi_n)$ given $M_n$.
\item Finally, the appendix is dedicated to some technical results and comments. Appendix~\ref{sec:discussR} contains several comments, examples and proofs of claims regarding Rozovkii's Theorem~\ref{thm:rozo6}.
  Appendix~\ref{app:stable} deals with the case of random variables in the domain of attraction of an $\alpha$-stable law with $\alpha\in(0,2)$, proving in particular Corollaries~\ref{cor:objnot=2} and~\ref{cor:objnot=2local}.
\end{itemize}

\section{The case of the domain of attraction of an $\alpha$-stable law with $\alpha\in (0,2)$}
\label{sect:alpha<2}

The one-big-jump phenomenon is well understood in the case where $\xi$ is in the domain of attraction of an $\alpha$-stable law with $\alpha\in (0,2)$:
in a nutshell, one has a ``big-jump'' whenever $\lim_{n\to\infty} \frac{x_n}{a_n} =\infty$, see~\cite{CH89,N79} for the case $\alpha\in (0,1)\cup(1,2)$ and~\cite{B19} for the case $\alpha=1$.
In order to put our results in perspective we recall here 
some results for large and local large deviations and give their consequences in terms of conditional laws.

\subsection{Large deviations and conditional laws}

Assume that $\alpha\in (0,2)$ and recall the definitions~\eqref{def:an}-\eqref{bn} of the normalising and centering sequences $(a_n)_{n\geq 1}$, $(b_n)_{n\geq 1}$. 
Then Theorem 2.1 in \cite{B19} collects large deviation results in that setting.

\begin{theorem}[Thm. 2.1 in \cite{B19}]\label{thm:alpha}
Assume that $\alpha\in (0,2)$, let $(a_n)_{n\geq 1}$, $(b_n)_{n\geq 1}$ be the sequences in~\eqref{attract}, and recall~\eqref{tails}.
If $\lim_{n\to\infty} \frac{x_n}{a_n} =\infty$, then we have, as $n\to\infty$,
\begin{align*}
    \P(S_n-b_n\geq x_n)& \sim npL(x_n)x_n^{-\alpha} \,, \\
    \P(S_n-b_n\leq -x_n)&\sim nqL(x_n)x_n^{-\alpha}\,.
\end{align*}
If $p=0$ or $q=0$, one interprets the corresponding term on the r.h.s as $o(nL(x)x^{-\alpha}).$
\end{theorem}

As a corollary, we obtain a result on the law of $\xi_1,\ldots, \xi_n$ conditionally on the large deviation event $\{S_n-b_n \geq x_n\}$.
\begin{corollary}[One-big-jump phenomenon]\label{cor:objnot=2}
Assume $\alpha\in (0,2)$, let $(a_n)_{n\geq 1}$, $(b_n)_{n\geq 1}$ be the sequences in~\eqref{attract} and suppose that $p>0$ in~\eqref{tails}. 
If $\lim_{n\to\infty} \frac{x_n}{a_n} =\infty$, then as $n\to\infty$
\begin{equation}\label{sntomn}
\P(S_n-b_n\geq x_n)\sim \P(S_n-b_n\geq x_n, M_n\geq x_n)\sim \P(M_n\geq x_n).
\end{equation}
Moreover the condition $\lim_{n\to\infty} \frac{x_n}{a_n} =\infty$ is necessary and sufficient for
\begin{align}
    \label{othersnormalSnlarge}
    \lim_{n\to\infty} d_{\mathrm{TV}}\Big( \mathscr{L}\big( R(\xi_1,\dots,\xi_n) \, \big| \, S_n-b_n \geq  x_n \big), \,
    \big(\mathscr{L}(\xi)\big)^{\otimes (n-1)} \Big) = 0.
  \end{align}
In particular, if $\lim_{n\to\infty} \frac{x_n}{a_n} =\infty$, we have
\begin{equation}
\LL\left(\frac{S_n-b_n-M_n}{a_n} \, \Big|\,S_n-b_n\geq x_n \right) \xrightarrow[n\to\infty]{w} \LL(\mathcal S_\alpha),
\end{equation}
where $\xrightarrow{w}$ denotes weak convergence 
and $\mathcal S_\alpha$ is the $\alpha$-stable random variable appearing in~\eqref{attract}.
\end{corollary}

\begin{remark}
When $p=0$, Theorem~\ref{thm:alpha} only gives that $\P(S_n -b_n \geq x_n) = o(nL(x_n)x_n^{-\alpha})$. We mention~\cite[Thms.~9.2 and~9.3]{DDS08} and also \cite{borovkov2003} for sufficient conditions to have a one-big-jump phenomenon in this case, but we are not aware of a general necessary and sufficient condition.
\end{remark}

\subsection{Local large deviations and conditional laws}

To simplify the statements, we assume that $\xi$ is integer valued. We are interested in estimating the local large deviation probabilities $\P(S_n - \lfloor b_n \rfloor =x_n)$ as $n\to\infty$, if $(x_n)_{n\geq 1}$ is a sequence of integers such that $\lim_{n\to\infty} \frac{x_n}{a_n} =\infty$.
For this type of results, we need an extra condition on the local tail behaviour of the distribution of $\xi$:
in addition to~\eqref{tails}, we assume that there is some $\alpha\in (0,2)$ and some slowly varying function $L(\cdot)$ such that
\begin{equation}
\label{=+x}
     \P(\xi=x)\sim p \alpha L(x)x^{-(1+\alpha)}, \qquad \text{ as } x\to\infty \,,
\end{equation}
with $p\geq 0$; note that this implies the first half of~\eqref{tails}.
If $p=0$, one interprets it as $o(nL(x)x^{-(\alpha+1)}).$
We now recall the result obtained in~\cite{B19} (see also~\cite{D97} for the case $\alpha\in (0,1)$).

\begin{theorem}[Thm.~2.4 in~\cite{B19}]\label{thm:alphalocal}
Assume that~\eqref{tails} holds for some $\alpha \in (0,2)$ and that additionally one has~\eqref{=+x}; let $(a_n)_{n\geq 1}$, $(b_n)_{n\geq 1}$ be the sequences in~\eqref{attract}.
Then, if $\lim_{n\to\infty} \frac{x_n}{a_n} =\infty$, we have 
\begin{equation}\label{corsm}
 \P(S_n-\lfloor b_n\rfloor =x_n)\sim np \alpha L(x_n)x_n^{-(1+\alpha)} \,.
\end{equation}
\end{theorem}

\begin{corollary}[Local one-big-jump phenomenon]\label{cor:objnot=2local}
Assume $\alpha\in (0,2)$, let $(a_n)_{n\geq 1}$, $(b_n)_{n\geq 1}$ be the sequences in~\eqref{attract} and assume that \eqref{=+x} holds for some $p>0$. 
Then the condition $\lim_{n\to\infty} \frac{x_n}{a_n} =\infty$ is necessary and sufficient for
\begin{align}\label{othersnormalSnlarge=x}
    \lim_{n\to\infty} d_{\mathrm{TV}}\Big( \mathscr{L}\big( R(\xi_1,\dots,\xi_n) \, \big| \, S_n-\lfloor b_n\rfloor=x_n \big), \,
    \big(\mathscr{L}(\xi)\big)^{\otimes (n-1)} \Big) = 0 .
\end{align}
In particular, under $\lim_{n\to\infty} \frac{x_n}{a_n} =\infty$, 
\begin{equation}
\LL\left(\frac{x_n-M_n- b_n}{a_n} \, \Big|\,S_n-\lfloor b_n\rfloor = x_n \right) \xrightarrow[n\to\infty]{w} \LL(\mathcal S_\alpha).
\end{equation}
\end{corollary}

\begin{remark}
Without assuming the local tail condition~\eqref{=+x}, one is still able to find an estimate on the local large deviation probability, see \cite[Thm.~1.1]{CD19} and \cite[Thm.~1.1]{MT22} (or \cite[Thm.~2.3]{B19} which includes the case $\alpha=1$), using only~\eqref{tails}. More precisely, if $\alpha\in (0,2)$ and $(a_n)_{n\geq 1}$, $(b_n)_{n\geq 1}$ are the sequences in~\eqref{attract}, then there is a constant $C_0$ such that uniformly for sequences $x_n\geq a_n$ we have
\[
 \P(S_n-\lfloor b_n\rfloor =x_n) \leq \frac{C_0}{a_n} \, n \P(|\xi| \geq x_n) \,.
\]
\end{remark}

\section{Further comments, applications and open questions}
\label{sec:maincomments}

The behaviour of summands in a sum of i.i.d.\ random variables is of course a ubiquitous and classical theme in probability theory: the
phenomenon of one big-jump being responsible for atypically large values of the sum features in many applications. Let us mention here
random walks and renewal processes (see, e.g.\ \cite{B19-, B19} and
the references there), population models in regimes with
asymptotically infinite offspring variance (where reproduction
proceeds in two steps, with each of $N$ individuals first generating a
random number~$\xi_i$ of juveniles, from which then $N$ are sampled to
form the next generation, as e.g.\ in \cite{S03, BLS18, BY22}), the
zero-range process on finite graphs subject to certain homogeneity conditions (since then the stationary law of
the occupation numbers is i.i.d.\ conditioned on the sum, one big-jump
corresponds to condensation, see Section~\ref{sec:zerorange} below), critical Galton-Watson trees with
heavy-tailed offspring distribution (see e.g.\ \cite{K15} and
\cite{KR19}, where conditioned on being large, a node with macroscopic
degree may emerge) and their applications in the study of random
planar maps (e.g.\ \cite{KR20}), ruin problems in insurance
mathematics (e.g.\ \cite{BBS15}).
Note also that for i.i.d.\ sums with exponential tails, if in the so-called borderline case the appropriately tilted laws are sufficiently heavy-tailed, even richer behaviour than one big jump may appear when conditioning on reaching an atypically large value, see \cite{foss2011limit} for instance.

\subsection{Conditional laws when $M_n$ is large}
\label{sec:condMax}

In Section~\ref{subsect:condlaws}, we considered the conditional joint
law of the $n-1$ smallest variables conditioned on a large deviation
of the sum. Let us briefly revisit this for the case when one instead
conditions on an atypically large value of the maximum.  Recall the
definition of the map $y\mapsto R(y)$ for
$y=(y_1,\ldots, y_n)\in \mathbb R^n$ from
Section~\ref{subsect:condlaws}.

In the following, we write $\mathscr{L}(X)$ to denote the law of a random variable (or a vector of random variables) $X$.
Let us mention that a regular conditional distribution of $R(\xi_1,\ldots,\xi_n)$ given~$M_n$ exists, that is
\begin{equation}\label{eq:regular00}
\LL\big( R(\xi_1,\dots,\xi_n) \, \big| \, M_n = x \big)
\end{equation}
is well defined for any $x$ in the support of $\xi$.
In fact, we give a concrete construction for the regular conditional distribution, see Section \ref{sec:conditional}.

The following result is very natural when considering a large deviation event that involves only the maximum $M_n$. We were unable to find a reference for it, so we will prove it in Section~\ref{sec:propMn}.

\begin{proposition}\label{prop:othersnormalMnl} 
Let  $(x_n)_{n\geq 1}$ be a sequence satisfying $\lim_{n\to\infty} n\overline F(x_n)=0$, with $\overline F(x_n)>0$ for all $n$. Then
\begin{align}
\label{othersnormalMnlarge}
    \lim_{n\to\infty} d_{\mathrm{TV}}\Big( \mathscr{L}\big( R(\xi_1,\dots,\xi_n) \, \big| \, M_n \geq x_n \big), \,
    \big(\mathscr{L}(\xi)\big)^{\otimes (n-1)} \Big) = 0,
\end{align}
and 
\begin{align}
    \label{othersnormalMnlarge.local}
    \lim_{n\to\infty} d_{\mathrm{TV}}\Big( \mathscr{L}\big( R(\xi_1,\dots,\xi_n) \, \big| \, M_n = x_n \big), \,
    \big(\mathscr{L}(\xi)\big)^{\otimes (n-1)} \Big) = 0 \,.
  \end{align}
\end{proposition}
Note that Proposition~\ref{prop:othersnormalMnl} requires no structural conditions on the distribution of the~$\xi$'s.

\subsection{Application to the zero-range process}
\label{sec:zerorange}

As an example of application, we discuss in detail how Corollary \ref{cor:extendzero} can be applied to the zero-range process.
We provide a short introduction to the zero-range process, using the terminology in \cite{AGL13} (up to minor differences) and rephrasing certain results in our notation. 
Our goal is to state (part of) Theorem~2.1 in \cite{AGL13} and explain how our Corollary~\ref{cor:extendzero} extends it.

\smallskip
Consider a finite set $\Lambda_L$ of $L$ sites. Each site can host any number of (indistinguishable) particles.
Informally, the zero-range process is a continuous-time Markov chain with the following dynamics:
a site $x\in \Lambda_L$ loses a particle at rate $g(\eta_x)$, where $\eta_x$ is the number of particles at site $x$ and $g:\N_0\mapsto [0,\infty)$ is a function such that $g(k)=0$ if and only if $k=0$.
That particle jumps to site $y\in \Lambda_L$ with probability $p(x,y)$, where $p(\cdot,\cdot)$ is a transition kernel on $\Lambda_L\times \Lambda_L$ (which is assumed to induce an irreducible Markov chain on $\Lambda_L$, with a spatially homogeneous invariant measure).
To proceed, we need some more notation as follows:
\begin{itemize}
\item A particle configuration is denoted $\eta=(\eta_x: x\in \Lambda_L)$;
\item for a configuration $\eta$ with $\eta_x>0$, define the configuration $\eta^{x,y} = (\eta_z^{x,y})_{z\in \Lambda_L}$ by $\eta_x^{x,y}=
\eta_x-1$, $\eta_y^{x,y}=\eta_y+1$ and $\eta_z^{x,y}=\eta_z$ for all $z\neq x,y$.
\end{itemize}
Then, the zero-range process is described by the following generator on the space of configurations:
\[
\mathcal Lf(\eta)=\sum_{x,y\in\Lambda_L}g(\eta_x)p(x,y)\left(f(\eta^{x,y})-f(\eta)\right) \,,
\]
which corresponds to the description given above.
Note that the total number of particles is preserved, and  we denote it by $N$.

Setting $w(n) := \prod_{k=1}^n\frac{1}{g(k)}$, it is known that there is a family of invariant measures, that are product measures indexed by a parameter $\varphi\geq 0$ (called \textit{fugacity}), that are defined by
\[
\nu_\varphi(\eta)=\frac{1}{z(\phi)^{\Lambda_L}} \prod_{x\in\Lambda_L} w(\eta_x)\phi^{\eta_x} \,,
\quad \text{ with }
z(\varphi)=\sum_{n=0}^{\infty} w(n)\phi^n \,,
\]
as soon as $\varphi$ is such that the normalising constant verifies $z(\varphi)<\infty$.
Then, one can easily observe that since the number of particles $S_L(\eta) := \sum_{x\in \Lambda_L} \eta_x$ is invariant under the microscopic dynamics, the invariant measure on the set of configurations $ \{\eta: S_{L}(\eta) =N\}$ is given by
\begin{equation}
\label{eq:mu=cond}
\mu_{N,L} (\cdot) = \nu_{\varphi}\big( \cdot \, \big| \, S_L(\eta)=N \big) \,.
\end{equation}
The measure $\mu_{N,L}$ is in fact independent of the choice of $\varphi$.
Then, a natural question is whether 
the convergence of the measures $\mu_{N,L}$ holds in the so-called thermodynamic limit, \textit{i.e.}\ taking $L, N\to \infty$ with $\frac{N}{L} \to \rho$ for some $\rho> 0$.

Since~\cite{Evans00}, a lot of interest has been put on the case where the jump rate $g(n)$ decays with the number of particles.
A natural assumption is that 
$g(k) = 1+ \frac{b}{k} + \frac{\gep_k}{k}$
for some $b> 1$ and some vanishing sequence $(\gep_k)_{k\geq 0}$.
Note that this implies that\footnote{Let us mention that there is a notational inaccuracy in~\cite{AGL13}: the authors in fact assume that $w(n) \sim c n^{-b}$ for some constant $c>0$, but for this, one would need the additional condition $\sum_{k\geq 1} \frac{\gep_k}{k}<\infty$.}
\begin{equation}
\label{def:w}
   w(n) =\prod_{k=1}^n\frac{1}{g(k)} \sim L(n) n^{-b} \quad \text{as } n\to\infty \,,
\end{equation}
for some slowly varying function $L(n)\sim c' \exp(\sum_{k=1}^n \frac{\gep_k}{k})$.
In fact, in a large part of the literature the choice $g(k)=1+\frac{b}{k}$ or $g(k)= (k/(k-1))^{b}$ is made, so that $w(n)$ is explicit with $w(n)\sim c n^{-b}$, see~\cite{Evans00,AL09,Xu20}.
Let us stress that when $\varphi=1$ we have
\[
\nu_1 (\eta_x =n) = \frac{w(n)}{\sum_{n\geq 1} w(n)} \sim c L(n)n^{-b} \,,
\]
which therefore corresponds to the assumption~\eqref{eq:localtail} with $b=\beta+1$ (note also that $\eta_x$ is a non-negative random variable).
We can then define the critical density
$\rho_c:= \E_{\nu_1}[\eta_x]\in (0,\infty]$.
Then, if $L, N\to \infty$ with $\frac{N}{L} \to \rho$ for some $\rho> 0$, we have that (see~\cite{JMP00}):
\begin{itemize}
\item  if $\rho<\rho_c$, the particle distribution $\mu_{N,L}$ converges to the limit stationary product measure~$\nu_\phi$ with~$\phi<1$ determined by  $\E_{\nu_\phi}[\eta_x]=\rho$, in the sense of finite-dimensional marginals;
\item if $\rho\geq \rho_c$, then we obtain the same result as above except with $\phi=1$.
\end{itemize}
Additionally, a condensation phase transition occurs: if $\rho>\rho_c$, then there exists a site containing a positive fraction of the particles, whereas if $\rho<\rho_c$ no site contains a non-zero fraction of particles, see~\cite{GSS03}.
The appearance of a condensation phenomenon can be understood when considering~\eqref{eq:mu=cond}: taking $\varphi=1$ so that $\nu_1(\eta_x=n) =\frac{w(n)}{\sum_{n\geq 1} w(n)}$, we obtain that the law $\mu_{N,L}$ corresponds to the law of i.i.d.\ random variables (with a heavy-tail distribution) conditioned by their sum being equal to $N \sim \rho L \gg \rho_c L = \E_{\nu_1}[\eta_1] L$.
In particular, a quantity of interest is $M_L(\eta)=\max\{\eta_x: x\in\Lambda_L\}$, which is the largest number of particles that a site contains; we refer to \cite{AL09} for some asymptotic results on $M_L$ in the case where $N/L\to \rho>\rho_c$.

In \cite{AGL13}, the authors focus on the critical case where $1\ll N-\rho_cL =o(L)$ (we have $N/L\to\rho_c$), with a parameter $b>3$ in~\eqref{def:w}, so that in particular $\nu_1[(\eta_x)^{2+\delta}]<\infty$ for some $\delta>0$.
Let us state a result that is contained in Theorem 2.1 of~\cite{AGL13}. 
\begin{theorem}[\cite{AGL13}]
\label{thm:AGL}
Assume that $b>3$ and $L(\cdot)\equiv c_1$ in~\eqref{def:w}, and denote $\rho_c = \E_{\nu_1}[\eta_x]$, $\sigma^2= \mathrm{Var}_{\nu_1}(\eta_x)$.
Assume that $N\geq \rho_cL$ and define $\gamma'_L$ via 
\begin{equation}\label{theirgamma}
N=\rho_cL+\sigma\sqrt{(b-3)L\log L}\left(1+\frac{b}{2(b-3)}\frac{\log\log L}{\log L}+\frac{\gamma'_L}{\log L}\right).
\end{equation}
Then, if $\gamma':= \lim_{L\to\infty} \gamma'_L \in [-\infty,\infty]$, we have
\begin{equation}
\label{conv:condensate}
    \LL\left(\frac{M_L}{S_L-\rho_cL}\, \Big|\, S_L=N\right)\stackrel{w}{\longrightarrow} (1-p) \delta_0 + p \delta_1 \,,
\end{equation}
with $p=p_{\gamma'}=(1+\frac{\sigma^{b-1}(b-3)^{b/2}}{c_1\sqrt{2\pi}}e^{-(b-3)\gamma'})^{-1}$.
\end{theorem}

This result corresponds exactly to our Corollary~\ref{cor:extendzero} in the case where~\eqref{eq:localtail} holds with $L(\cdot)$ constant, see Example~\ref{ex:second} (one can check that our definition of $\tilde \gamma_n$ corresponds to the definition~\eqref{theirgamma} of $\gamma'_L$).
In fact, our Corollary \ref{cor:extendzero} extends Theorem \ref{thm:AGL} to the case $b=3$ ($\beta=2$ in \eqref{eq:localtail}) and allows a slowly varying function in~\eqref{def:w}.
We refer to Example~\ref{example:local} for a definition of the threshold for the appearance of a condensate in that case; note that Rozovskii's condition~\eqref{cond} holds since $\eta_x$ is non-negative, see Example~\ref{ex:central} in the Appendix.

\begin{example}

Assume that $w(n)\sim c_1 n^{-3}$ in~\eqref{def:w}, or equivalently $\nu_1(\eta_x=n) \sim c n^{-3}$, which is a natural example with $b=3$. One then has $\nu_1(\eta_x>n) \sim \frac c2 n^{-2}$, $\mm(n)\sim c \log n$, $q(n)\sim (2\log n)^{-1}$ and $a_n \sim \sqrt{\frac c2 n\log n}$.
Assuming a slightly stronger condition on $w(\cdot)$, namely $w(n) = c n^{-3} (1+o(\frac{1}{\log\log n}))$ (which holds if $g(k)= 1+\frac{3}{k}$ or $g(k)= (k/(k-1))^3$), we easily get that one can take exactly $a_n=\sqrt{\frac c2 n\log n}$ and Rozovskii's condition~\eqref{cond} holds, using Proposition \ref{thm:bn8equiv}.
Now we can use Example~\ref{example:local} to obtain that, setting $x_L= N-\rho_cL$ and 
\[
\tilde \gamma_L := \frac{x_L^2}{2a_L^2} - \log \log a_L - \frac32 \log\log\log a_L
= \frac{x_L^2}{2a_L^2} - \log \log L +\log 2 - \frac32 \log\log\log L +o(1)\,,
\]
if in addition $\lim_{L\to\infty} \tilde \gamma_L =\tilde \gamma_{\infty} \in [-\infty,\infty]$, then \eqref{conv:condensate} holds with $p=(1+\frac{1}{\sqrt{\pi}}e^{-\tilde \gamma_{\infty}})^{-1}$. 
We can reframe this in the same way as in~\eqref{theirgamma}: setting $\gamma'_L$ to be such that
\[
N=\rho_cL+\sqrt{c L\log L \log\log L}\left(1+\frac{3}{2}\frac{\log\log \log L}{\log \log L}+\frac{\gamma'_L}{\log \log L}\right),
\]
then if $\lim_{L\to\infty} \gamma'_L=\gamma'$, we obtain that \eqref{conv:condensate} holds with $p=(1+\frac{1}{2\sqrt{\pi}}e^{-\gamma'})^{-1}$.\qed
\end{example}

Additionally, our Corollary~\ref{cor:objnot=2local} answers the question in the case where $b\in [2,3)$ (with $\E_{\nu_1}[\eta_x]<\infty$ in the case $b=2$). It shows that a condensate that contains all the excess mass appears as soon as $N-\rho_cL \gg a_L$, with $a_L$ the normalising sequence analogous to~\eqref{def:an}; this question was raised in~\cite[p.~3477]{AGL13} (actually, the authors noticed that it mostly relied on Theorem~\ref{thm:alphalocal} which was missing at the time).

To conclude, let us mention that the case $b\in [1,2]$ with $\E_{\nu_1}[\eta_x] =\infty$ is also considered in the recent paper~\cite{Xu20}: in that case, one needs to define a finite volume version of the critical density, $\rho_{c,L}$, and the condensation phenomenon is shown to occur in some regime $N\gg \rho_{c,L} L$. We refer to~\cite{Xu20} for details.

\subsection{Going further: scaling of the (recentered) maximum}
\label{sec:comments}

Similarly to what is done for the overshoot in Corollary~\ref{cor:overshoot}, we could try to obtain the correct scale of the maximum.
We have in mind the following results, analogous to those in~\cite[Thm.~2.1]{AGL13} (in the setting of Theorem~\ref{thm:AGL}), which identify the scaling limit of the condensate in the zero-range process (we present them here with a integral and a local conditioning and to a wider range of distributions).

\begin{conjecture}
\label{conj:max2global}
Let $\alpha=2$, let $(a_n)_{n\geq 1}$ be a normalising sequence as in~\eqref{def:an}, and assume that $\overline{F}$ is intermediate regularly varying (see~\eqref{def:intermediateregvar}).
Let $(x_n)_{n\geq 1}$ be a sequence such that $\lim_{n\to\infty} \frac{x_n}{a_n} = \infty$ and let $r_n$ be defined as in~\eqref{def:rrn}. 
\begin{itemize}
    \item If $\lim_{n\to\infty}\frac{n\overline F(x_n)}{\P(S_n-b_n\geq x_n)}<1$, we have 
    \begin{equation}
    \label{smallxglobal}
    \lim_{n\to\infty} \P(M_n\leq t_n\,|\,S_n-b_n\geq x_n, M_n\leq r_n) = e^{-t} \,,
    \end{equation}
    for any sequence $(t_n)_{n\geq 0}$ such that $\lim_{n\to\infty} n\overline F(t_n) = t \in [0,\infty]$.
    \item If $\lim_{n\to\infty}\frac{n\overline F(x_n)}{\P(S_n-b_n\geq x_n)}>0$, we have
    \begin{equation}
    \label{bigxglobal}
    \LL\left(\frac{S_n-b_n-M_n}{a_n} \, \Big|\,S_n-b_n\geq x_n, M_n>r_n \right) \xrightarrow[n\to\infty]{w} N(0,1).
    \end{equation}
\end{itemize}
Similar results hold if we replace $\{S_n-b_n\geq x_n\}$ in the conditioning by $\{S_n - \lfloor b_n\rfloor =x_n\}$, assuming $\P(\xi=x)$ is intermediate regularly varying and~\eqref{cond:taillocal}, and replacing the condition on $\frac{n\overline F(x_n)}{\P(S_n-b_n\geq x_n)}$ by a condition on $\frac{n\P(\xi = x_n)}{\P(S_n-\lfloor b_n\rfloor = x_n)}$.
\end{conjecture}

We stress that the conjecture is actually only about the statement in~\eqref{smallxglobal}, since~\eqref{bigxglobal} is a direct consequence of Theorem~\ref{thm:bigjump}; we have kept the statement inside the conjecture since it gives a complete picture of the phenomenon. 
We have stated~\eqref{smallxglobal} as a conjecture since it is not a simple consequence of  results proved in this paper, but the main idea of the proof would be to use the conditional probability formula and a similar approach as for Proposition~\ref{prop:tilting}.

The statements in Conjecture \ref{conj:max2global} can roughly be understood as follows:
\begin{itemize}
    \item Conditioning on having a large deviation but no big-jump ($M_n\leq r_n$), the maximum behaves exactly as the maximum of i.i.d.\ random variables with law $\P$ \textit{i.e.}\ without conditioning.
     Note that~\eqref{smallxglobal} is a unified way of considering convergence to extreme value distributions: if $\overline F(x)\sim e^{-x}$ (resp.\ $\overline F(x)\sim x^{-\beta}$), then $t_n=\log n - \log t +o(1)$ (resp.\ $t_n \sim t^{-1/\beta} n^{1/\beta}$) so one recovers the Gumbel (resp.\ Fr\'echet) 
     distribution by a simple change of variable $u=-\log t$ (resp.\ $u=t^{-1/\beta}$).
    \item Conditioning on having a big jump ($M_n>r_n$), the fluctuations of the maximum around its typical value $S_n-b_n (\geq x_n)$ is Gaussian, on a scale $a_n$.
\end{itemize}
In particular, Conjecture~\ref{conj:max2global} sheds light on the law of the maximum (the condensate) in a regime where $\lim_{n\to\infty}\frac{n\overline F(x_n)}{\P(S_n- b_n \geq  x_n)} = s\in (0,1)$: by using Proposition~\ref{prop:decomp}, we have the decomposition 
\begin{align*}
\P(M_n \le t_n \mid S_n-b_n \ge x_n ) &  = (1+o(1)) s \P(M_n \le t_n \mid S_n-b_n \ge x_n, M_n > r_n )\\
& \qquad + (1+o(1))(1-s) \P(M_n \le t_n \mid S_n-b_n \ge x_n, M_n \leq r_n ) \,,
\end{align*}
so that the conditional law of the maximum will be a non-trivial mixture of a scaled exponential and a (differently) scaled and translated Gaussian.

\subsection{Multiple big-jumps}

An interesting direction of research would be to generalise the results obtained in the present paper to a setting where the large deviation is obtained not by a single big jump but by multiple ones.
This occurs for instance if the random variables are subject to a cutoff (or a dampening) which may depend on the number of variables in the sum, as considered in~\cite{KM22}.
The idea is to consider large deviation probabilities of the type
\[
\P(S_n-b_n\geq x_n, M_n \leq c_n) \,,
\]
where $c_n$ is a given cutoff. Alternatively, one may consider a triangular array of variables $(\xi^{(n)}_i)_{1\leq i \leq n}$, with a heavy-tailed law which is truncated (or dampened) at a threshold $c_n$.
Then, if~$\xi$ has a heavy tail and $x_n$ grows sufficiently fast, one expects a ``fewest-big-jumps'' principle to replace the ``one-big-jump'' principle:
the large deviation should be realised mostly thanks to $k_n = \lceil x_n/c_n \rceil$ random variables being close to the cutoff $c_n$. This is what is investigated in~\cite{KM22}; let us also mention \cite[Thm.~1.1]{CD19} where a general upper bound for the local large deviation is given, with this underlying philosophy.

Several questions have been answered in~\cite{KM22}, making the fewest-big-jumps principle precise, but many problems remain open. 
For instance, can we obtain a result analogous to Propositions~\ref{prop:decomp} and~\ref{prop:decomplocal} if one adds a dampening to the law of $\xi$?
More precisely, if one assumes that $\P(\xi=x) \sim \beta x^{-(1+\beta)}$ for some $\beta \geq 2$, one could conjecture that, for some reasonable choices of the cutoff $(c_n)_{n\geq 1}$ (e.g.\ $c_n = \gamma n $ for some $\gamma \in (0,1)$), one has roughly
\[
\begin{split}
\P(S_n-b_n\geq x_n, M_n \leq c_n) & \approx \P(S_n-b_n\geq x_n, M_n \leq r_n) + \binom{n}{k_n} \P(\xi= c_n)^{k_n} \\
& \approx \overline\Phi\Big(\frac{x_n}{a_n}\Big) + \binom{n}{k_n} \P(\xi= c_n)^{k_n}\,,
\end{split}
\]
where $k_n = \lceil x_n/c_n \rceil$.
A natural question would then be to understand the transition between a collective (Gaussian) and a few-individuals (multiple big-jumps) behaviour, analogously to Examples~\ref{ex:first} and \ref{ex:second}. The same question can be considered in the local large deviation setting.

Let us conclude by mentioning that in the context of the zero-range processes, one could consider a model with a saturation threshold $c_L$, meaning that a site cannot host (much) more than $c_L$ particles.
This corresponds to taking the function $g$ in Section~\ref{sec:zerorange} as being $L$-dependent and dampening the effective weight $w(n)$ at the threshold $c_L$; for instance taking $g(n)=\infty$ for $n>c_L$, one gets that $w(n)=0$ for all $n\geq c_L$.
Proving the ``fewest big-jumps'' principle, \textit{i.e.}\ obtaining sharp asymptotics of the type $\P(S_L=N, M_L\leq c_L)\sim  \binom{L}{k_L} \P(\xi=c_L)^{k_L}$ with $k_L = \lceil L/c_L \rceil$, could then possibly be translated into a statement that the zero-range process possesses $k_L$ condensates of size close to $c_L$.

\section{Proofs of the large deviations and conditional laws results}
\label{sec:proofnormal}

Before we start the proof of our results, let us introduce some notation.
We use the same letter $c,c_0,c_1,c_2$ etc,\ to denote constants at various places, but they may refer to different values.
For positive functions $f,g$ we write $f(y)\asymp g(y)$ if there exist $c>0,C>0$ such that $cg(y)\leq f(y)\leq Cg(y)$ for $y\geq 1$.

Also, to simplify the statements, we will assume that $\mu=\E[\xi]=0$.
Recall the definition~\eqref{m2} of $\mm(x)=\E[\xi^2 \ind_{|\xi|\leq x}]$; as we assume in this section that $\xi$ is in the domain of attraction of the normal law, $\mm(x)$ is slowly varying at $\infty$.

\subsection{Some preliminary estimates}

Let us now collect some estimates that will be useful in the rest of the paper: bounds on $\P(|\xi|>x)$ and truncated moments that involve~$\mm(x)$;
a Fuk--Nagaev type inequality for $\P(S_n\geq x, M_n \leq y)$ in the case $\alpha=2$.

\subsubsection{Estimating the tail with the truncated second moment}

For $x> 0$, let us set
\[
\q(x) := \frac{x^2}{\mm(x)} \P(|\xi|>x) \,,
\]
and $q^*(x) := \sup_{y\geq x} \q(y)$ and $\tilde q(x) := \frac{1}{x} \int_0^x \q(t) \dd t$.

\begin{claim}
\label{claim:tail} 
If $x\mapsto\mm(x)$ is slowly varying, then we have $\lim_{x\to\infty} \q(x)=0$;
a direct consequence is that $\lim_{x\to\infty} q^*(x)= \lim_{x\to\infty} \tilde q(x)=0$.
Also, there is a constant $c>0$ such that
\begin{align}
&\E[|\xi|\ind_{\{|\xi| > x\}}] \leq c\, q^*(x)   \mm(x) x^{-1} \,, \label{inequality:xi}\\
&\E[|\xi|^3\ind_{\{|\xi| \leq x\}}] \leq c \,\tilde q(x) \mm(x) x \,.  
\label{inequality:xi^3}
\end{align}
\end{claim}

\begin{proof}
 For the first part of the statement, the fact that $\lim_{x\to\infty} \q(x)=0$ is recalled in~\eqref{f<sigma}, see e.g.\ \cite[Ch.~IX.8, Eq.~(8.5)]{F71}.

To prove \eqref{inequality:xi}, we write
\[
\E[|\xi|\ind_{\{|\xi| > x\}}] = x\P(|\xi|>x) + \int_x^{\infty} \P(|\xi|>t) \dd t \leq \frac{\overline{q}(x) \mm(x)}{x} + q^*(x) \int_x^{\infty} \mm(t) t^{-2} \dd t \,,
\]
where we have used that $\P(|\xi|>t)=\overline{q}(t) \mm(t)t^{-2}$ and the definition of $q^*$.
Then, using the properties of regularly varying functions, the integral is asymptotically equivalent to $\mm(x)x^{-1}$, which gives the desired bound.

To prove \eqref{inequality:xi^3}, write
\[
\begin{split}
\E[|\xi|^3\ind_{\{|\xi| \leq x\}}] & = 3\int_0^{\infty} t^2 \P( |\xi| \ind_{\{|\xi| \leq x\}} >t ) \dd t
\leq 3 \int_0^{x} t^2 \P( |\xi| >t ) \dd t
\leq 3 \mm(x) \int_0^{x} \overline{q}(t) \dd t \,,
\end{split}
\]
where we used the fact that $t\mapsto\mm(t)$ is non-decreasing. With the definition of $\tilde q(x)$, this gives the desired conclusion.
\end{proof}

\subsubsection{A Fuk--Nagaev inequality}
\begin{lemma}
\label{lem:FN2}
Assume that $\mu=0$ and that $x\mapsto\sigma^2(x)$ is slowly varying. Then there is some $r_0>0$ such that for any $x, y \geq r_0$
\[
\P(S_n \geq x, M_n\leq y) \leq e^{x/y} \Big(1+ \frac{xy}{n\mm(y)} \Big)^{-x/y} \,.
\]
\end{lemma}

\begin{proof}
We use \cite[Theorem~1.2]{N79} with $t=2$, to get the following:
\[
\P(S_n \geq x, M_n\leq y) \leq  e^{\frac xy} \Big( 1+ \frac{xy}{n\mm(y)} \Big)^{-\frac{x}{y} +\frac{n \mu(y)}{y} - \frac{n\mm(y)}{y^2}} \,,
\]
with $\mu(y) = \E[\xi \ind_{|\xi|\leq y}]$.
Since $\E[\xi]=0$, we have $\mu(y) = -\E[\xi \ind_{|\xi|> y}] \leq  \E[|\xi|\ind_{|\xi|>y}]$.
Therefore, thanks to Claim~\ref{claim:tail} above, we have that
$\E[|\xi|\ind_{|\xi|>y}] = o(\mm(y)/y)$: if $y$ is large enough we have $y\mu(y) \leq \mm(y)$, so $\frac{n\mu(y)}{y} - \frac{n\mm(y)}{y^2}\leq 0$, which concludes the proof.
\end{proof}

\subsection{Proof of Proposition~\ref{prop:tilting}}

We present here a proof in a slightly more general setup: we will replace the constraint $M_n\leq r_n$ by $M_n\leq ar_n$, for some fixed $a>0$. We stated Proposition~\ref{prop:tilting} only for $a=1$ but we also need the result in the proof of Proposition~\ref{prop:decomp} with $a=4$, see~\eqref{eq:applyProptilting} below.
Let us start with estimates in the case of a generic $r$, which will later on be replaced by $r=ar_n$.

\subsubsection{Estimates on the tilted measure}

Let us redefine $\mathbb P_u$ from~\eqref{def:Pu} with a general $r$ instead of $r=r_n$ (we keep the same notation for simplicity).
For any $u \geq 0$ and $r> 0$ \begin{equation}
\label{def:Pu2}
\frac{\d \P_{u}}{\d  \P} (x) = \frac{\d \P_{u}^{(r)}}{\d  \P} (x) = \frac{1}{M(u)} e^{u x} \ind_{(-\infty, r]}(x) \,,\quad \text{with }  M(u) = M_r(u) := \E[e^{u  \xi} \ind_{\{\xi\leq r\}}] \,.    
\end{equation}

\begin{claim}
\label{claim:Eu}
Let $0<c<C$ be fixed constants.  Assume $\mu=0$.
If $x\mapsto \sigma^2(x)$ is slowly varying at infinity, then as $r\to \infty$, uniformly for $\frac{c}{r} \leq  u\leq \frac{C}{r}$, we have
\begin{align}
\tag{i}
&  M(u)-1 = o(r^{-1} \mm(r)) \,;\\
\tag{ii}
& M'(u) = (1+o(1)) u \, \mm(r) \,; \\
\tag{iii}
& M''(u) = (1+o(1))\, \mm(r) \,.
 \end{align}
We also have that $M''(u) \geq (1+o(1)) \mm(r)$ for any $0\leq u \leq \frac{C}{r}$.
Finally, setting $\tilde q^*(x) := q^*(x)+\tilde q(x)$, we have
\begin{equation}
    \tag{iv}
 \E\big[ |\xi|^3 e^{u\xi} \ind_{(-\infty,r]}(\xi) \big] \leq  c_0 \tilde q^*(r)  r \mm(r) \,.
\end{equation}
\end{claim}

\begin{proof}
Define, for $k\geq 0$,
$m_k(r) = \E[\xi^{k} \ind_{[-r,r]}(\xi)]$; in particular we have $\sigma^2(r)=m_2(r).$

\textit{Item (i).}
We have
\begin{equation}
\label{eq:Muexpand}
M(u)= \E[ e^{u\xi} \ind_{\{|\xi|\leq r\}}] + \E[e^{u\xi} \ind_{\{\xi <  -r\}}] \,.
\end{equation}
We can estimate the first term as follows: using that $|u\xi| \leq C$ for any $|\xi| \leq r$, expanding the exponential, we get
\[
\E[ e^{u\xi} \ind_{\{|\xi|\leq r\}}]
= m_0(r) + u \, m_1(r) + O ( u^2 m_2(r) ) \,.
\]
Using that $u \leq C/r$,  we get that the last term is bounded by a constant times $r^{-2} \mm(r)= o(r^{-1}\mm(r))$.
Now,  thanks to Claim~\ref{claim:tail}, we obtain the following estimates (using that $\mu=0$ to get that $m_1(r) = -\E[\xi \ind_{\{|\xi|>r\}}]$):
\[
\begin{split}
    m_0(r) &= 1- \P(|\xi|>r) = 1- o(r^{-2}\mm(r)) \,, \\
    |m_1(r)| & \leq  \E[|\xi| \ind_{\{|\xi|>r\}}]= o(r^{-1} \mm(r)) \,. \\
\end{split}
\]
Therefore, we obtain $\E[ e^{u\xi} \ind_{\{|\xi|\leq r\}}]= o(r^{-1} \mm(r)).$

For the second term in \eqref{eq:Muexpand}, recalling Claim~\ref{claim:tail}, we have
\[
\begin{split}
0\leq \E[ e^{ u \xi } \ind_{\{\xi< -r\}}]
&\leq  \sum_{k=0}^{\infty} e^{ - c 2^k } \P( \xi\in [-2^{k+1}r, -2^{k}r) ) \\
&\leq  \sum_{k=0}^{\infty} e^{ - c 2^k } \P(|\xi|>2^{k}r) )=\sum_{k=0}^{\infty} e^{ - c 2^k } \frac{\bar q(2^kr)\sigma^2(2^kr)}{2^{2k}r^2} ) \\
&\leq r^{-2} q^*(r) \sum_{k=0}^{\infty} 2^{2k} e^{ - c2^k} \mm(2^kr) \leq C r^{-2} q^*(r) \mm(r) \,,
\end{split}
\]
where we used Potter's bound (see~\cite[\S1.5.4]{BGT89}) for the last inequality, to get that $\mm(2^kr) \leq 2^k \mm(r)$. All together, we obtain (i).

\textit{Item (ii).}
Let us now turn to $M'(u)$.
Again, we write $M'(u)$ as
 \[
M'(u) = \E[e^{u\xi} \xi  \ind_{\{|\xi| \leq r\}}] + \E[e^{u\xi} \xi  \ind_{\{\xi < -r\}}] \,.
 \]
For the first term, expanding the exponential, we get
\[
\E[e^{u\xi} \xi  \ind_{\{|\xi| \leq r\}}]
= m_1(r) + u m_2(r)  + O(u^2 m_3(r)) \,.
\]
As above, we get that $m_1(r) = o(r^{-1} \mm(r))$
and  thanks to Claim~\ref{claim:tail}-\eqref{inequality:xi^3} we also have $u^2 m_3(r) = O(r \tilde q(r) \mm(r) ) = o(r^{-1}\mm(r))$. 
For the second term, recalling Claim~\ref{claim:tail}, we get similarly as above that
\[
- c r^{-1} q^*(r) \mm(r)\leq \E[ e^{ u \xi } \xi \ind_{\{\xi< -r\}}]
\leq 0 \,,
\]
which concludes the proof for (ii).

\textit{Item (iii).} As above, we write
 \[
 M''(u) =\E[ e^{u\xi} \xi^2  \ind_{\{|\xi| \leq r\}}] +  \E[ e^{u\xi} \xi^2  \ind_{\{ \xi<-r \}}]  \,.
 \]
Expanding the exponential, we get that the first term is $\mm(r) + O( r^{-1} m_3(r))$, 
with $|m_3(r)| \leq  r \tilde q(r) \mm(r)$ thanks to Claim~\ref{claim:tail}.
Since the second term is non-negative, this gives the general lower bound $M''(u)\geq (1+o(1)) \mm(r)$.

When $u\geq c/r$,  the second term is treated as above: using Claim~\ref{claim:tail} and Potter's bound, we get
\[
\begin{split}
0\leq \E[ e^{ u \xi } \xi^2 \ind_{\{\xi< -r\}}] & \leq  \sum_{k=1}^{\infty} 2^{2k+2} r^2 e^{ - c2^k } \P( \xi\in [-2^{k+1}r, -2^{k}r) ) \leq C q^*(r) \sum_{k=1}^{\infty} 2^{5k} e^{ -c 2^k} \mm(r) \,.
\end{split}
\]
The last term is bounded by a constant times $q^*(r) \mm(r) = o(\mm(r))$. Then we can conclude the proof for (iii).

 \textit{Item (iv).} 
 We have
 \[
 \E\big[ |\xi|^3 e^{u\xi} \ind_{(-\infty,r]}(\xi) \big] \leq \E[e^{u\xi} |\xi|^3\ind_{\{\xi\leq r\}} ] + \E[e^{u\xi} |\xi|^3 \ind_{\{\xi \leq -r \}}] \,.
 \]
 The first term is bounded by a constant times $\E[|\xi|^3 \ind_{\{ |\xi|\leq r \}}] \leq c r \tilde q(r) \mm(r)$ thanks to Claim~\ref{claim:tail}.
 The second term is treated as above and is bounded by a constant times $r q^*(r) \mm(r)$.
\end{proof}

We now present a corollary, in the case where $r = a r_n$ for some fixed constant $a>0$. Let us set a few notations first.
From now on, we set,  $\P_u := \P_u^{(ar_n)}$ and $M(u):=M_{ar_n}(u)$ for any $u \geq 0$, with the notation from~\eqref{def:Pu2}.
Then, recall the definition of $\lambda(t) := m^{-1}(t)$,
where $m(u)=\E_{u}[\xi]  =\frac{M'(u)}{M(u)}$ (recall that $u\mapsto m(u)$ is increasing).
Hence, $\lambda(t)$ is such that $\E_{\lambda(t)} [\xi] =t$. Note that for $\lambda(t)$ to be well-defined and non-negative, we must take $t \geq m(0):=\E[\xi \mid \xi\leq a r_n]$.

\begin{corollary}
\label{cor:lambdasymp}
Assume $\mu=0$.
If $\lim_{n\to\infty}\frac{x_n}{n}=0$, then $m(0)=o(\frac{x_n}{n})$.
Let $0<c\leq d<\infty$ be fixed, then uniformly for $s\in [c,d]$ we have $\lambda(s\,\frac{x_n}{n})=(1+o(1)) \frac{s}{r_n}$.
As a consequence, we have $\lambda(t) \in [(1+o(1))\frac{c}{r_n},(1+o(1))\frac{d}{r_n}]$ uniformly for $t\in [ c\, \frac{x_n}{n}, d\,\frac{x_n}{n}]$.
\end{corollary}

\begin{proof}
Note that since $\E[\xi]=0$, we have
\begin{equation}
\label{eq:mu0}
m(0)=\E[ \xi \mid \xi\leq ar_n] = - \frac{1}{\P(\xi\leq a r_n)}\E[\xi\ind_{\xi> ar_n}]  \leq 0 \,.
\end{equation}
Now, recalling the definition~\eqref{def:rn} for $r_n$, we get that $\lim_{n\to\infty}r_n =\infty$ since $\lim_{n\to\infty}\frac{x_n}{n}=0$.
Thanks to Claim~\ref{claim:tail}, we therefore get that $|m(0)|   = o(r_n^{-1}\mm(r_n))$. In particular, in view of~\eqref{def:rn}, we get that $m(0)=o(\frac{x_n}{n})$.

From Claim~\ref{claim:Eu}, if $u= \frac{s}{r_n}$ then $m(u) = \frac{M'(u)}{M(u)} = (1+o(1)) \frac{s}{r_n} \mm(r_n) =(1+o(1)) s \frac{x_n}{n}$ (using that $\mm$ is slowly varying), with the $o(1)$ uniform for all $s \in [c,d]$.
Hence, if $t = s \frac{x_n}{n}$ we get that  $\lambda(t) =(1+o(1)) \frac{s}{r_n}$ with the $o(1)$ uniform for $s\in [c,d]$.
\end{proof}

\subsubsection{Estimate of the relative entropy: proof of Lemma~\ref{lem:entropy}}

Before we prove Proposition~\ref{prop:tilting}, as a warm up computation, we analyse the entropy~$H(\cdot)$ defined in \eqref{entropy}, that is we prove Lemma~\ref{lem:entropy} above\footnote{The corresponding result is proved in p.~107 of~\cite{D89}, where it is claimed that $nH(\frac{x_n}{n}) = \frac{x_n^2}{2 a_n^2} +o(1)$; we actually realised that an argument is missing in~\cite{D89} (one needs to control the error term in the variance of the tilted law).  
In our Lemma~\ref{lem:entropy}, we have $o(1)$ as a factor: to obtain $nH(\frac{x_n}{n}) = \frac{x_n^2}{2 a_n^2} +o(1)$, at least for $a_n\leq x_n\leq Ca_n \sqrt{|\log q(a_n)|}$, one needs to assume that $\overline{F}$ is extended regularly varying~\eqref{def:extendedregvar}, and Rozovskii's condition~\eqref{cond}, see Remark~\ref{rem:rozov}.}. We present the proof with $ar_n$ instead of $r_n$ in the definition~\eqref{entropy} of~$H$, that is we take $H(t) = H_{ar_n}(t)$.

\begin{proof}[Proof of Lemma~\ref{lem:entropy}]
First of all, notice that 
\[
H'(t) = -\lambda'(t) \frac{M'(\lambda(t))}{M(\lambda(t))} + \lambda(t) + t \lambda'(t) =\lambda(t) \,,
\]
since $\frac{M'(\lambda(t))}{M(\lambda(t))} = m(\lambda(t)) =t$, by definition of $\lambda(t)$, see \eqref{eqn:lambdat}.
Recalling that $m_0=m(0)$, we get that $H(m_0) = -\log M(0)=-\log \P(\xi\leq ar_n)$ so \begin{equation}\label{eqn:hm0} 
H(m_0) \sim \overline F(ar_n)\sim q(ar_n)\frac{\sigma^2(r_n)}{a^2r_n^2}=o\left(\frac{x_n^2}{n^2\sigma^2(r_n)}\right).
\end{equation}
Here we have used the fact that $\lim_{x\to\infty}q(x)=0$ and the definition \eqref{def:rn} of $r_n$. Moreover, we have $H'(m_0)=\lambda(m_0)=0$.
By Taylor's theorem, we have
\[
H\Big(\frac{x_n}{n}\Big)-H(m_0) = \int_{m_0}^{x_n/n} \Big( \frac{x_n}{n} -t\Big) H''(t) \d t.
\]
Using that $H'(t) =\lambda(t)$, we get that 
\[
H''(t) = \frac{1}{m'(\lambda(t))} =  \frac{M(\lambda(t))^2}{ M''(\lambda(t)) M(\lambda(t))-M'(\lambda(t))^2} = \frac{1}{ \E_{\lambda(t)}[\xi^2] - t^2}\,.
\]
Now, we have that $0\leq \lambda(t)\leq (1+o(1)) \frac{1}{r_n}$ for all $t\in [m_0, \frac{x_n}{n}]$, thanks to Corollary~\ref{cor:lambdasymp}, which gives
\[
\E_{\lambda(t)}[\xi^2] = \frac{M''(\lambda(t))}{M(\lambda(t))} \geq  (1+o(1)) \mm(r_n)\,,
\]
thanks to Claim~\ref{claim:Eu} (see the lower bound below (iii)). Note that $\sigma^2(r_n)/(\frac{x_n}{n})^2 \to\infty,$ due to \eqref{def:rn}. Therefore, using \eqref{eqn:hm0} and  that $m_0 =o(\frac{x_n}{n})$  from Corollary~\ref{cor:lambdasymp}, we have
\[
H\Big(\frac{x_n}{n}\Big) \leq (1+o(1)) \frac{1}{2 \mm(r_n)} \Big( \frac{x_n}{n} -m_0\Big)^2 +H(m_0) = (1+o(1)) \frac{x_n^2}{2 n^2\mm(r_n)}\,.
\]

To get a lower bound, we use the fact that for $t \in [\gep \frac{x_n}{n}, \frac{x_n}{n}]$ we have $ \frac{c}{r_n} \leq \lambda(t) \leq \frac{C}{r_n}$ for some positive constants $c,C$, see Corollary~\ref{cor:lambdasymp}.
Hence, thanks to Claim~\ref{claim:Eu}, we get that $H''(t) =(1+o(1)) \mm(r_n)^{-1}$ uniformly for $t \in [\gep \frac{x_n}{n}, \frac{x_n}{n}]$, using also that $t^2=o(\mm(r_n))$. 
We end up with
\[
H\Big(\frac{x_n}{n}\Big) \geq (1+o(1)) \frac{1}{\mm(r)} \int_{\gep x_{n}/n}^{x/n} \Big( \frac{x_{n}}{n} -t\Big)  \d t +H(m_0) \geq  (1-2\gep) \frac{x_{n}^2}{2 n^2\mm(r_{n})} \,.
\]
This concludes the proof.
\end{proof}

\subsubsection{Completion of the proof of Proposition~\ref{prop:tilting}}

Recall that we assume $\mu=0$ in this section. We fix $\lambda =\lambda_n := \lambda(\frac{x_n}{n})$ so that $\E_{\lambda}[S_n] = n m(\lambda) = x_n$.
Recalling the definition~\eqref{def:Pu2} of $\P_{\lambda}$, that also denotes with an abuse of notation the law of $n$ i.i.d.\ random variables with law $\P_{\lambda}$, we can write
\begin{align*}
\E_{\lambda}\Big[ e^{-\lambda S_n} \ind_{\{S_n \geq x_n\}}\Big] & =  \E\Big[ e^{-\lambda S_n} \ind_{\{S_n \geq x_n\}} \times \prod_{i=1}^n \frac{e^{-\lambda X_i}}{M(\lambda)} \ind_{\{X_i \leq a r_n\}}\Big]\\
&= \frac{1}{M(\lambda)^n}\P( S_n \geq x_n, M_n \leq a r_n) \,.
\end{align*}
We therefore get that
\begin{equation}
\label{changemeasure}
\begin{split}
\P( S_n \geq x_n, M_n \leq a r_n) & = M(\lambda)^n \E_{\lambda}\Big[ e^{-\lambda S_n} \ind_{\{S_n \geq x_n\}}\Big] \\
& = e^{ n [\log M(\lambda) -\lambda m(\lambda)]}  \E_{\lambda}\Big[ e^{-\lambda (S_n- x_n)} \ind_{\{S_n-x_n \geq 0\}}\Big] \\
& =e^{- n H(\frac{x_n}{n})} \E_{\lambda}\Big[ e^{-\lambda (S_n- x_n)} \ind_{\{S_n-x_n \geq 0\}}\Big]. 
\end{split}
\end{equation}

In this step of the proof of Proposition~\ref{prop:tilting}, we simply need to control the expectation on the r.h.s in the third equality.
The behaviour of $n H(\frac{x_n}{n})$ has been studied in Lemma~\ref{lem:entropy}. We can write the last term as $\E_{\lambda}[e^{-\lambda \tilde S_n} \ind_{\tilde S_n \geq 0}]$, where $\tilde S_n = \sum_{i=1}^n \tilde \xi_i$ with $\tilde \xi_i = \xi_i-\frac{x_n}{n}$.
In particular, using Claim~\ref{claim:Eu}, we have 
\begin{equation}
\label{momentsE}
\E_\lambda[\tilde \xi_i] =0 \,,\qquad
\E_{\lambda} [\tilde \xi_i^2] \sim \mm(r_n) \,, \qquad
\E_{\lambda} [|\tilde \xi_i|^3] \leq c_{a}  \tilde q^*(r_n) r_n\mm(r_n) \,.
\end{equation}
Indeed, using Corollary \ref{cor:lambdasymp} we have $\lambda = \lambda(\frac{x_n}{n}) \sim \frac{1}{r_n}$, which yields  $M(\lambda) \to 1$, see~(i) in Claim~\ref{claim:Eu}.
Then applying (ii) and (iii) in Claim~\ref{claim:Eu}, we have   $\mathrm{Var}_{\lambda}[\tilde\xi_i]=\mathrm{Var}_{\lambda}[\xi_i] = \frac{M''(\lambda)}{M(\lambda)^2} \sim \mm(r_n)$
and $\E_{\lambda}[|\tilde\xi_i|^3]\leq 4\E_{\lambda}[|\xi_i|^3]+4\left(\frac{x_n}{n}\right)^3 \leq c \tilde q^*(r_n) r_n \mm(r_n)$. 
The fact that $\E_{\lambda}[\xi] = \frac{x_n}{n}$ is by the definition of $\lambda.$ This gives~\eqref{momentsE}.

Thanks to the Berry--Esseen bound, this gives that uniformly for $y \in \mathbb R$,
\begin{equation}
\label{eq:BerryEsseen}
\Big| \P_{\lambda}\Big(  \frac{1}{\sqrt{n \mm(r_n)}} \widetilde S_n\leq y \Big) - \Phi(y) \Big| \leq \frac{C}{\sqrt{n\mm(r_n)}} \frac{\E_{\lambda} [|\tilde \xi_i|^3]}{\mm(r_n)} 
\leq c \tilde q^*(r_n) \frac{r_n}{\sqrt{n \mm(r_n)}} \,,
\end{equation}
which goes to $0$.
Indeed, note that since $r_n\sim \frac{n}{x_n} \mm(r_n)$ and $\sqrt{n} \sim a_n/\sqrt{\mm(a_n)}$,
we have 
\begin{equation}
\label{asymprn}
\frac{r_n}{\sqrt{n \mm(r_n)}} \sim  \frac{a_n}{x_n} \sqrt{\frac{\mm(r_n)}{\mm(a_n)}}\,.
\end{equation}
In particular, we get that $\frac{r_n}{\sqrt{n \mm(r_n)}} \leq  c\frac{a_n}{x_n}$:
this is bounded by a constant (since we are considering $x_n\geq a_n$) and goes to~$0$ if $\lim_{n\to\infty} \frac{x_n}{a_n} =\infty$. Since $\tilde q^*(r_n) \to 0$,
\eqref{eq:BerryEsseen} therefore shows that $\frac{1}{\sqrt{n\mm(r_n)}} \tilde S_n$ converges under $\P_{\lambda}$ to a standard Gaussian.

Additionally, we can use the Berry--Esseen bound~\eqref{eq:BerryEsseen} to get that 
\begin{equation}
\label{BerryEsseen2}
\Big| \E_{\lambda}[e^{-\lambda  \tilde S_n} \ind_{\tilde S_n \geq 0}] - \E[e^{- \lambda \sqrt{n\mm(r_n)} Z} \ind_{\{Z\geq 0\}}] \Big| = O\Big( \tilde q^*(r_n) \frac{r_n}{\sqrt{n \mm(r_n)}} \Big) = o(1) \,.
\end{equation}
Indeed, we can write
\[
\begin{split}
\E_{\lambda}[e^{-\lambda  \tilde S_n} \ind_{\tilde S_n \geq 0}] &= \int_0^1 \P_{\lambda}(t\leq e^{-\lambda \tilde S_n} \leq 1) \dd t  = \int_0^1 \P_{\lambda}\Big( 0\leq  \frac{1}{\sqrt{n \mm(r_n)}} \widetilde S_n\leq \frac{\log(1/t)}{\lambda \sqrt{n \mm(r_n)}}   \Big) \dd t \\
& = \int_0^1 \P_{\lambda}\Big( 0\leq Z\leq \frac{\log(1/t)}{\lambda \sqrt{n \mm(r_n)}}  \Big)\dd t + O\Big( \tilde q^*(r_n) \frac{r_n}{\sqrt{n \mm(r_n)}} \Big) \,,
\end{split}
\]
where we have used~\eqref{eq:BerryEsseen} for the last equality.
Now, with the same manipulation, the last integral is equal to $\E[e^{- \lambda \sqrt{n\mm(r_n)} Z} \ind_{\{Z\geq 0\}}] $.

\smallskip
Now, if $\lim_{n\to\infty}\frac{x_n}{a_n}=\infty$, then $\frac{r_n}{\sqrt{n \mm(r_n)}}$ goes to $0$, recall~\eqref{asymprn}, so $\lambda \sqrt{n\mm(r_n)} \sim \frac{\sqrt{n \mm(r_n)}}{r_n}$ goes to $\infty$: using that $\E[e^{-u Z} \ind_{\{Z\geq 0\}}] = e^{u^2/2} \overline{\Phi}(u)$, we obtain
\[
\E[e^{- \lambda \sqrt{n\mm(r_n)} Z} \ind_{\{Z\geq 0\}}] \sim \frac{1}{\sqrt{2\pi}} \frac{r_n}{\sqrt{n\mm(r_n)}}
\sim \frac{1}{\sqrt{2\pi}}  \frac{a_n}{x_n} \sqrt{\frac{\mm(r_n)}{\mm(a_n)}} \,,
\]
so that thanks to~\eqref{BerryEsseen2} we end up with
\begin{equation}
\label{eq:asymptildeSn}
   \E_{\lambda}[e^{-\lambda  \tilde S_n} \ind_{\tilde S_n \geq 0}] \sim \frac{1}{\sqrt{2\pi}}  \frac{a_n}{x_n} \sqrt{\frac{\mm(r_n)}{\mm(a_n)}} \,. 
\end{equation}

If on the other hand $x_n=O(a_n)$, then recalling~\eqref{asymprn} and the fact that $\mm(r_n)\sim \mm(a_n)$ if $a_n\leq x_n\leq C a_n$, we get that  $\tilde \lambda_n:=\lambda \sqrt{n\mm(r_n)} =(1+o(1)) \sqrt{\frac{x_n}{r_n}}=(1+o(1)) \frac{x_n}{a_n}$, where in the last equality we used $x_nr_n\sim n\sigma^2(r_n)\sim n\sigma^2(a_n)\sim a_n^2$.
Furthermore, we have $\E[e^{- \tilde \lambda_n Z} \ind_{\{Z\geq 0\}}] = e^{\tilde \lambda_n^2/2} \overline{\Phi}( \tilde \lambda_n)$.
We therefore end up with 
\[
\P(S_n\geq x_n, M_n\leq ar_n) = (1+o(1))e^{-nH(\frac{x_n}{n}) } e^{\tilde \lambda_n^2/2} \overline{\Phi}(\tilde \lambda_n) = (1+o(1)) \overline{\Phi}\left(\frac{x_n}{a_n}\right) \,,
\]
where for the last identity we used Lemma~\ref{lem:entropy} to get $nH(\frac{x_n}{n}) = \frac{x_n^2}{2a_n^2} +o(1) = \frac12 \tilde \lambda_n^2 +o(1)$ (since $x_n=O(a_n)$).\qed

\subsection{Proof of Proposition~\ref{prop:decomp}}

\subsubsection{Preliminary: a tail estimate using intermediate regular variation of $\overline{F}$}

\begin{lemma}\label{lem:1-epsilon}
Let $\alpha=2$, let $(a_n)_{n\geq 1}$ be a normalising sequence as in~\eqref{def:an} and assume that $\mu=0$ and that $\overline{F}$ is intermediate regularly varying (recall~\eqref{def:intermediateregvar}).  Let $(C_n)_{n\geq 1}$  be such that $\lim_{n\to\infty} n\overline F(C_n)=0$. Then, we have $ \lim_{n\to\infty} C_n=\infty$ and 
\begin{equation}
\label{epsilonequiv}
\lim_{\epsilon\to0}\lim_{n\to\infty}\sup_{x\geq C_n}\left| \frac{\P(S_n\geq x, M_n\geq (1-\epsilon)x)}{\P(S_n\geq x, M_n\geq  x)}-1\right|=0.
\end{equation}
In particular, if $\lim_{n\to\infty} \frac{x_n}{a_n} =\infty$ holds, then we can use $(x_n)_{n\geq 1}$ in place of $(C_n)_{n\geq 1}$ and  
\begin{equation}\label{sntomn2}
\P(S_n \geq x_n, M_n \geq x_n) \sim \P(M_n \geq x_n) \sim n \overline F(x_n) \,.
\end{equation}
\end{lemma}

\begin{proof} The intermediate regular variation of $\overline{F}$ and the definition of $(C_n)_{n\geq 1}$ imply that $\lim_{n\to\infty}C_n =\infty$ and $\overline F(C_n)>0$ for all $n$.  
Next, we note that 
\begin{equation}\label{eq:mn=fn}
\P(M_n \geq x) \sim n \overline F(x) \,,
\quad \text{as } n\to\infty
\end{equation}
uniformly for $x\geq C_n$. 
Note also that given $n\overline F(C_n)=o(1)$ and $x\geq C_n,$ we have
\begin{align*}
\big|\P(S_n\geq x, M_n\geq (1-\epsilon)x)-\P(S_n\geq x, M_n\geq x)\big|&\leq \P((1-\epsilon)x\leq M_n< x) \,,
\end{align*}
with 
\begin{align*}
\P((1-\epsilon)x\leq M_n< x) &= \P(M_n \geq (1-\epsilon) x) - \P(M_n \geq  x) &\\
&= (1+o(1)) n \left(\overline{F}((1-\epsilon)x) - \overline F(x)\right)+o(1)n\overline F((1-\epsilon)x)&
\end{align*}
where the second equality is due to \eqref{eq:mn=fn}. 
Moreover,
\begin{align*}
\P(S_n\geq x, M_n\geq  x)&\geq n\P(\xi \geq  x)\P(S_{n-1}\geq 0)- \binom{n}{2}\P(\xi_1\geq x,\xi_2\geq  x) \\
&\geq  (1+o(1))n\overline F(  x) \P(S_{n-1}\geq 0)-(1+o(1))(n\overline F( x))^2 \geq  \frac{n}{4} \overline F(x) \,,
\end{align*}
where the last two inequalities hold for $n$ large enough, uniformly for $x\geq C_n$.
Indeed, we have that $\lim_{n\to\infty} \P(S_{n-1}\geq 0 ) = \frac12$ by the Central Limit Theorem and $(n\overline F(x))^2 = o(n\overline F(x))$ for $x\geq C_n$.
Combining the above three displays, 
\begin{equation}\label{eq:tailestimatepart1}
\left| \frac{\P(S_n\geq x, M_n\geq (1-\epsilon)x)}{\P(S_n\geq x, M_n\geq  x)}-1\right| \leq 8 \Big( \sup_{x\geq C_n} \frac{\overline F((1-\gep)x)}{ \overline F(x)}  -1 \Big) + o(1)\sup_{x\geq C_n} \frac{\overline F((1-\gep)x)}{ \overline F(x)} \,,
\end{equation}
for $n$ large enough and $x\geq C_n.$ Using Potter's bound for intermediate regularly varying functions~\cite[Thm.~2.3]{Cline94}, see the bound~\eqref{cond:F2} in Claim~\ref{claim:equivalentcondF2}, we get that 
\[
1\leq \sup_{x\geq C_n} \frac{\overline F((1-\gep)x)}{ \overline F(x)}\leq (1+\sup_{x\geq C_n}\delta_{(1-\epsilon)x})\kappa\left(\frac{1}{1-\epsilon}\right)\,.
\]
Note that $\sup_{x\geq C_n}\delta_{(1-\epsilon)x}$ converges to $0$ as $n\to\infty$, and $\kappa(\frac{1}{1-\epsilon})$ converges to $1$ as $\epsilon\to0.$ Therefore, \eqref{eq:tailestimatepart1} entails \eqref{epsilonequiv}.

\smallskip
For the second part of the lemma, the fact that $(x_n)_{n\geq 1}$ is a candidate for $(C_n)_{n\geq 1}$ is due to \eqref{f<sigma} and \eqref{def:an}. Then, applying \eqref{eq:mn=fn}, we get
\[
\P(S_n\geq x_n , M_n \geq x_n ) \leq \P(M_n \geq x_n) \sim n \overline F(x_n) \,.
\]
Hence it only remains to prove a matching lower bound. 
For any $\gep >0$, we have that 
\begin{multline*}
\P(S_n\geq x_n , M_n \geq x_n )  \geq \P(S_n\geq x_n , M_n \geq (1+\gep) x_n ) \\
 \geq n \P\big( \xi>(1+\gep)x_n \big) \P\big( S_{n-1} \geq -\gep x_n\big) - \binom{n}{2} \P\big( \xi_1\geq (1+\gep) x_n, \xi_2 \geq  (1+\gep)x_n \big)\,.
\end{multline*}
Since $\lim_{n\to\infty} \P(S_{n-1} \geq -\gep x_n)=1$ for any $\gep >0$, because $\lim_{n\to\infty} \frac{x_n}{a_n} =\infty$, we get that the lower bound is asymptotically equivalent to $n \overline F( (1+\gep)x_n)$.
Then, thanks to the intermediate regular variation of $\overline{F}$, we get 
$\lim_{\gep\downarrow 0} \liminf_{x\to\infty}  \frac{ \overline F( (1+\gep)x)}{\overline F(x)} =1$,
which concludes the proof.
\end{proof}

\subsubsection{Completion of the proof of Proposition~\ref{prop:decomp}}

Recall that we assume that $\overline{F}$ is intermediate regularly varying and also that $\mu=0$. We will focus on the case where $\lim_{n\to\infty} \frac{x_n}{a_n} =\infty$ and $\lim_{n\to\infty}\frac{x_n}{n}=0$; recall also that \eqref{def:rn} holds. 
(We briefly discuss at the end of the proof for the case where $\liminf_{n\to\infty} \frac{x_n}{n}>0$, which can also be found e.g.\ in~\cite{DDS08}.) 
Using Lemma \ref{lem:1-epsilon}, we simply need to show that \begin{equation}\label{snmnbigasymp}\P(S_n \geq x_n, M_n >r_n)\sim \P(S_n \geq x_n, M_n> (1-\epsilon)x_n),\end{equation}
for any $\epsilon\in(0,1)$.

We  fix $\gep\in(0,1)$ and write the probability as $Q_1+Q_2+Q_3+Q_4$, with
\[
\begin{split}
   & Q_1  =\P(S_n \geq x_n,  M_n >  (1-\gep) x_n ) \,,\quad 
     Q_2 = \P(S_n \geq x_n, \gep x_n < M_n \leq  (1-\gep)x_n )\,,\\
    & Q_3  = \P(S_n\geq x_n, 4 r_n< M_n\leq \gep x_n) \,,\quad
    Q_4 = \P(S_n \geq x_n, r_n < M_n \leq 4r_n ) \,.
\end{split}
\]

By Lemma \ref{lem:1-epsilon}, $Q_1\sim n\overline F(x_n)$. We show next that all other terms are negligible compared to $n\overline F(x_n).$ 
The term $Q_2$ can be estimated as follows: we have that
\[
Q_2  \leq n \P\big(  \gep x_n < \xi \leq  (1-\gep) x_n \big) \P\big( S_{n-1} \geq \gep x_n \big) \,.
\]
Using Potter's bound~\eqref{cond:F2} (see~\cite[Thm.~2.3]{Cline94}), we get that $\P(\gep x_n < \xi \leq  (1-\gep) x_n) \leq \overline F(\gep x_n) \leq  c' \kappa(1/\gep) \overline F(x_n)$.
Also, since $\lim_{n\to\infty} \frac{x_n}{a_n} =\infty$ holds, by the convergence in distribution of~\eqref{attract} (recall $b_n=0$ since $\mu=0$) we get that  $\lim_{n\to\infty} \P( S_{n-1} \geq \gep x_n ) =0$ for any fixed $\gep>0$.
All together, this shows that $\lim_{n\to\infty} \frac{Q_2}{n \overline F(x_n)} = 0$.

Note that we have 
\begin{equation}\label{eq:xn>an>rn}
\frac{r_n}{a_n} \frac{\mm(a_n)}{\mm(r_n)}\sim \frac{a_n}{x_n}, \quad \text{ and } \quad \lim_{n\to\infty}  \frac{r_n}{a_n} = 0 \quad \text{ if} \lim_{n\to\infty} \frac{x_n}{a_n}=\infty.
\end{equation}
The first result is due to \eqref{def:rn} and the definition of $a_n$, see \eqref{def:an}. The second result relies on the first one and uses that $\sigma^2(\cdot)$ is slowly varying.
Then, since $\lim_{n\to\infty} \frac{x_n}{a_n} =\infty$ we have $\lim_{n\to\infty} \frac{x_n}{r_n} =\infty$. We can now bound $Q_3$ as follows (omitting  the floor function $\lfloor \cdot\rfloor $ to lighten notation):
\[
\begin{split}
Q_3 & = \sum_{k=\log_2(\frac 1\gep)}^{\log_2(\frac{x_n}{r_n}) -1} \P\left( S_n \geq x_n, M_n \in (2^{-k-1}x_n, 2^{-k}x_n] \right) \\
    & \leq \sum_{k=\log_2(\frac 1\gep)}^{\log_2(\frac{x_n}{r_n})-1} n\overline F(2^{-k-1}x_n) \P\big(S_{n-1} \geq \tfrac12 x_n , M_{n-1} \leq 2^{-k}x_n \big) \,.
\end{split}
\]
Then, using Potter's bound~\eqref{cond:F2},  we have that $n\overline F(2^{-k-1}x_n)\leq c_0 \kappa(2^{k+1})n \overline F(x_n)$ uniformly in $k\geq 1$ for some $c_0>0$, for $n$ large enough.
Using the Fuk--Nagaev inequality of Lemma~\ref{lem:FN2}, we get that
\[
\begin{split}
\frac{Q_3}{n \overline F(x_n)} \leq c_0 \sum_{k=\log_2(\frac1\gep)}^{\log_2(\frac{x_n}{r_n})-1} \kappa(2^{k+1}) e^{2^{k}}\Big(1+ \frac{2^{-k} x_n^2}{n \mm(2^{-k} x_n)} \Big)^{-2^k} \leq c_0 \sum_{k=\log_2( \frac1\gep)}^{\log_2(\frac{x_n}{r_n})-1} \kappa(2^{k+1}) e^{2^k} \Big(\frac{11}{4}\Big)^{-2^k}  \,,
\end{split}
\]
where we have used that $\frac{x_n y}{n\mm(y)} \geq 7/4$ uniformly for $y\geq 2 r_n$ (provided that $n$ is large enough). 
Indeed, we have that $\frac{x_n y}{n\mm (y)} \geq (1+o(1)) 2\frac{x_n r_n}{n \mm(r_n)}$ uniformly for $y\geq r_n$, and the lower bound goes to $2$ as $n\to\infty$.
We therefore have shown that 
\begin{equation}
\label{Q3prop2.8}
\frac{Q_3}{n\overline F(x_n)} \leq c_0\sum_{k=\log_2( \frac1\gep)}^{\infty} \kappa(2^{k+1}) (11/4e)^{-2^k}\xrightarrow{} 0,\quad \text{as } \gep \downarrow 0.
\end{equation}
The convergence is due to the definition of $\kappa(\cdot)$ and the fact that $11>4e.$

It remains to treat $Q_4$. 
We apply again Potter's bound~\eqref{cond:F2} to obtain that 
\[
Q_4 \leq n\overline F(r_n) \P\big(S_{n-1} \geq \tfrac12 x_n, M_{n-1}\leq 4 r_n \big) \leq c n \overline F(x_n) \kappa\Big(\frac{x_n}{r_n}\Big) \P\big(S_{n-1} \geq \tfrac12 x_n, M_{n-1}\leq 4 r_n \big)\,.
\]
For the last probability, the Fuk--Nagaev inequality of Lemma~\ref{lem:FN2} is not enough, because of the first factor $e^{x_n/r_n}$.
However, thanks to Proposition~\ref{prop:tilting} (one can replace $n$ by $n-1$, $x_n$ by $\frac12 x_n$ and $r_n$ by $4r_n$, that changes only some constants, see in particular~\eqref{changemeasure}), we have 
\begin{equation}
\label{eq:applyProptilting}
\P\big(S_{n-1} \geq  x_n/2, M_{n-1}\leq 4 r_n \big)
\leq e^{- nH(\frac{x_n}{2n})} \leq e^{-c \frac{x_n^2}{n \mm(r_n)}} \leq e^{- c \frac{x_n}{r_n}}\,,
\end{equation}
see~\eqref{changemeasure} and Lemma~\ref{lem:entropy}; we also have used that $n \mm(r_n)\sim r_n x_n$.
Hence, we get that
\[
\frac{Q_4}{n \overline F(x_n)} \leq c\kappa \Big(\frac{x_n}{r_n}\Big)e^{- c x_n/r_n} \,,
\]
which goes to $0$ since $x_n/r_n \to\infty$ as noticed above.

All together, this concludes the proof of Proposition~\ref{prop:decomp} in the case where $\lim_{n\to\infty} \frac{x_n}{n} =0$ and $\lim_{n\to\infty} \frac{x_n}{a_n} =\infty$.
In the case where $\liminf_{n\to\infty} \frac{x_n}{n}>0$, we get that $r_n =O(1)$. Then, by exactly the same proof we can control the terms $Q_1,Q_2,Q_3$. It then simply remains to show that for any fixed $R>0$, $\hat Q_4:= \P(S_n\geq x_n, M_n \leq R)$ is negligible compared to $Q_1 \sim n \overline{F}(x_n)$. 
But this is simply a large deviation estimate for bounded random variables (with negative mean): we get that $\hat Q_4 \leq e^{- c x_n}$ (note that $\hat Q_4=0$ if $x_n\geq Rn$), which is indeed negligible compared to $n \overline F(x_n)$ since $\overline F$ decays slower than any exponential function.
\qed

\subsection{Proofs of Theorems~\ref{thm:bigjump},  of Corollary~\ref{cor:overshoot} and of Corollary~\ref{cor:extendzero1}}
\begin{proof}[Proof of Theorem~\ref{thm:bigjump}]

By Proposition \ref{prop:decomp} and equation \eqref{sntomn2}, we have that 
\[
\P(S_n-b_n\geq x_n, M_n>r_n)\sim \P(S_n-b_n\geq x_n, M_n \geq x_n)\sim \P(M_n\geq x_n).
\]
Then we apply \eqref{othersnormalMnlarge} to obtain \eqref{dtv:big}. The fact $\lim_{n\to\infty}\P(M_{n-1}\leq r_n)=0$ entails \eqref{dtv:gaussian}. The convergence \eqref{normalSnlarge} is a direct consequence of \eqref{dtv:big} and \eqref{dtv:gaussian}.  The last result is obtained simply by applying Theorem \ref{thm:rozo6}. This completes the proof.
\end{proof}

\begin{proof}[Proof of Corollary~\ref{cor:overshoot}] 
Assume $\mu=0$ for simplicity so that in particular $b_n \equiv 0$. Recall we assume $\lim_{n\to\infty} \frac{x_n}{a_n} =\infty$. 
\smallskip

\noindent
\emph{Item 1.} 
Recall in this case we assume $\lim_{n\to\infty} \frac{x_n}{n} =0$. For $A_n \in \mathcal{B}(\R^n)$ we have, using a change of measure as in~\eqref{changemeasure},
\begin{equation*}
\begin{split}
   \P\big( (\xi_1,\dots,\xi_n) \in A_n \,\big|\, S_n\geq x_n, M_n\leq ar_n \big)
    & = \frac{\P\big( (\xi_1,\dots,\xi_n) \in A_n, S_n\geq x_n, M_n\leq ar_n \big)}{\P(S_n\geq x_n, M_n\leq ar_n)}  \\
  & = \frac{\E_{\lambda_n}\big[ e^{-\lambda_n (S_n- x_n)} \ind_{\{S_n-x_n \geq 0\}} \ind_{A_n}(\xi_1,\dots,\xi_n) \big]}{\E_{\lambda_n}\big[ e^{-\lambda_n (S_n- x_n)} \ind_{\{S_n-x_n \geq 0\}}\big]}\,.
\end{split}
\end{equation*}
Note that \eqref{eq:BerryEsseen} together with the argument for \eqref{BerryEsseen2} actually
gives the quantitative bound 
\begin{align}
  & \sup_{0 \le b_1 < b_2 \le \infty} \Big| \E_{\lambda_n}[e^{-\lambda_n (S_n-x_n)} \ind_{b_1 \le S_n - x_n < b_2}] - \E[e^{- \lambda_n \sqrt{n\mm(r_n)} Z} \ind_{b_1 \le Z \le b_2\}}] \Big| \notag \\
  & \hspace{24em} 
    = O\Big( \tilde q^*(r_n) \frac{r_n}{\sqrt{n \mm(r_n)}} \Big)\,.
\label{BerryEsseen2b}
\end{align}
Using
$\E_{\lambda_n}\big[ e^{-\lambda_n (S_n- x_n)} \ind_{\{S_n-x_n \geq
  0\}}\big] \sim \frac{1}{\sqrt{2\pi}} \frac{a_n}{x_n}
\sqrt{\frac{\mm(r_n)}{\mm(a_n)}}$ (see~\eqref{eq:asymptildeSn}) and \eqref{asymprn}, together with the fact that $q^*(r_n) \to 0$,
then yields
\begin{align}
 \P&\big( (\xi_1,\dots,\xi_n) \in A_n \,\big|\, S_n\geq x_n, M_n\leq ar_n \big) \notag \\
  & = \lambda_n \sqrt{2\pi n \sigma^2(r_n)}
    \, \E_{\lambda_n}\big[ e^{-\lambda_n (S_n- x_n)} \ind_{\{S_n-x_n \geq 0\}} \ind_{A_n}(\xi_1,\dots,\xi_n) \big]
    \big( 1 + o(1) \big).
    \label{eq:conditionalnormal}
\end{align}
In fact, more quantitatively, the $o(1)$ term above is $O(q^*(r_n))$.

Put
\begin{equation*}
  u_n := \lambda_n \sqrt{n \sigma^2(r_n)} \sim \frac{1}{r_n} \sqrt{n \sigma^2(r_n)} \sim\sqrt{x_n/r_n}\,.
\end{equation*}
Notice that combining $\lambda_n=\lambda(x_n/n) \sim 1/r_n$ from Corollary~\ref{cor:lambdasymp}
with \eqref{asymprn} shows that $\lim_{n\to\infty}u_n$ diverges.

For $0 \le z_1 < z_2 < \infty$ let
\begin{equation}
  A_n = \{ (y_1,\dots,y_n) \in \R^n : x_n + u_n z_1 \le y_1+\cdots+y_n < x_n + u_n z_2 \} \,,
\end{equation}
so that $\{ (\xi_1,\dots,\xi_n) \in A_n \} = \{ z_1 \le (S_n-x_n)/u_n < z_2 \}$. Then from~\eqref{eq:conditionalnormal} and~\eqref{BerryEsseen2b}, using the definition of $u_n$, we get 
\begin{align*}
\P&\big( z_1 \le (S_n-x_n)/u_n < z_2 \,\big|\, S_n\geq x_n, M_n\leq ar_n \big) \notag \\
  & = \E[\sqrt{2\pi}\, u_n e^{- u_n Z} \ind_{z_1 u_n \le Z < z_2 u_n\}}] \big( 1 + o(1) \big)
    = \int_{z_1 u_n}^{z_2 u_n} e^{-u_n z - z^2/2} \, u_n dz \, \big( 1 + o(1) \big) \notag \\
  & \xrightarrow[n\to\infty]{} e^{-z_1} - e^{-z_2}\,,
\end{align*}
which concludes the proof of item 1.

\smallskip
\noindent
\emph{Item 2.}  Recall in this case we assume that $\overline{F}$ is intermediate regularly varying. The arguments for Proposition~\ref{prop:decomp} show that  $\P(S_n \geq x_n, M_n >r_n)\sim n \overline F(x_n)$, recall \eqref{snmnbigasymp}. 
Applying this again with $x_n$ replaced by $x_n' \ge x_n$ yields \eqref{overshoot big}.
\end{proof}

\begin{proof}[Proof of Corollary~\ref{cor:extendzero1}]
By Proposition \ref{prop:decomp}, we have that \[
\lim_{n\to\infty}\P(M_n\geq x_n \,|\, S_n-b_n\geq x_n) = s
\quad \text{ and }\quad
\lim_{n\to\infty}\P(M_n\leq r_n \,|\, S_n-b_n\geq x_n)= 1-s\,.
\]
Conditionally on $\{S_n-b_n\geq x_n, M_n\geq x_n\}$, if we write $S_n=M_n+S_n'$, then $S_n'-b_n$ is of order $a_n$ thanks to \eqref{normalSnlarge} and~\eqref{attract}:
since $x_n/a_n\to\infty$, we obtain that
\[
\LL\Big(\frac{M_n}{S_n-b_n}\,\Big|\, S_n-b_n\geq x_n, M_n\geq x_n\Big)\stackrel{w}{\longrightarrow}\delta_1.
\]
On the other hand, since we have that $r_n/x_n\to0$ (using that $x_n/a_n\to\infty$), we get 
\[
\LL\Big(\frac{M_n}{S_n-b_n}\,\Big|\, S_n-b_n\geq x_n, M_n\leq r_n\Big)\stackrel{w}{\longrightarrow}\delta_0.
\]
This concludes the proof.
\end{proof}

\section{Proofs of the local versions of the theorems}
\label{sec:LLD}

\subsection{Proof of Proposition~\ref{prop:tiltinglocal}}
Recall $x_n\geq a_n$ with $\lim_{n\to\infty}\frac{x_n}{n}=0$. For simplicity, we also assume $\mu=0.$ Analogous to\ \eqref{changemeasure}, we have that
\[
\P( S_n = x_n, M_n \leq a r_n)  =e^{- n H(\frac{x_n}{n})} \P_{\lambda}\big( S_n =x_n \big)  
\]
with $\lambda = \lambda(\frac{x_n}{n})$.
Then, we can apply a local limit theorem to estimate $\P_{\lambda}\big( S_n =x_n \big)$; in particular, since the distribution $\P_{\lambda}$ depends on $n$ and $x_n$, we need explicit estimates on the rate of convergence.
Several results exist in this direction, such as \cite[Thm.~5]{M92} or~\cite{GW17,SN21}, but none of these appear to give a sufficient bound for our purpose.
Instead, we start from~\cite[Lem.~3]{D97}: setting $\sigma_\lambda^2:= \mathrm{Var}_{\lambda}(\xi)$ and $\nu_{\lambda} :=\E_{\lambda}[|\xi - \frac{x_n}{n}|^3]$, it gives that
\begin{equation}
\label{localCLT}
\Big| \sqrt{n} \sigma_{\lambda} \P_{\lambda}(S_n=x_n) -  \frac{1}{\sqrt{2\pi}}\Big| \leq  C \frac{\nu_{\lambda}}{\sqrt{n} \sigma_{\lambda}^3} + 2 \sqrt{n\sigma_{\lambda}}  \int_{\ell}^{\pi} e^{-n(1-|\phi_{\lambda}(t)|)} \dd t \,,
\end{equation}
with $\ell := \frac{\sigma_{\lambda}^2}{4\nu_{\lambda}}$
and $\phi_{\lambda}(t) = \E_{\lambda}(e^{i t \xi})$ the characteristic function of $\xi$ under $\P_{\lambda}$.
Since we have that $\sigma^2_{\lambda}\sim \mm(r_n)$, see the first two results in \eqref{momentsE}, we only have to show that both terms on the r.h.s.\ of~\eqref{localCLT} go to zero.

For the first term, we have from~\eqref{momentsE} that
\[
\frac{\nu_{\lambda}}{\sqrt n \sigma_{\lambda}^3} \leq C  \frac{ \tilde q^*(r_n) r_n}{ \sqrt{ n \mm(r_n)}} \,,   
\]
which goes to $0$ as proved for the last term in \eqref{eq:BerryEsseen}.

For the remaining term, notice that thanks to \eqref{momentsE} we have that $\ell =\frac{\sigma_{\lambda}^2}{4\nu_{\lambda}} \geq \ell_n := \frac{c}{\tilde q^*(r_n) r_n}$ and $\mm_{\lambda}\sim \mm(r_n)$.
We fix $\epsilon>0$ (small enough so that the last inequality in~\eqref{eq:inequalityg} below holds uniformly on $[0,\epsilon]$) and we need to show that both of the following terms go to $0$:
\[
I_1 = \sqrt{n \mm(r_n) } \int_{\ell_n}^{\epsilon} e^{-n(1-|\phi_{\lambda}(t)|)} \dd t \,,
\qquad 
I_2 = \sqrt{n \mm(r_n) } \int_{\epsilon}^{\pi} e^{-n(1-|\phi_{\lambda}(t)|)} \dd t \,.
\]

Let us now estimate $1-|\phi_{\lambda}(t)|$.
Since $\xi$ has maximum span $1$ (equivalent to $(S_n)$ being aperiodic; see the definition of maximum span in \cite{G54}), we can choose some $K>0$ such that  the law of $\xi$ restricted to $\{-K,\ldots, K\}$ has maximum span $1$.
Then, it is standard to get that for any $0<\gep<1$, there is a constant $c_0>0$ (that depends on $K,\epsilon$) such that for any $t\in [\gep,\pi]$
\[
\left|\E\big[e^{it\xi} \, \big|\, |\xi|\leq K \big]\right| \leq 1-c_0 \,,
\]
see \textit{e.g.}\ \cite[\S14, Cor.~2 to Thm.~5]{G54}.
Now, turning back to $\phi_{\lambda}$, we have that, for $n$ large enough so that $K\leq r_n$,
\[
\begin{split}
|\phi_{\lambda}(t)| = \left|\E_{\lambda}\big[e^{it\xi} \ind_{\{|\xi|\leq K\}}\big]\right| + \left|\E_{\lambda}\big[e^{it\xi} \ind_{\{|\xi|> K\}}\big]\right|
\leq  \left|\E_{\lambda}\big[e^{it\xi} \ind_{\{|\xi|\leq K\}}\big]\right| + \P_{\lambda}(|\xi|>K)\,.
\end{split}
\]
Now, recalling that $\lambda\to 0$, we can apply Taylor expansion to $e^{\lambda \xi}$ on the event $\{|\xi|\leq K\} $: for $n$ large enough, we obtain
\[
\begin{split}
M(\lambda)\left| \E_{\lambda}\big[e^{it\xi} \ind_{\{|\xi|\leq K\}}\big] \right|
& = \left|\E\big[e^{it\xi +\lambda \xi} \ind_{\{|\xi|\leq K\}}\big]\right| \\
 &\leq \P(|\xi|\leq K) \left( \left| \E_{\lambda}\big[ e^{it\xi} \,\, \big|\, \, |\xi|\leq K \big] \right| +  2\lambda K  \right)\leq  (1-\tfrac{c_0}{2}) \P(|\xi|\leq K)\,.
\end{split}
\]
Now, since $\lambda \to 0$, we also easily get that 
$M(\lambda)\P(|\xi|\leq K) =(1+o(1)) \P_{\lambda}(|\xi|\leq K)$, so we obtain that for $n$ large enough
\[
|\phi_{\lambda}(t)| \leq (1-\tfrac{c_0}{3}) \P_{\lambda} (|\xi|\leq K) +\P_{\lambda}(|\xi|>K) 
= 1- \tfrac{c_0}{3} \P_{\lambda} (|\xi|\leq K) \leq 1- c'\,,
\]
where we also have used for the last inequality that $\P_{\lambda}(|\xi|\leq K)=(1+o(1)) \P(|\xi|\leq K)$, since $\lambda \to 0$, $M(\lambda)\to 1$.
All together, we get that there is a constant $c'=c'(\gep)>0$ such that
$1-|\phi_{\lambda}(t)| \geq c'$ for any $t\in [\gep,\pi]$.
Therefore, we obtain that
\[
I_2 \leq \sqrt{n \mm(r_n) }  \pi e^{-c' n}  \to 0\,.
\]

To deal with $I_1$, we use the idea of~\cite[\S50, Lemma]{G54}: we prove that, for all $t \in [\frac{1}{r_n},\epsilon]$,
\begin{equation}
\label{boundphi}
1-|\phi_\lambda(t)| \geq c t^2 \mm(1/t) \,.
\end{equation}
One could have used Lemma~4 in\cite{M92}, together with Lemmas~1-2 in\cite{M92} to get that $1-|\phi_{\lambda}(t)|\geq ct^2$ uniformly for $t\in [0,\gep]$; this is however not enough for our purpose.

Let us introduce $\P_{\lambda}^*$, the distribution of the symmetrised version of $\P_{\lambda}$, \textit{i.e.}\ the law of $\xi^* =\xi-\xi'$, where $\xi,\xi'$ are two independent copies with law $\P_{\lambda}$.
Then, $|\phi_{\lambda}(t)|^2 = \phi_{\lambda}^*(t) = \E^*_{\lambda}[e^{it\xi^*}]$, so we only need to show that $1- \phi_{\lambda}^*(t)\geq c t^2 \mm(1/t)$.
We have that
\[
1-\phi_\lambda^*(t) = \int_{\mathbb{R}} (1-\cos(tx)) \dd \P_{\lambda}^*(x) \geq \int_{-1/t}^{1/t} (1-\cos(tx)) \dd \P_{\lambda}^*(x) \,,
\]
using that $1-\cos(tx)$ is always non-negative.
Now, expanding the cosine (using $|tx|\leq 1$), we get
\[
1-\phi^*(t) \geq  c t^2 \int_{-1/t}^{1/t} x^2 \dd \P_{\lambda}^*(x) \,.
\]
Now, let $C>0$ be fixed such that $c' =\P(|\xi| \leq C) >0$; we also have $\P_{\lambda}(|\xi|\leq C)\geq \frac{c'}{2}$ for $n$ large enough, since $\lambda\to0$.
Then, for $t$ small enough so that $1/t \geq 4C$,
\[
\int_{-1/t}^{1/t} x^2 \dd \P_{\lambda}^*(x) \geq  \E_{\lambda}^{\otimes 2} \left[ (\xi-\xi')^2  \ind_{\{|\xi -\xi'|\leq \frac1t\}} \ind_{\{|\xi'|\leq C, |\xi|\geq 2C\}}\right] \geq \frac{c'}{8} \E_{\lambda}\left[ \xi^2 \ind_{\{2C\leq |\xi|\leq \frac1t -C\}} \right]\,.
\]
Then, since $1/t \leq r_n$ and $\lambda \leq c/r_n$ (recall Corollary~\ref{cor:lambdasymp}), we get that
\[
\E_{\lambda}\left[ \xi^2 \ind_{\{2C\leq |\xi|\leq \frac1t -C\}} \right] \geq M(\lambda)^{-1} e^{-\lambda r_n} \E\left[ \xi^2 \ind_{\{2C\leq |\xi|\leq \frac1t -C\}} \right] \geq c' \mm(1/t)  \,,
\]
using also that $M(\lambda)\to1$ as $\lambda\to0$. This concludes the proof of~\eqref{boundphi}.

With~\eqref{boundphi} at hand, we get that 
\[
I_1 \leq \sqrt{n\mm(r_n)} \int_{\ell_n}^{\gep} e^{- c n t^2 \mm(1/t)} \dd t \,.
\]
Now, for $t\in (0,1]$, let $g(t) = t \bar \sigma^2(1/t)$, with $\bar \sigma^2(y) = 2\int_{0}^{y} u \P(|\xi|>u) \dd u$ which verifies $\bar\sigma(y)\sim \mm(y)$ as $y\to\infty$ (see~\eqref{ddbar}).
Then, $g$ is differentiable and one can easily check that $g'(t)\geq 0$, at least for $t$ small enough (\textit{e.g.}\ using that $\P(|\xi|>1/t) = o(t^2\mm(1/t^2))$, see~\eqref{f<sigma}).
We also let $h(t) =\int_0^t g(u) \dd u$, which verifies $h(t)\sim t^2\mm(1/t) $ as $t\to 0$.
Then, by an integration by parts, we get that
\begin{equation}
\label{eq:inequalityg}
\int_{\ell_n}^{\gep} e^{- c n h(t)} \dd t 
\leq   \frac{1}{ c n g(\ell_n)} e^{- c n h(\ell_n)} - \int_{\ell_n}^{\gep}  \frac{g'(t)}{ c ng(t)^2} e^{- c n h(t)} \dd t 
\leq \frac{1}{ c ng(\ell_n)} e^{- c n h(\ell_n)} \,,
\end{equation}
where the last inequality holds if $\gep$ has been fixed small enough so that $g'(t)\geq 0$ on $[0,\gep]$.
Now, using that
$g(\ell_n) = \ell_n \bar\sigma^2(r_n)$ and $h(\ell_n) \geq (1+o(1))\ell_n^2 \mm(r_n)$, 
\[
I_1 \leq \sqrt{n\mm(r_n)} \int_{\ell_n}^{\gep} e^{- c n h(t)} \dd t \leq \frac{1}{\ell_n \sqrt{n \mm(r_n)}} e^{-c n h(\ell_n)} \leq  \frac{c' \tilde q^*(r_n) r_n}{\sqrt{n \mm(r_n)}} e^{-  \frac{c'' n \mm(r_n)}{\tilde q^*(r_n)^2 r_n^2}} \,,
\]
where we have used that $\ell_n  = \frac{c}{\tilde q^*(r_n) r_n}\geq \frac{1}{r_n}$
for the last inequality.

Since $r_n/\sqrt{n \mm(r_n)}$ remains bounded (recall~\eqref{asymprn} and the sentence below), we get that 
$I_1 \leq c_1 \tilde q^*(r_n) e^{- c_2/\tilde q^*(r_n)^2}$, which proves that $I_1$ goes to $0$ as $n\to\infty$.

Going back to~\eqref{localCLT}, we have obtained that $\P_{\lambda}(S_n=x_n) \sim \frac{1}{\sqrt{2\pi n \mm(r_n)}}$, which finally gives that
\[
\P( S_n = x_n, M_n \leq a r_n)  \sim  \frac{1}{\sqrt{2\pi n \mm(r_n)}}  e^{- n H(\frac{x_n}{n})}  \,.
\vspace{-1.5\baselineskip} 
\]
\qed

\begin{remark}
Similarly to what is done for the integral version of the theorem, the case where $x_n=O(a_n)$ in the proof of Proposition~\ref{prop:tiltinglocal} brings some simplifications: using the first sentence of Remark~\ref{rem:rozov}, we recover the local limit theorem:
\[
\P( S_n = x_n, M_n \leq a r_n)  \sim  \frac{1}{\sqrt{2\pi} a_n}  e^{- \frac{x_n^2}{2a_n}} \,.
\]
\end{remark}

\subsection{Proof of Proposition~\ref{prop:decomplocal}}

Recall that we assume that $\P(\xi=x)$ is intermediate regularly varying and \eqref{cond:taillocal} holds, and $\lim_{n\to\infty}\frac{x_n}{a_n}=\infty, \lim_{n\to\infty}\frac{x_n}{n}=0$. For simplicity, we also assume $\mu=0.$
As in the proof of Proposition~\ref{prop:decomp}, we focus on the case where $\lim_{n\to\infty} \frac{x_n}{n} =0$.
(We briefly discuss at the end of the proof the case $\liminf_{n\to\infty} \frac{x_n}{n} >0$, which can be found e.g.\ in \cite{AL09}).
We decompose the probability as before, in the following terms:
\[
\begin{split}
& Q_0 = \P(S_n =x_n , M_n >(1+\gep)x_n) \,, 
   \quad Q_1  =\P(S_n = x_n,   (1-\gep) x_n \leq  M_n \leq (1+\gep) x_n ) \,,\\ 
&     Q_2 = \P(S_n = x_n, \gep x_n < M_n < (1-\gep)x_n )\\
    & Q_3  = \P(S_n = x_n, 4 r_n< M_n\leq \gep x_n) \,,\quad
    Q_4 = \P(S_n = x_n, r_n < M_n \leq 4r_n  ) \,.
\end{split}
\]

For the term $Q_2$, we have, for some constant $c = c(\gep) \in (0,\infty)$,
\[
\begin{split}
Q_2 & \leq \sum_{y \in  (\gep x_n, (1-\gep)x_n)} n \P(\xi=y) \P(S_{n-1} = x_n-y, M_{n-1} \leq y ) \\
& \leq  c n \P(\xi=x_n) \sum_{y \in  (\gep x_n, (1-\gep)x_n]}  \P(S_{n-1} = x_n-y) \leq  c n \P(\xi=x_n)  \P(S_{n-1} \geq \gep x_n)
\end{split}
\]
where we have used Potter's bound~\eqref{cond:F2} 
for the second line.
Then, since  $\lim_{n\to\infty} \frac{x_n}{a_n} =\infty$,  we have $\lim_{n\to\infty}\P(S_{n-1} \geq \gep x_n)=0$, so $\lim_{n\to\infty} \frac{Q_2}{n \P(\xi=x_n)} =0$.

For the term $Q_3$, we have (omitting integer parts to lighten notation)
\[
\begin{split}
Q_3 & = \sum_{k=\log_2(\frac 1\gep)}^{\log_2(\frac{x_n}{r_n}) -1} \P\left( S_n = x_n, M_n \in (2^{-k-1}x_n, 2^{-k}x_n] \right) \\
    & \leq \sum_{k=\log_2(\frac 1\gep)}^{\log_2(\frac{x_n}{r_n})-1}  \sum_{y\in  (2^{-k-1}x_n, 2^{-k}x_n]} n \P(\xi=y) \P\big(S_{n-1}  =x_n-y , M_{n-1} \leq 2^{k} x_n \big)  \\
    &\leq c n\P(\xi=x_n) \sum_{k=\log_2(\frac 1\gep)}^{\log_2(\frac{x_n}{r_n})-1}   2^{c k}  \P\big(S_{n-1} \geq \tfrac12 x_n , M_{n-1} \leq 2^{-k}x_n \big) \,,
\end{split}
\]
where we have used \eqref{eqn:kpoly} and summed over $y\in  (2^{-k-1}x_n, 2^{-k}x_n]$.
We conclude as in the proof of Proposition~\ref{prop:decomp} that
$\frac{Q_3}{n\P(\xi=x_n)} \leq  c'\gep^3 (11/4e)^{-1/\gep}$, see~\eqref{Q3prop2.8}.

For the term $Q_4$, we also get thanks again to  \eqref{eqn:kpoly}
\[
\begin{split}
Q_4 & \leq \sum_{y\in (r_n, 4r_n]} n \P(\xi=y) \P\big(S_{n-1}  = x_n-y, M_{n-1}\leq 4 r_n \big) \\
& \leq c n \P(\xi=x_n) \Big(\frac{x_n}{r_n}\Big)^c \P\big(S_{n-1} \geq \tfrac12 x_n, M_{n-1}\leq 4 r_n \big)\,,
\end{split}
\]
and the conclusion follows exactly as in the proof for the term $Q_4$ of Proposition~\ref{prop:decomp}.

For the term $Q_0$, by sub-additivity and using the ``almost monotonicity'' \eqref{cond:taillocal}, we obtain 
\[
Q_0\leq n\sum_{y>(1+\epsilon) x_n}\P(\xi=y)\P(S_{n-1}=x_n-y) \leq c \P(\xi =x_n) \P(S_{n-1} \leq - \gep x_n) \,,
\]
which is negligible compared to $\P(\xi=x_n)$, since $\lim_{n\to\infty} \frac{x_n}{a_n} =\infty$.

\smallskip
It remains only to deal with $Q_1$; we follow Doney's approach~\cite{D01}. We prove a lower and an upper bound separately.
We have
\begin{align*}
Q_1 & \geq n\P(S_n=x_n,  \exists\, i  \text{ s.t. }\xi_i\in ((1-\gep)x_n, (1+\gep)x_n), \forall \, j\neq i, \xi_j\leq (1-\gep)x_n )\\ 
& \geq  n\sum_{y=(1-\epsilon)x_n}^{(1+\epsilon)x_n} \P(\xi=y) \P(S_{n-1}=x_n-y, M_{n-1}\leq (1-\epsilon)x_n) \\
&\geq n \; \min_{|y-x_n| \leq \gep x_n}\P(\xi=y) \; \, \P\big( |S_{n-1}| \leq \epsilon x_n, M_{n-1}\leq (1-\epsilon)x_n \big).
\end{align*}
Since $\lim_{n\to\infty} \frac{x_n}{a_n} =\infty$, the last probability in the last line goes to $1$.
Using also that $\P(\xi=x)$ is intermediate regularly varying, recall~\eqref{def:intermediateregvar}, we have that 
\[
\lim_{\gep\downarrow 0} \liminf_{n\to\infty} \min_{|y-x_n| \leq \gep x_n} \frac{\P(\xi=y)}{\P(\xi=x_n)} =1 \,.
\]
We therefore end up with 
$
\lim_{\gep\downarrow 0} \liminf_{n\to\infty}\frac{Q_1}{n\P(\xi=x_n)}\geq 1.$

In the other direction, we have 
$$Q_1\leq n\sum_{y=(1-\epsilon)x_n}^{(1+\epsilon)x_n}\P(\xi=y)\P(S_{n-1}=x_n-y)\leq n\max_{|y-x_n| \leq \gep x_n}\P(\xi=y).$$
Using again the intermediate regular variation of $\P(\xi=x)$, we therefore end up with
$\lim_{\gep\downarrow 0} \limsup_{n\to\infty} \frac{Q_1}{n\P(\xi=x_n)}\leq 1.$

\smallskip
In the case $\liminf_{n\to\infty} \frac{x_n}{n} >0$, we can deal with all the terms $Q_0,Q_1,Q_2,Q_3$ exactly as above.
Recalling that in this case we have $r_n=O(1)$, its then only remains to show that for any $R>0$ the term $\hat Q_4 :=\P(S_n=x_n , M_n \leq R)$ is negligible compared to $n \P(\xi=x_n)$.
But one simply bounds $\hat Q_4 :=\P(S_n=x_n , M_n \leq R)\leq \P(S_n\geq x_n , M_n \leq R) \leq e^{-c x_n}$, by a standard large deviation bound (already mentioned in the integral case); this is negligible compared to $n \P(\xi=x_n)$.
\qed

\subsection{Proof of Theorem~\ref{thm:bigjumplocal} }
\label{sec:proofBJlocal}

Recall that we assume $\lim_{n\to\infty} \frac{x_n}{a_n} =\infty$ and that $\P(\xi=x)$ is intermediate regularly varying. For simplicity of notation, we also assume that $\mu=0$, so $b_n=0$.

We start with the proof of the first convergence, following the approach in \cite{AL09}.  
Let $0<\epsilon<1$. Using \eqref{eq:proplocal}, we know that 
\[
\lim_{n\to\infty}\P\big(S_n=x_n, |M_n-x_n|\leq \epsilon x_n\,\big|\,S_n=x_n, M_n>r_n\big)=1 \,.
\]
For simplicity of notation, let 
\[
\M_1=\big(\mathscr{L}(\xi)\big)^{\otimes (n-1)}, \quad \M_2=\mathscr{L}\big( R(\xi_1,\dots,\xi_n) \, \big| \, S_n=x_n, |M_n-x_n|\leq \epsilon x_n \big)\,,
\]
so we only need to show that
$\lim_{n\to\infty}d_{TV}(\M_1,\M_2)=0$.

Let $A$ be any measurable set in~$\Z^{n-1}$.
Let $0<\epsilon<1$. We define the event 
\[
B_n = \{S_n- b_n=x_n \}\cap \{ M_{n-1}< (1-\epsilon) x_n\}\cap \{|S_{n-1}- b_n| < \epsilon x_n\} \,,
\]
which is a subset of the event $\{M_{n-1}<(1-\epsilon)x_n<\xi_n=M_n<(1+\epsilon)x_n\}$.
Then
\begin{align*}
\P&(S_n =x_n,\, |M_n-x_n|\leq \epsilon x_n)\M_2(A) =\P(R(\xi_1,\ldots,\xi_{n})\in A, \, S_n =x_n,\, |M_n-x_n|\leq \epsilon x_n )\\
&\geq  n\P\Big((\xi_1,\ldots,\xi_{n-1})\in A, \,  B_n ,\, |\xi_n-x_n|<\epsilon x_n\Big)\\
& \geq n  \min_{|t-x_n|<\epsilon x_n}\P(\xi=t) \M_1\Big((\xi_1,\ldots,\xi_{n-1})\in A, \, M_{n-1} < (1-\epsilon)x_n, \, |S_{n-1}| < \epsilon x_n\Big).
\end{align*}
Using the Potter's bound \eqref{claim:equivalentcondF2} 
for $\P(\xi=x)$, we obtain that
\[
\min_{|t-x_n|<\epsilon x_n}\P(\xi=t)\geq c(\epsilon) \P(\xi=x_n)
\]
for $n$ large enough,
with $\lim_{\gep\downarrow0}c(\epsilon)=1$.
We also have 
\[
\lim_{n\to\infty}\M_1(M_{n-1} <  (1-\epsilon)x_n)= 1,\quad  \lim_{n\to\infty}\M_1(|S_{n-1}| < \epsilon x_n)= 1\,,
\]
using that $\lim_{n\to\infty} \frac{x_n}{a_n}=\infty$ for the last limit.
Thus we get that
\begin{align*}
\M_2(A) \geq \frac{n\P(\xi=x_n)}{\P(S_n- b_n=x_n, |M_n-x_n|\leq \epsilon x_n)}\M_1(A)(1+o(1)),
\end{align*}
where the error term $o(1)$ is uniformly small in all sets $A \subset \Z^{n-1}$.  
Using Proposition~\ref{prop:decomplocal}, we obtain that
$\M_2(A)\geq (1+o(1)) \M_1(A)$ as $n\to\infty$.
Since this holds uniformly for any set~$A$, we conclude that  
$\lim_{n\to\infty} d_{\mathrm{TV}}(\M_1, \M_2)= 0$.

For the second convergence, letting $\tilde\M_2=\mathscr{L}\big( R(\xi_1,\dots,\xi_n) \, \big| \, S_n=x_n, M_n\leq r_n\big)$, we also need to prove that
$d_{\mathrm{TV}}(\M_1, \tilde \M_2)= 1$.
But this is clear: if one writes $R(\xi_1,\ldots, \xi_n) = (\eta_1,\ldots, \eta_{n-1})$ and consider the event $A_n= \{ \sum_{i=1}^{n-1} \eta_i \geq x_n- r_n\}$, then we naturally have that $\tilde \M_2(A_n) =1$, but $\lim_{n\to\infty} \M(A_n) =0$, since $x_n':=x_n-r_n$ verifies $\lim_{n\to\infty}\frac{x_n'}{a_n}=\infty$.

For the last result, \eqref{normalSnlargelocal} follows directly from Proposition~\ref{prop:decomplocal}.
\qed

\section{Law of $\xi_1,\ldots, \xi_n$ conditioned on large $M_n$}
\label{sec:conditional}

In this section, we prove Proposition~\ref{prop:othersnormalMnl}, but first, we give a construction of the regular conditional distribution of $R(\xi_1,\ldots,\xi_n)$ given $M_n$, which is used in the statement~\eqref{othersnormalMnlarge}.

\subsection{A regular conditional distribution}
We give here a concrete construction of
\begin{equation}\label{rcd}
  \mathscr{L}\big( R(\xi_1,\dots,\xi_n) \, \big| \, M_n = x \big),\quad 
  x\in \R
\end{equation}
as follows:
\begin{itemize}
    \item If $\P(x-\delta \leq \xi\leq x+\delta)=0$ for some  $\delta>0,$ \textit{i.e.}\ $x$ is not in the support, we can basically assume any conditional distribution. But for simplicity, let
    \begin{equation}
    \label{mn=x1}
    \mathscr{L}\big( R(\xi_1,\dots,\xi_n) \, \big| \, M_n = x \big)=\mathscr{L}\big(\xi\big)^{\otimes (n-1)}.
    \end{equation}
    
    \item If $\P(\xi=x)>0,$ then we use the usual conditional probability formula to define \eqref{rcd}.
    
    \item If $\P(\xi=x)=0$ and $\P(x-\delta \leq \xi\leq x+\delta)>0$ for any $\delta>0$, it turns out that
    \begin{equation}
    \label{xdelta} 
    \lim_{\delta\to0}d_{\mathrm{TV}}\left(\mathscr{L}\big( R(\xi_1,\dots,\xi_n) \, \big| \, x-\delta\leq M_n \leq x+\delta \big), \mathscr{L}\big(\xi \, \big|\,  \xi \leq x\big)^{\otimes (n-1)} \right)=0.
    \end{equation}
This is proved in Proposition \ref{p=0} below. Then we can define 
\begin{equation}\label{mn=x2}
\mathscr{L}\big( R(\xi_1,\dots,\xi_n) \, \big| \, M_n = x \big)=\mathscr{L}\big(\xi\, \big|\,  \xi \leq x\big)^{\otimes (n-1)}.
\end{equation}

\end{itemize}
Using standard measure-theoretic approaches, the construction indeed gives a version of regular conditional distribution. 
It remains to prove the following proposition. 

\begin{proposition}\label{p=0}
  Let $\xi$ be any real-valued random variable. Assume that $x\in \R$ is such that 
  $\P(\xi=x)=0$ and $\P(x-\delta \leq \xi\leq x+\delta)>0$ for any $\delta>0$, then \eqref{xdelta} holds. \end{proposition}
\begin{proof}
 Fix $n \in \N$ and $x \in \R$ satisfying the assumptions of Proposition~\ref{p=0}.
First of all, by assumption, it is straightforward to see that since $F$ is continuous at $x$,
\begin{equation}\label{mxi}
\P\big(M_n\in [x-\delta,  x+\delta]\big)\sim n\P\big(\xi\in [x-\delta,  x+\delta]\big) F(x)^{n-1}=o(1), \quad \text{ as }\delta\downarrow 0.
\end{equation}
For $A\subset \BB(\R^{n-1})$, let us bound
\begin{align*}
     \Big| \P\big( R(\xi_1,\dots,\xi_n) &\in A\,\big|\, M_n\in [x-\delta,x+\delta]\big)
        - \mathscr{L}\big(\xi\,\big|\, \xi\leq x\big)^{\otimes (n-1)}(A) \Big| \\
      & =\bigg| \frac{\P\big( R(\xi_1,\dots,\xi_n) \in A,  M_n\in [x-\delta,x+\delta]\big)}{\P(M_n\in [x-\delta,x+\delta])}
        - \mathscr{L}\big(\xi\,\big|\, \xi\leq x\big)^{\otimes (n-1)}(A) \bigg| 
\end{align*}
by the sum of three terms $P_1,P_2,P_3$ that we now define and treat separately.
The first term is
\begin{align*}
 P_1 & := \frac{\P(\exists \, i \neq j \le n \, : \, \xi_i, \xi_j \in [x-\delta,x+\delta])}{\P(M_n\in [x-\delta,x+\delta])} \leq \frac{n^2 \P(\xi \in[x-\delta,x+\delta])^2}{\P(M_n \in[x-\delta,x+\delta])} \xrightarrow[\delta\to 0]{} 0 \,, 
\end{align*}
where we have used~\eqref{mxi} and the assumption that $\P(\xi=0)$ to obtain the last limit.
The second term $P_2$ is
\begin{multline*}
\!\!\!\!\bigg| \frac{\P\big( R(\xi_1,\dots,\xi_n) \in A \cap (-\infty, x-\delta)^{n-1}, M_n \in[x-\delta,x+\delta]\big)}{\P(M_n \in[x-\delta,x+\delta])}
        - \mathscr{L}\big(\xi\,\big|\, \xi<x-\delta\big)^{\otimes (n-1)}(A) \bigg| \\
    = \bigg|\frac{n\P(\xi\in [x-\delta,x+\delta])F(x-\delta)^{n-1}}{\P(M_n \in[x-\delta,x+\delta])} - 1 \bigg|
      \mathscr{L}\big(\xi\,\big|\, \xi<x-\delta\big)^{\otimes (n-1)}(A)  \xrightarrow[\delta\to 0]{} 0\,,
\end{multline*}
where we have used~\eqref{mxi} to obtain the last limit.
The last term is 
\begin{align*}
    P_3 & := \Big| \mathscr{L}\big(\xi\,\big|\, \xi\leq x\big)^{\otimes (n-1)}(A)-\mathscr{L}(\xi\,\big|\, \xi< x-\delta\big)^{\otimes (n-1)}(A) \Big| \\
    & \leq n\P\big(\xi\in[x-\delta,x]\,\big|\,\xi\leq x \big)  \xrightarrow[\delta\to 0]{} 0 \,,
\end{align*}
where the last limit is due to the fact that $\P(\xi=x)=0$. This concludes the proof of~\eqref{xdelta}.
\end{proof}

\subsection{Proof of Proposition \ref{prop:othersnormalMnl}}
\label{sec:propMn}

Let us first prove \eqref{othersnormalMnlarge}.
For $A \subset \BB(\R^{n-1})$, we bound
\begin{multline*}
    \bigg| \P\big( R(\xi_1,\dots,\xi_n) \in A\,|\, M_n\geq x_n\big)
        - \P\big( (\xi_1,\dots,\xi_{n-1}) \in A \big) \bigg| \\
       =\bigg| \frac{\P\big( R(\xi_1,\dots,\xi_n) \in A,  M_n \geq x_n\big)}{\P(M_n \geq x_n)}
        - \P\big( (\xi_1,\dots,\xi_{n-1}) \in A \big) \bigg| 
\end{multline*}
by the sum of three terms $P_1,P_2,P_3$ that we now define and treat separately.
The first term is
\begin{align*}
P_1 & :=  \frac{\P(\exists \, i \neq j \le n \, : \, \xi_i, \xi_j \geq  x_n)}{\P(M_n \geq  x_n)}
    \leq  \frac{n^2 \P(\xi \geq  x_n)^2}{\P(M_n \geq x_n)} \xrightarrow[n\to\infty]{} 0 \,,
\end{align*}
since we have 
\begin{equation}
\label{mn>=xn}
\P(M_n \geq x_n) \sim n \P(\xi \geq x_n) = o(1) \quad \text{as } n\to\infty
\end{equation}
by assumption.
The second term $P_2$ is
\begin{multline*}
\!\!\! \left| \frac{\P\big( R(\xi_1,\dots,\xi_n) \in A \cap (-\infty, x_n)^{n-1}, M_n \geq  x_n\big)}{\P(M_n \geq x_n)}
        - \P\big( (\xi_1,\dots,\xi_{n-1}) \in A \cap (-\infty, x_n)^{n-1} \big) \right| \\
        = \left| \frac{n\P(\xi \geq x_n)F(x_n-)^{n-1}}{\P(M_n \geq x_n)} - 1 \right|
        \P\Big( (\xi_1,\dots,\xi_{n-1}) \in A \cap (-\infty, x_n)^{n-1} \Big) \xrightarrow[n\to\infty]{} 0 \,,
\end{multline*}
where the last limit comes from~\eqref{mn>=xn}.
The last term is $P_3 := \P(M_{n-1}\geq x_n)$, which also goes to $0$ as $n\to\infty$, thanks to~\eqref{mn>=xn}.

\smallskip
Now we turn to the proof of \eqref{othersnormalMnlarge.local}. By assumption, for the sequence $(x_n)_{n\geq 1}$, we know that $\P(M_{n-1}\geq x_n)\sim (n-1)\overline F(x_n)=o(1)$ as $n\to\infty$. Then
\begin{equation}\label{<x}
d_{\mathrm{TV}}\Big(  \mathscr{L}\big(\xi \, \big|\, \xi\leq x_n\big)^{\otimes (n-1)} , \,
    \mathscr{L}\big(\xi\big)^{\otimes (n-1)} \Big)\longrightarrow 0,\quad \text{ as } n\to\infty.
\end{equation}
There are two cases we need to treat: (i) $(x_n)_{n\geq 1}$ is a sequence such that $\P(\xi=x_n)=0$ for any $n$; (ii) $(x_n)_{n\geq 0}$ is a sequence such that $\P(\xi=x_n)>0$ for any $n$. For case~(i),  by  \eqref{mn=x1}, \eqref{mn=x2} and \eqref{<x}, we conclude that  \eqref{othersnormalMnlarge.local} holds. Let us now consider case~(ii).  Note that $n\overline F(x_n)=o(1)$ implies that  $n\P(\xi=x_n)=o(1)$ as $n\to\infty$. Then 
\begin{equation}
\label{atom}
\P(M_n=x_n)\sim n\P(\xi=x_n)=o(1) \quad \text{as } n\to\infty.
\end{equation}
For $A \subset \BB(\R^{n-1})$, we again bound 
\begin{multline*}
    \left| \P\big( R(\xi_1,\dots,\xi_n) \in A\,|\, M_n= x_n\big)
        - \P\big( (\xi_1,\dots,\xi_{n-1}) \in A \big) \right| \\
       =\bigg| \frac{\P\big( R(\xi_1,\dots,\xi_n) \in A,  M_n = x_n\big)}{\P(M_n = x_n)}
        - \P\big( (\xi_1,\dots,\xi_{n-1}) \in A \big) \bigg| 
\end{multline*}
by the sum of three terms $\tilde P_1, \tilde P_2, \tilde P_3$, that we now define and treat separately.
The first term is 
\[
\tilde P_1 := \frac{\P(\exists \, i \neq j \le n \, : \, \xi_i, \xi_j =  x_n)}{\P(M_n =  x_n)} \leq  \frac{n^2 \P(\xi =  x_n)^2}{\P(M_n = x_n)} \xrightarrow[n\to\infty]{} 0 \,,
\]
using~\eqref{atom}. The second term $\tilde P_2$ is
\begin{multline*}
\!\!\! \bigg| \frac{\P\big( R(\xi_1,\dots,\xi_n) \in A \cap (-\infty, x_n)^{n-1}, M_n = x_n\big)}{\P(M_n = x_n)}
        - \P\big( (\xi_1,\dots,\xi_{n-1}) \in A \cap (-\infty, x_n)^{n-1} \big) \bigg| \\
      =  \left| \frac{n\P(\xi = x_n)}{\P(M_n = x_n)} - 1 \right|
        \P\big( (\xi_1,\dots,\xi_{n-1}) \in A \cap (-\infty, x_n)^{n-1} \big)   \xrightarrow[n\to\infty]{} 0 \,, 
\end{multline*}
using again~\eqref{atom}.
The last term is $\tilde P_3 := \P(M_{n-1}\geq  x_n)$, which goes to $0$, thanks to~\eqref{mn>=xn}.
This concludes the proof of~\eqref{othersnormalMnlarge.local}. 
\qed

\appendix

\section{Discussions on Rozovskii's theorem}\label{sec:discussR}

In this section, we make some comments on Rozovskii's result, in several directions:
\begin{itemize}
    \item We discuss the condition under which Theorem~\ref{thm:rozo6} is true: Rozovskii states its result in terms of the function $q(\cdot)$, recall the definition~\eqref{defq}; we show in Section~\ref{sec:asympincreasing} that it is equivalent to~\eqref{cond:F0}, which is in turn equivalent to $\overline{F}$ being extended 
    regularly varying.
    \item We discuss how one may find equivalent (or sufficient) conditions to have~\eqref{cond}: our goal is to obtain a more tractable condition in order to get the large deviation asymptotics~\eqref{barvn}; this is the purpose of Proposition~\ref{sec:mainpropR} below.
    \item We provide some examples of distributions for which Rozovskii's result can be applied, some for which it cannot be applied.
\end{itemize}
Again, for simplicity of notation, we assume that $\mu=\E[\xi]=0$ in this appendix.

\subsection{About the condition \eqref{cond:F0} and extended/intermediate regular variation}
\label{sec:asympincreasing}

In \cite[Thm.~6]{R90}, Rozovskii assumes that the function $q$ defined in~\eqref{defq} by $q(x) = x^2\overline{F}(x)/\mm(x)$
verifies that there exists some $c>0$ such that $x^c q(x)$ is asymptotically equivalent to a non-decreasing  
function.
We now state some equivalent statements and in particular we show that it is equivalent to $\overline{F}$ being extended regularly varying, see~\eqref{def:extendedregvar}.

\begin{claim}
\label{claim:equivalentcondF}
Assume that $x\mapsto\mm(x)$ is slowly varying at $\infty$. The following statements are equivalent:
\begin{itemize}[leftmargin=1.5\parindent]
    \item[(i)] $\exists$ $c>0$ such that $x^c q(x)$ is  equivalent to a non-decreasing function;
    \item[(ii)] $\exists$ $c>0$ such that $x^c \overline F(x)$ is equivalent to a non-decreasing function, \textit{i.e.}\ \eqref{cond:F0};
    \item[(iii)] there is a constant $c>0$ and a function $y\mapsto \delta_y$ such that $\lim_{y\to\infty} \delta_y =0$, for which one has
\begin{equation}
\label{cond:F}
 \overline F(x) \leq \overline F (y) \leq (1+\delta_y)   \Big(\frac{x}{y}\Big)^c \, \overline F(x) \,\qquad \forall x\geq y \,;
\end{equation}
    \item[(iv)] The function $\overline{F}$ is extended regularly varying at infinity (see \eqref{def:extendedregvar}).
\end{itemize}
\end{claim}
\begin{proof}
$(i) \Rightarrow (ii)$.
If $x^c q(x)$ is asymptotically equivalent to a non-decreasing function $v(x)$,
then
\[
x^{c+2} \overline F(x) = x^{c} \mm(x) q(x) = (1+o(1)) \mm(x) v(x),
\]
with both $\mm(x)$ and $v(x)$ non-decreasing. Thus condition \eqref{cond:F0} follows. 

$(ii) \Rightarrow (i)$.
If $x^c \overline F(x)$ is  asymptotically equivalent to a non-decreasing function $v(x)$ (\textit{i.e.}  assuming \eqref{cond:F0}), then $x^c q(x) = (1+o(1)) v(x) x^2/\mm(x)$.
Note that $x^2/\mm(x)$ is asymptotically equivalent to $x^2/\overline\sigma^2(x)$ (recall that $\overline \sigma(x) := \E[(|\xi|\wedge x)^2]$) which is differentiable with derivative $\frac{2x}{\overline\sigma^2(x)} (1- \frac{x^2\P(|\xi|>x)}{\overline\sigma^2(x)}) \geq 0$ for $x$ large. Hence $x^c q(x)$ is asymptotic to a non-decreasing function. 

$(ii) \Rightarrow (iii)$. If $x^c \overline F(x)$ is  asymptotically equivalent to an increasing function $v(x)$, we can write $x^c \overline F(x) = (1+\gep_x) v(x)$, with $\lim_{x\to \infty} \gep_x =0$. Hence, for $y\leq x$, using that $v$ is increasing, we have
\[
\frac{y^c \overline{F}(y)}{x^c \overline{F}(x)} = \frac{1+\gep_y}{1+\gep_x} \frac{v(y)}{v(x)} \leq  \frac{1+|\gep_y|}{1-\inf_{x\geq y} |\gep_x|} \to1,\,\text{ as } y\to\infty,
\]
which proves the upper bound in~\eqref{cond:F} (the lower bound is trivial).

$(iii) \Rightarrow (ii)$. 
Let $v(y)= y^c \overline{F}(y)$ and $\tilde v(y) = \inf_{x\geq y} v(x)$.
Then, clearly, $\tilde v$ is non-decreasing and by~\eqref{cond:F} we have that 
$ (1+\delta_y)^{-1} v(y) \leq \tilde v(y) \leq v(y)$, so $v(y)$ and $\tilde v(y)$ are asymptotically equivalent. Hence $(ii)$ is verified.

$(iii) \Rightarrow (iv)$ is obvious from the definition of extended regular variation; note that~\eqref{cond:F} gives the upper and lower Matuszewska indices are $c$ and $1$ respectively.

$(iv) \Rightarrow (iii)$ follows by Potter's bound for extended regularly varying functions, see \cite[Thm.~2.3]{Cline94}: for any $\delta>0$, there is some $x_0>0$ such that for any $\lambda <1$ and $x\geq x_0$
$ 1 \leq \frac{\overline{F}(\lambda x)}{\overline{F}(x)} \leq  (1+\delta) \lambda^{-c-1}$,
where $c$ is the upper Matuszewska index;
the lower bound simply follows from the fact that $\overline{F}$ is non-increasing.
\end{proof}

\begin{claim}
\label{claim:equivalentcondF2}
The fact that function $\overline{F}$ is intermediate regularly varying (see \eqref{def:intermediateregvar})
is equivalent to the following:
there exists some non-decreasing function $\kappa:[1,\infty) \to [1,\infty)$ such that 
\begin{equation}
\label{cond:F2}
\overline F(x) \leq  \overline F (y) \leq   (1+\delta_y)\kappa\Big(\frac xy\Big)\, \overline F(x) \qquad \forall x\geq y \,,
\end{equation}
with $\kappa$ satisfying $\lim_{s\to1+}\kappa(s)=1$ and $\lim_{s\to\infty} \frac1s \log \kappa(s)=0$.
Additionally, there is a constant $c>0$ such that \begin{equation}\label{eqn:kpoly}
\kappa(s)\leq c s^{c}, \text{ for all } s\geq 1.
\end{equation}
\end{claim}

\begin{proof}
Note that \eqref{cond:F2} readily implies intermediate regular variation by verifying the definition.
For the reverse implication, we proceed exactly as for Claim \ref{claim:equivalentcondF}: 
using Potter's bound~\cite[Thm.~2.3]{Cline94} for intermediate regularly varying functions, we obtain that for all $\delta>0$ there is some $x_0>0$ such that for any $\lambda >1$ and $x\geq x_0$, 
$1 \leq \frac{\overline{F}(\lambda x)}{\overline{F}(x)} \leq  (1+\delta) \kappa(\lambda)^{-1}$ for some $\kappa(\lambda)$ with $\lim_{\lambda\to 1} \kappa(\lambda)=1$.

The fact that $\kappa$ grows at most polynomially follows directly from the representation theorem for intermediate regularly varying functions, see~\cite[Cor.~3.2I]{Cline94}.
\end{proof}

\subsection{About the condition~\eqref{cond} in Theorem~\ref{thm:rozo6}}
\label{sec:mainpropR}

We now find equivalent conditions to the condition \eqref{cond} in Theorem \ref{thm:rozo6}.
Recall that we defined
\[
\overline \sigma^2(x) =\E[(|\xi| \wedge x)^2]= 2 \int_0^x t \P(|\xi|>t) dt\,,
\]
and let $(a_n)_{n\geq 1}$ be a sequence defined by the relation \eqref{an}. Then, noticing that \begin{equation}
\label{sigmabartosigma}
\overline \sigma^2(x)=\sigma^2(x)+x^2\P(|\xi|> x),
\end{equation}
we have $\lim_{y\to\infty} \frac{y^2\P(|\xi|> y)}{\sigma^2(y)}=0$ thanks to \eqref{f<sigma}, so
\begin{equation}\label{ddbar}
 \overline \sigma^2(y)\sim \sigma^2(y).
\end{equation}
This shows that $(a_n)_{n\geq 1}$ defined by~\eqref{an}, \textit{i.e.}\ $a_n^2= n \overline \sigma(a_n)$, is a valid normalising sequence since it verifies~\eqref{def:an} (with $\alpha=2$).

\begin{proposition}\label{thm:bn8equiv}
Let $\alpha=2$, let $(a_n)_{n\geq 1}$ be defined by \eqref{an}, and assume that $\overline{F}$ is intermediate regularly varying.
Then the relation \eqref{cond} is equivalent to \eqref{final}, that is to
\begin{equation}
    \label{final2}
\sigma^2(x\sqrt{|\log q(x)|})- \sigma^2(x)=o\left(\frac{ \sigma^2(x)}{|\log q(x)|}\right)\qquad \text{ as } x\to\infty \,.
\end{equation}
\end{proposition}
The rest of Section \ref{sec:mainpropR} will be devoted to the proof of this proposition.

\subsubsection{Preliminary observations}

Let us first give some consequences of the intermediate regular variation of $\overline{F}$, which we recall is equivalent to~\eqref{cond:F2}.

\begin{claim}
Assume that $x\mapsto\mm(x)$ is slowly varying at infinity and that $\overline{F}$ is intermediate regularly varying.  Let $x\geq y>0.$ Then for $y$ large enough
\begin{equation}\label{qx/qy}
\left(\frac{y}{x}\right)^2\leq \frac{q(y)}{q(x)}\leq (1+\delta_y)\kappa\Big(\frac{x}{y}\Big)\frac{y}{x},
\end{equation}
where $\delta_y$ and $\kappa(\cdot)$ are the same as in \eqref{cond:F2}.
\end{claim}

\begin{proof}
Note that by~\eqref{cond:F2}, we have
$1\leq \frac{\overline F(y)}{\overline F(x)}\leq (1+\delta_y)\kappa (\frac{x}{y})$ for all $x\geq y>0$. Then we can use the definition~\eqref{defq} of $q$ to obtain
\[
\frac{\sigma^2(x)}{\sigma^2(y)}\left(\frac{y}{x}\right)^2\leq \frac{q(y)}{q(x)}\leq (1+\delta_y)\kappa\left(\frac{x}{y}\right)\frac{\sigma^2(x)}{\sigma^2(y)}\left(\frac{y}{x}\right)^2.
\]
Note that $\sigma^{2}(x)\geq \sigma^2(y)$ for $x\geq y>0$ using monotonicity, and $\frac{\sigma^2(x)}{\sigma^2(y)}\leq \frac{x}{y}$ for $y$ large enough using Potter's bound for slowly varying function. Then we obtain \eqref{qx/qy}. 
\end{proof}

\begin{lemma}
\label{lem:asymp}
Assume that $x\mapsto\mm(x)$ is slowly varying at infinity and that $\overline{F}$ is intermediate regularly varying. Then uniformly for $\frac{x}{C |\log q(x)|} \leq y \leq  x$,
we have 
\[
\label{asymp}
|\log q(x)|\sim |\log q(y)|\qquad \text{as } x\to\infty.
\]
In particular, recalling that $\omega_n := \frac{a_n}{\sqrt{|\log q(a_n)|}}$, we have $|\log q(\omega_n)|\sim |\log q(a_n)|$ as $n\to\infty$.

\noindent
Similarly, the above display also holds uniformly for $x\leq y\leq C x|\log q(x)|$.
\end{lemma}
\begin{proof}
We will prove only the first part as the second part follows similar lines of reasoning. Recall that $\lim_{x\to\infty}q(x)=0$, see \eqref{f<sigma}.
Let us assume that $x$ is large enough so that $\log q(x)<0$ and let us set $h:=x/y$, which verifies $1\leq h \leq C|\log q(x)|$.  
Using the second inequality in \eqref{qx/qy}, we have 
$$\frac{q(y)}{q(x)}\leq (1+\delta_y)\frac{\kappa(h)}{h}\leq 2\frac{\kappa(h)}{h}, \quad \text{ for } x \text{ large enough}.$$
Then 
$\log q(x)\geq \log q(y)-\log 2\kappa(h) + \log h$, so that
\[
|\log q(x)|\leq |\log q(y)|+\log 2\kappa(h) \leq  |\log q(y)|+\log 2\kappa( C|\log q(x)|) \,,
\]
using that $\kappa$ is non-decreasing.
On the other hand, the first inequality in \eqref{qx/qy} shows that  $\frac{q(y)}{q(x)}\geq  h^{-2}$,
which implies 
\[
|\log q(x)|\geq |\log q(y)|-2\log h \geq |\log q(y)|-2\log\kappa(C |\log q(x)|).
\]
Since $\lim_{x\to\infty}q(x)=0$ and recalling that $\lim_{s\to\infty} \frac 1s \log \kappa(s) =0$, see~\eqref{cond:F2}, this concludes the proof.
\end{proof}

\begin{remark}
\label{rem:threshold}
Thanks to Lemma~\ref{lem:asymp}, if $x_n\sim  c a_n \sqrt{|\log q(a_n)|}$ for some $c>0$, we get that 
$x_n \sim  c a_n \sqrt{|\log q(x_n)|}$.
Hence, using also that $\frac{x_n}{a_n} \sim \frac{c}{\sqrt{|\log q(a_n)|} }= q(x_n)^{o(1)}$, we obtain
\[
\overline{\Phi}\Big(\frac{x_n}{a_n}\Big) \sim \frac{a_n}{x_n\sqrt{2\pi}} e^{-\frac{x_n^2}{2 a_n^2}} =  q(x_n)^{c^2/2+o(1)} \,.
\]
On the other hand, recalling the definition~\eqref{def:an} of $a_n$,
\[
n \overline{F}(x_n) = \frac{n \mm(x_n)}{x_n^2} q(x_n) \sim \frac{ a_n^2 \mm(x_n)}{x_n^2 \mm(a_n)} q(x_n) =q(x_n)^{1+o(1)} \,,
\]
where we used Potter's bound since $\sigma(\cdot)$ is slowly varying. 
In conclusion, a sufficient condition to have that $\overline{\Phi}(\frac{x_n}{a_n}) = o(n\overline{F}(x_n))$ is $c>\sqrt{2}$.
\end{remark}

\subsubsection{Proof of Proposition \ref{thm:bn8equiv}}

Recall that $(a_n)_{n\geq 1}$ is defined by \eqref{an} and that we assumed that $\overline{F}$ is intermediate regularly varying. 
Let us recall the definition $\omega_n = \frac{a_n}{\sqrt{|\log q(a_n)|}}$ with $q(x) = \frac{x^2}{\mm(x)} \overline{F}(x)$ defined in~\eqref{defq}.

\begin{lemma}\label{ddd}
 The condition \eqref{cond} is equivalent to 
 \begin{align}\label{part1}
     \lim_{n\to\infty} nF(-\omega_n)&= 0,\\
\text{ and } \qquad \overline \sigma^2(a_n)- \sigma^2(\omega_n)& =o\left(\frac{ \overline \sigma^2(a_n)}{\left|\log q(a_n)\right|}\right).
\label{bob} 
\end{align}
The above conditions~\eqref{part1}-\eqref{bob} are further equivalent to 
\begin{equation}
\label{boo} \sigma^2(a_n)- \sigma^2(\omega_n)=o\left(\frac{ \sigma^2(\omega_n)}{\left|\log q(\omega_n)\right|}\right).
\end{equation}
\end{lemma}
\begin{proof}
Note that \eqref{cond} is equivalent to \eqref{part1} and 
$\lim_{n\to\infty} \big|\frac{n}{\omega_n^2}\sigma^2(\omega_n)-\frac{a_n^2}{\omega_n^2}\big| = 0$.
Using the definition~\eqref{an} of $a_n$, we can write
$
\frac{\overline \sigma^2(a_n)}{|\log q(a_n)|}= n^{-1} \omega_n^2,
$
so we obtain that \eqref{cond} is equivalent to \eqref{part1} and
\eqref{bob}. Next, we show that \eqref{boo} is equivalent to \eqref{part1}-\eqref{bob}.

\medskip
\noindent \textit{Step 1}. We first show that \eqref{boo} implies \eqref{part1}; we actually prove that \eqref{boo} implies
\begin{equation}\label{nomegan>0}
\lim_{n\to\infty} n\P(|\xi|\geq \omega_n)= 0.
\end{equation}
Note that we have
\begin{equation}\label{proba-sigma}
n\P(\omega_n<|\xi|\leq a_n) = \frac{n}{\omega_n^2} \omega_n^2\P(\omega_n<|\xi|\leq a_n) \leq  \frac{n}{\omega_n^2} (\sigma^2(a_n)-\sigma^2(\omega_n)) \,.
\end{equation}
If \eqref{boo} holds, then we have
\begin{equation}\label{3sigma}
\overline \sigma^2(a_n) \sim \sigma^2(a_n)\sim  \sigma^2(\omega_n).
\end{equation}
Using Lemma~\ref{lem:asymp}, we therefore get that
\begin{equation}\label{remainder} 
n^{-1}\omega_n^2 =\frac{\overline \sigma^2(a_n)}{|\log q(a_n)|}\sim \frac{\sigma^2(\omega_n)}{|\log q(\omega_n)|} \,.
\end{equation}
Hence \eqref{boo} implies that $\sigma^2(a_n)- \sigma^2(\omega_n)=o(n^{-1}\omega_n^2)$. Then using \eqref{proba-sigma}, we obtain $\lim_{n\to\infty} n\P(\omega_n<|\xi|\leq a_n) =0$.

It remains to show that $\lim_{n\to\infty} n\P(|\xi|> a_n)=0$.
But this is simply due to the fact that $n = \frac{a_n^2}{\overline\sigma^2(a_n)}$ (see  definition~\eqref{an}) and~\eqref{f<sigma}.
Thus we have shown that \eqref{boo} implies \eqref{part1}.

\medskip

\noindent \textit{Step 2}. We now show that \eqref{boo} implies \eqref{bob}.
Thanks to \eqref{sigmabartosigma} and \eqref{remainder}, we only have to show that 
\[
\overline \sigma^2(a_n) - \mm(a_n) = a_n \P(|\xi|>a_n) = o\left(\frac{ \overline \sigma^2(a_n)}{\left|\log q(a_n)\right|}\right) \,.
\]
Recall that~\eqref{boo} implies \eqref{nomegan>0}. Using that $a_n^2= n\overline\sigma^2(a_n)$, we get that \eqref{nomegan>0} is equivalent to  
$
\omega_n^2\P(|\xi|\geq \omega_n)=o\big(\frac{ \overline\sigma^2(a_n)}{|\log q(a_n)|}\big)$, as $n\to\infty$,
which is further equivalent to the following, using \eqref{remainder}, 
\[
\omega_n^2\P(|\xi|\geq \omega_n)=o\left(\frac{ \sigma^2(\omega_n)}{\left|\log q(\omega_n)\right|}\right),\qquad \text{as } n\to\infty \,.
\]
This concludes the proof since $\omega_n$ can be replaced by any sequence growing to infinity (one can use the argument that $\omega_{\lfloor cn\rfloor}\sim \sqrt{c}\omega_{n}$ for any fixed $c>0,$ which can be proved  using Lemma~\ref{lem:asymp} and the fact that $a_n$ also satisfies this property).

\medskip

\noindent\textit{Step 3}. Finally we show that \eqref{bob} implies \eqref{boo}. Note that \eqref{bob} implies \eqref{3sigma}.
Applying Lemma \ref{lem:asymp}, we find that \eqref{remainder} holds. We also observe that
\[
0\leq \sigma^2(a_n)- \sigma^2(\omega_n) \leq \overline \sigma^2(a_n)- \sigma^2(\omega_n) = o\left(\frac{ \overline \sigma^2(a_n)}{\left|\log q(a_n)\right|}\right) \,,
\]
where the last step is due to \eqref{bob} and \eqref{3sigma}. This completes the proof, thanks to~\eqref{remainder}.
\end{proof}

\begin{proof}[Proof of Proposition \ref{thm:bn8equiv}]
By Lemma \ref{ddd}, setting $y:=x/\sqrt{|\log q(x)|}$, the condition \eqref{cond} is equivalent to
\begin{equation}\label{oneb}
\sigma^2(x)- \sigma^2(y)=o\left(\frac{ \sigma^2(y)}{|\log q(y)|}\right)\,, \qquad \text{ as } x\to\infty \,.
\end{equation}
This is because \eqref{boo} still holds if we replace $(a_n)_{n\geq 1}$ by an arbitrary increasing sequence. We only need  to show that \eqref{oneb} is equivalent to \eqref{final2}. To this purpose, it suffices to show that 
\begin{equation*}\label{ylogtox}
\mm(y\sqrt{|\log q(y)|})-\mm(x)=o\left(\frac{\sigma^2(y)}{|\log q(y)|}\right)\,,
\end{equation*}
under either the assumption of \eqref{oneb} or \eqref{final2}. This is implied by the following claim, using that $y\sqrt{|\log q(y)|}\sim x$ thanks to Lemma \ref{lem:asymp} (and the fact that $\mm$ is non-decreasing).
\end{proof}

\begin{claim}
\label{claim:last}
If~\eqref{oneb} or if~\eqref{final} holds, then 
$\mm(x)-\mm(\frac12 x) =o\left(\frac{ \sigma^2(x)}{|\log q(x)|}\right)$; similarly we have $\mm(2x)-\mm(x) =o\left(\frac{ \sigma^2(x)}{|\log q(x)|}\right)$.
\end{claim}

\begin{proof}
Assume that~\eqref{oneb} holds; the proof is identical if we assume that~\eqref{final} holds instead.
First, since $y=x/\sqrt{|\log q(x)|}$,  we have that $y\leq \frac12 x$ for $x$ large. Hence,
\[
0\leq \mm(x)-\mm(\tfrac12 x) \leq \mm(x)-\mm(y) =o\left(\frac{ \sigma^2(x)}{|\log q(x)|}\right) \,,
\]
by~\eqref{oneb} and $\frac{ \sigma^2(x)}{|\log q(x)|}\sim \frac{ \sigma^2(y)}{|\log q(y)|}.$ Indeed the later holds due to \eqref{3sigma} and Lemma \ref{lem:asymp}. 
Moreover, we can replace $x$ by $2x$ in the above display: indeed, we have $\mm(2x)\sim \mm(x)$ (since $\mm(\cdot)$ is slowly varying) and $|\log q(2x)|\sim |\log q(x)|$, see Lemma~\ref{lem:asymp}.
\end{proof}

\begin{remark}\label{a-2a}
Assume that Proposition \ref{thm:bn8equiv}, \textit{i.e.}~\eqref{final2}, holds. If $x_n\sim ca_n\sqrt{|\log q(a_n)|}$ for some constant $c>0$, then recalling the asymptotics~\eqref{eq:xn>an>rn} of $r_n$ and using~\eqref{3sigma}, we have $r_n\sim\frac{a_n}{c\sqrt{|\log q(a_n)|}}$ and $\sigma^2(r_n)\sim \sigma^2(a_n)$ (using also that $\sigma^2(\cdot)$ is slowly varying). 
Therefore, we have $\sigma^2(r_n)\sim \sigma^2(a_n)$ uniformly for $a_n\leq x_n\leq ca_n\sqrt{|\log q(a_n)|}$, using that $\sigma^2(\cdot)$ is an increasing function (we also clearly have $r_n\sim a_n/c$ and so $\sigma^2(r_n)\sim \sigma^2(a_n)$, if $x_n\sim c a_n$).
\end{remark}

\subsection{A few examples}

First, let us give a generic example where Rozovskii's condition is verified.

\begin{example}
\label{ex:central}
Assume that $\overline F(x)\sim L(x)x^{-\beta}$ as $x\to\infty$, for some slowly varying function $L(\cdot)$ and some $\beta\geq 2$.
Assume also that the left tail is ``dominated'' by the right tail, in the sense that $F(-x)\sim c\overline{F}(x)$ for some $c\geq 0$; with $F(-x) = o(\overline{F}(x))$ if $c=0$.
This is for instance the case if $\xi =X-\E[X]$ for some non-negative random variable $X$.

If $\beta>2$, then we have $\E[|\xi|^{2+\delta}]<\infty$ for some $\delta >0$, so Nagaev's condition is verified. 
We will therefore focus on the case $\beta=2$: Nagaev's condition is not verified, but we will see that Rozovskii's condition~\eqref{cond} is verified for $a_n$ defined by $a_n^2 = n \overline\sigma^2(a_n)$.

If $\beta=2$, then we have $\overline\sigma^2(x) =2 \int_0^{x} t \P(|\xi|>t) d t$, with $\P(|\xi|>t) \sim (1+c)L(t)t^{-2}$ as $t\to\infty$. Denote $\ell(t) = 2 t^2 \P(|\xi|>t)$, so that
\begin{equation}\label{sigmaint}
\overline \sigma(x)^2 = \int_0^x \ell(t) t^{-1} \d t 
\end{equation}
and note that $\ell(x)\sim  2(1+c) L(x)$ as $x\to\infty$.
Then we have that $q(x) \sim  c' \ell(x)/\mm(x)$, with $\ell(x)/\mm(x)$ which is known to verify $\lim_{x\to\infty} \ell(x)/\mm(x)=0$, see~\cite[Prop.~1.5.9.a]{BGT89}.

Then, thanks to Proposition~\ref{thm:bn8equiv}, we simply need to verify that 
\begin{equation}
\label{rozov:goal}
 \mm(x\sqrt{|\log q(x)|}) -\mm(x) = o\left( \frac{\mm(x)}{|\log q(x)|} \right) \,.
\end{equation}
Since $\overline \sigma^2(x) -\mm(x) = x^2 \P(|\xi|>x) = \ell(x)$
and since $\ell(x) \sim 2(1+c) \mm(x) q(x) =o\left( \frac{\mm(x)}{|\log q(x)|} \right)$, it suffices to prove the above display with $\mm(\cdot)$ replaced by $\overline \sigma^2(\cdot)$ on the l.h.s.\ of~\eqref{rozov:goal}, thanks to Lemma \ref{lem:asymp}.
We can use Claim~\ref{claim:deHaan} below to get that
\[
\begin{split}
0\leq \overline\sigma^2(x\sqrt{|\log q(x)|}) -\overline\sigma^2(x)
& \leq \ell(x) \Big(\tfrac12\log |\log q(x)| + \gep_x \sqrt{|\log q(x)|}\Big) \\
&  \leq C \mm(x) q(x)  \sqrt{|\log q(x)|}  = o\left( \frac{\mm(x)}{|\log q(x)|} \right) \,,
\end{split}
\]
which concludes the proof.\qed
\end{example}

\begin{claim}
\label{claim:deHaan}
Let $\ell$ be a non-negative slowly varying function.
Then there exists $\gep_x$ that verifies $\lim_{x\to\infty} \gep_x =0$ such that for any $z\ge x$
\[
\int_x^z \frac{\ell(t)}{t} \d t \leq \ell(x) \Big(\log \Big(\frac zx\Big) + \gep_x \, \Big(\frac{z}{x}-1\Big)\Big) \,, \qquad \text{ as }x\to\infty \,.
\] 
\end{claim}

\begin{proof}
Setting $g(y) :=\ell(y) y^{-1}$, we get by a change of variable that
\[
\int_x^{z} \ell(t)t^{-1} \d t =  \ell(x) \int_1^{z/x} \frac{g(ux)}{g(x)} \d u \,.
\]
Now, using that $g$ is regularly varying of index $-1$, we have that uniformly for $u\geq 1$,  $g(ux)/g(x) = u^{-1} + \gep_x$ with $\lim_{x\to\infty}\gep_x = 0$ (this is the uniform convergence theorem, see~\cite{BGT89}).
We end up with 
\[
0\leq \int_x^{z} \ell(t)t^{-1} \d t \leq  \ell(x)\Big( \log \Big( \frac zx \Big)  + \gep_x\, \Big(\frac{z}{x}-1\Big) \Big) \,,
\]
which concludes the proof.
\end{proof}

We now present an example where the left tail is heavier than the right tail: we give some conditions for~\eqref{cond} to hold; this will also  stress the importance of the choice of the normalising sequence $(a_n)_{n\geq 1}$.

\begin{example}
\label{ex:lefttail}
Assume that $\xi$ is centered and that $\sigma^2 =\mathrm{Var}(\xi)<\infty$. We will focus on a case where Nagaev's condition $\E[|\xi|^{2+\delta}]<\infty$ for some $\delta>0$ is not satisfied.
We consider a case where $|\log q(x)|\asymp \log x$ as $x\to\infty$; this is ensured for instance if the right tail verifies $x^{-\beta_+} \leq \overline F(x) \leq x^{-\beta_-}$  for some $\beta_+>\beta_->2$ (for $x$ large).

Our goal is to discuss the conditions imposed by~\eqref{cond} on the left tail of the distribution. We assume that $F(-x) \sim \ell(x) x^{-2}$ for some slowly varying function $\ell(\cdot)$; think about taking $\ell(x) = (\log x)^{-a}$ for some $a>1$ (so that $\sigma^2<\infty$).
We now compare two possible choices for the normalising sequence $(a_n)_{n\geq 1}$.

If one defines $a_n=\sigma \sqrt{n}$, then recalling~\eqref{sqrtn}, \eqref{ddbar} and \eqref{sigmaint},  \eqref{cond} is equivalent to
\[
\int_x^{\infty} \frac{\ell(t)}{t} \d t  \sim \E[\xi^2 \ind_{\{|\xi|>x\}}] = o\left(\frac{1}{\log x}\right) \,,
\]
 (recall also that $|\log q(x)|\asymp \log x$).
Since $\ell(x) = (\log x)^{-a}$ for some $a>1$, this amounts to having $a>2$ for this particular choice of $a_n$.

If on the other hand one defines $a_n$ by $a_n^2=n \overline{\sigma}(a_n)$, then Proposition~\ref{thm:bn8equiv} tells that~\eqref{cond} is equivalent to
\[
\E[\xi^2 \ind_{\{x< |\xi|<x\sqrt{|\log q(x)|}\}}] \, \sim
\, \int_x^{x\sqrt{|\log q(x)}|} \frac{\ell(t)}{t} \d t = o\left(\frac{1}{\log x}\right)
\]
By Claim~\ref{claim:deHaan}, we have that the integral is $o(\ell(x) \sqrt{\log x})$ (using that $|\log q(x)|\asymp \log x$), so it is enough to have $\ell(x) \leq (\log x)^{-3/2}$ (or $a>3/2$).
If $\ell(x)$ is non-increasing, then the integral is bounded by a constant times $\ell(x) \log \log x$, so a sufficient condition becomes $\ell(x) = o( \frac{1}{\log x \log \log x} )$; this is verified if $\ell(x) = (\log x)^{-a}$ with any $a>1$.\qed
\end{example}

\begin{example}
\label{ex:lefttail2}
Let us consider the case where $\xi$ is centered and where the left and right tails verify $\overline F(x) \sim L(x) x^{-2}$ and $F(-x)\sim \ell(x) x^{-2}$, with the left tail being heavier, that is $L(x)=o(\ell(x))$ as $x\to\infty$.
In this case, $q(x)$ is slowly varying and the interplay between the slowly varying functions $L,\ell,\mm$ is subtle.
For simplicity, let us consider the case where $\ell(x) \sim (\log x)^{-a}$ for some $a\in \mathbb R$.
We will find conditions on the right tail for~\eqref{cond} to be verified.

If $\sigma^2 <\infty$ (\textit{i.e.}\ $a>1$), let us set for now $a_n=\sigma \sqrt{n}$.
We have that
\[
\E[\xi^2\ind_{\{|\xi|>x\}}] \sim \int_x^\infty \frac{\ell(t)}{t} \d t \sim  (\log x)^{1-a} \,,
\]
so \eqref{sqrtn} is equivalent to $|\log q(x)| = o( (\log x)^{a-1})$.
If $a\geq 2$, the fact that $q$ is slowly varying ensures that $|\log q(x)|=o(\log x)$, thus \eqref{sqrtn} holds. If $a\in (1,2)$, we have to check the right tail: for instance, if $L(x)\asymp e^{(\log x)^b}$ for some $b\in (0,1)$, then $|\log q(x)| \asymp (\log x)^b$ and the condition~\eqref{sqrtn} is verified provided that $a-1>b$.

If we define $a_n$ by the relation $a_n^2 = n \overline \sigma^2(a_n)$, we need to consider
\[
\E\big[\xi^2\ind_{\{x<|\xi|<x \sqrt{|\log q(x)|}\}} \big]
\sim \int_x^{x \sqrt{|\log q(x)|}} \frac{\ell(t)}{t} \d t 
\sim \frac12 (\log x)^{-a} \log |\log q(x)| \,.
\]
Therefore, \eqref{final} is equivalent to
$|\log q(x)| \log (|\log q(x)|) = o\left( \mm(x) (\log x)^a\right)$.
If $a>1$, this is automatic, since we have $|\log q(x)| = o(\log x)$ as noticed above.
If $a=1$, this is also verified thanks to the same observation, using also that $\mm(x)\sim \log\log x$.
If $a<1$, then $\mm(x) \asymp (\log x)^{1-a}$, so the condition becomes $|\log q(x)| \log |\log q(x)|= o(\log x)$;
this is verified for instance if in the right tail we have $L(x) \asymp e^{(\log x)^b} $ for some $b\in (0,1)$, but is not verified if one takes for instance $L(x) \asymp \exp( \log x/\log\log x)$.\qed
\end{example}

We now provide an example of a distribution in the domain of attraction of the normal law where one cannot find any normalising sequence $(a_n)_{n\geq 1}$ that verifies Rozovskii's condition~\eqref{cond}.

\begin{example}\label{condnottrue}
The example is constructed to make the left tail too heavy compared to the right tail. Take some function $q(x)$ and suppose that  $\sigma^2(x) = \exp\big(\log x /\log^{(3)}q(x) \big)$,  for $x$ large enough with $\log^{(3)} t =\log \log \log \frac1t$. Let $(a_n)_{n\geq1}$ be any normalising sequence, \textit{i.e.}\ verifying~\eqref{def:an}. 

Since $\omega_n =a_n /\sqrt{|\log q(a_n)|}$, we have that
\[
\frac{\log \omega_n}{\log^{(3)} q(\omega_n)} = \frac{\log a_n}{ \log^{(3)} q(\omega_n)}  - \frac{\frac12\log |\log q(a_n)|}{ \log^{(3)} q(\omega_n)}. 
\]
So if we assume that $\log^{(3)} q(\omega_n) -\log^{(3)} q(a_n)  = o(\log^{(3)} q(a_n)/\log a_n) $ (take for instance $q(y)=y^{-1}$ or $q(y)=1/\log y$), we get that
\[
\frac{\log \omega_n}{\log^{(3)} q(\omega_n)} = \frac{\log a_n}{ \log^{(3)} q(a_n)}  - \frac12 \frac{\log |\log q(a_n)|}{ \log^{(3)} q(a_n)} +o(1)\,. 
\]
Hence
\[
\sigma^2(\omega_n) \sim  e^{-\frac12 \frac{\log |\log q(a_n)|}{ \log^{(3)} q(a_n)}} \sigma^2(a_n) \,
\]
and
\[
\left|\frac{n}{w_n^2}\sigma^2(\omega_n)-\frac{a_n^2}{\omega_n^2}\right|
= \frac{a_n^2}{\omega_n^2} \left| \frac{n \sigma^2(\omega_n)}{a_n^2}-1\right| \,,
\]
with $\frac{a_n^2}{\omega_n^2} \to \infty$ and
\[
\frac{n \sigma^2(\omega_n)}{a_n^2} 
\sim  e^{-\frac12 \frac{\log |\log q(a_n)|}{ \log^{(3)} q(a_n)} }  \frac{n \sigma^2(a_n)}{a_n^2} \to 0.
\]
This shows that
$\left|\frac{n}{w_n^2}\sigma^2(\omega_n)-\frac{a_n^2}{\omega_n^2}\right|\to\infty$, which tells that \eqref{cond} does not hold.\qed
\end{example}

\section{One-big-jump phenomenon when $\alpha\in(0,2)$}
\label{app:stable}

In this section we prove Corollaries~\ref{cor:objnot=2} and \ref{cor:objnot=2local}.
We begin with the following lemma. 

\begin{lemma}
\label{noconstraint}
For any $\xi$ that is non-degenerate (\textit{i.e.}\ whose law is not concentrated on a single point), we have that
\begin{equation}\label{dtvnot0}d_{\mathrm{TV}}\Big( \mathscr{L}\big( R(\xi_1,\dots,\xi_n) \big), \,
    \big(\mathscr{L}(\xi)\big)^{\otimes (n-1)} \Big)\text{ does not converge to }0 \text{ as } n\to\infty.
\end{equation}
\end{lemma}

\begin{proof}
Let us denote 
\[
(\eta_1,\eta_2,\ldots,\eta_{n-1})=R(\xi_1,\xi_2,\ldots,\xi_n)
\]
and $N_n$ be the second largest variable in $\xi_1,\ldots,\xi_n$, so  $N_n=\max \{\eta_1,\ldots, \eta_{n-1} \}$. 
Define $T:=\sup\{t:\P(\xi>t)>0\}$ and let us distinguish two cases.

If $\P(\xi=T)>0$ it is easy to see that  
\[
\P(\eta_1=T)=\P(\xi_1=\xi_2=T)=\P(\xi=T)^2<\P(\xi_1=T).
\]
Thus \eqref{dtvnot0} holds in this case.

If $\P(\xi=T)=0,$ let $x$ converge to $T$ from the left and put $n=n(x)=\lfloor \frac{1}{\P(\xi >  x)}\rfloor.$ Then 
\[
\lim_{n\to\infty}\P(M_{n-1}\leq  x)=\lim_{n\to\infty} (1-\P(\xi>x))^{n-1}=e^{-1},
\]
and 
\[
\lim_{n\to\infty}\P(N_n\leq  x)=\lim_{n\to\infty}\P(M_{n}\leq  x)+\lim_{n\to\infty}n\P(\xi>x)\P(M_{n-1}\leq x)=2e^{-1}.
\]
Therefore 
\[
\lim_{n\to\infty}\P(N_n\leq  x)=e^{-1}\neq 2e^{-1} = \lim_{n\to\infty}\P(M_{n-1}\leq  x).
\]
Hence \eqref{dtvnot0} holds in this case as well.
\end{proof}

\subsection{Proof of Corollary \ref{cor:objnot=2}}
\label{sec:cor1.4}

The statement \eqref{sntomn} is already contained in the proof of \cite[Theorem 2.1]{B19}, so we focus on the proof of~\eqref{othersnormalSnlarge}.

The fact that $\lim_{n\to\infty} \frac{x_n}{a_n}=\infty$ is sufficient to obtain \eqref{othersnormalSnlarge}
can be seen by using \eqref{sntomn} and the statement \eqref{othersnormalMnlarge} in Proposition~\ref{prop:othersnormalMnl}. We will prove here the necessity.  

Let $(x_n)_{n\geq 1}$ be such that 
\begin{equation}
\label{xan}
\lim_{n\to\infty}\frac{x_n}{a_n} = l \in [-\infty, \infty) \,.
\end{equation}
If $l=-\infty,$ then 
\begin{equation}\label{prelemma} 
\lim_{n\to\infty} d_{\mathrm{TV}}\Big( \mathscr{L}\big( R(\xi_1,\dots,\xi_n) \, \big| \, S_n-b_n \geq  x_n \big), \, \mathscr{L}\big( R(\xi_1,\dots,\xi_n) \big) \Big) = 0,
\end{equation}
since $\lim_{n\to\infty}\P(S_n-b_n\geq x_n)=1$. By Lemma \ref{noconstraint}, \eqref{othersnormalSnlarge} does not hold. Next we assume $l\in(-\infty,\infty)$ and we prove that \eqref{othersnormalSnlarge} does not hold.
We split the proof according to several cases: recall our assumption~\eqref{tails}, which ensures the convergence of $S_n$ (properly centered and normalised) to an $\alpha$-stable random variable $\mathcal S_{\alpha}$, recall~\eqref{attract}.

(i) If $\alpha\in(0,1)$,  $p=1$, and $l\in (-\infty,0]$.
The limiting $\alpha$-stable random variable $\mathcal S_{\alpha}$ is then supported on $[0,\infty)$ (see~\cite[XIII.7 Thm. 2]{F71} and the following remark). Then we have that
$\lim_{n\to\infty}\P(S_n-b_n>x_n)=1$ and we deduce \eqref{prelemma}. By the same argument as in the case $l=-\infty$, we conclude that \eqref{othersnormalSnlarge} does not hold.

(ii) In the rest, we treat all the following cases in the same way (so we will not distinguish them in the following analysis):
\begin{equation}
\label{hardwork}
\begin{cases}
\alpha \in (0,1),\   p=1, \   l\in (0,\infty), \\[2pt]
\alpha\in(0,1),\   p\in (0,1), \  l\in(-\infty,\infty),\\[2pt]
\alpha\in [1,2), \  p\in (0,1], \  l\in (-\infty,\infty).
\end{cases}
\end{equation}
Let $\Lambda$ be the L\'evy measure associated to the $\alpha$-stable random variable $\mathcal S_{\alpha}$, given by $\Lambda(\dd x) = \frac{2-\alpha}{\alpha} |x|^{-\alpha-1} (p\ind_{\{x>0\}} +q\ind_{\{x<0\}}) \dd x$;
the characteristic function of $\mathcal S_{\alpha}$ is given by: 
\begin{equation}
\label{chara}
\phi(t):=\E[e^{it\mathcal S_{\alpha}}]=\exp\left(\int_{-\infty}^{\infty}(e^{itx}-1-it\sin(x))d\Lambda(x)\right), \quad t\in\R \,.
\end{equation}
Let $Y=(Y_u)_{u\geq 0}$ be a L\'evy process with L\'evy measure $\Lambda$ with $Y_0=0$, and denote by $\zeta_1,\zeta_2$ the largest and second largest jump sizes of the L\'evy process $Y$ in the interval $[0,1]$. Let $N_n$ be the second largest value of $\xi_1,\ldots,\xi_n$. Then by \eqref{chara} and the stable functional central limit theorem \cite[Thm.~4.5.3]{W02}, we have for $y\in \mathbb R$
\begin{equation}\label{pp1}
\lim_{n\to\infty}\P\big( a_n^{-1} N_n <y \,\big|\, S_n-b_n\geq x_n \big)=\P\big(\zeta_2<y \,\big|\, Y_1\geq l\big)
\end{equation}
and 
\begin{equation}\label{pp2}
\lim_{n\to\infty}\P\big(a_n^{-1} M_n < y\big)=\P\big(\zeta_1
< y\big)=\exp(-\Lambda_y),
\end{equation}
where $\Lambda_y=\Lambda([y,\infty)).$ 
To prove that \eqref{othersnormalSnlarge} does not hold, it suffices to show that there exists some $y>0$ such that 
$\lim_{n\to\infty}\P( a_n^{-1} N_n <y \,|\, S_n-b_n\geq x_n )\neq \lim_{n\to\infty}\P(a_n^{-1} M_n < y)$,
or in other words that there is some $y>0$ such that
\begin{equation}\label{c}
\P(\zeta_2<y \, |\, Y_1\geq l )\neq \P(\zeta_1<y)=\exp(-\Lambda_y).
\end{equation}

Considering jumps that occur on $[0,1]$, let $Y_1^{(1)}$ be the sum of all jump sizes that are not smaller than $y$;
let also $Y_1^{(2)}=Y_1-Y_1^{(1)}$. Then $Y_1^{(1)}$ and $Y_1^{(2)}$ are independent, with respective characteristic functions 
\[
\E\Big[e^{itY_1^{(1)}}\Big]=\exp\left(\int_{y}^{\infty}(e^{itx}-1)d\Lambda(x)\right), \quad t\in \R
\]
and 
\[
\E\Big[e^{itY_1^{(2)}}\Big]=\exp\left(\int_{y}^{\infty}-it\sin(x)d\Lambda(x)+\int_{-\infty}^{y}(e^{itx}-1-it\sin(x))d\Lambda(x)\right), \quad t\in \R.
\]

Let $V$ be a random variable with distribution $\Lambda\1_{[y,\infty)}/\Lambda_y$ which is independent of $Y_1^{(2)}$. Note that if $\zeta_1<y$ then $Y_1^{(1)}=0$: using this fact,  we obtain 
\begin{align*}\P(\zeta_2<y \, | \, Y_1\geq l)&=\frac{\P(\zeta_1<y, Y_1^{(2)}\geq l)}{\P(Y_1\geq l)}+\frac{\P(\zeta_2< y \leq  \zeta_1, \zeta_1+Y_1^{(2)}\geq l)}{\P(Y_1\geq l)}\\
&=\P(K=0)\frac{\P(Y_1^{(2)}\geq l)}{\P(Y_1\geq l)}+\P(K=1)\frac{\P(V+Y_1^{(2)}\geq l)}{\P(Y_1\geq l)},\end{align*}
where $K$ denotes the number of jumps with size not smaller than $y$. Note that $K$ follows the Poisson distribution $\mathrm{Poi}(\Lambda_y)$ with parameter $\Lambda_y$ and is independent of $Y_1^{(2)}$.
Then we have
\begin{equation}
\label{2compare}
\begin{split}
\P(\zeta_1<y)&-\P(\zeta_2<y \, | \, Y_1\geq l)\\
&  =\P(K=0)- \P(K=0)\frac{\P(Y_1^{(2)}\geq l)}{\P(Y_1\geq l)}-\P(K=1)\frac{\P(V+Y_1^{(2)}\geq l)}{\P(Y_1\geq l)}\\
&  =\frac{\exp(-\Lambda_y)}{\P(Y_1\geq l)}\left(\P(Y_1^{(2)}< l, \,Y_1^{(1)}+Y_1^{(2)}\geq l)-\Lambda_y\P(V+Y_1^{(2)}\geq l)\right)\\
& \leq \frac{\exp(-\Lambda_y)}{\P(Y_1\geq l)}\left(\P(K>0)\P(Y_1^{(2)}<l)-\Lambda_y\P(Y_1^{(2)}\geq l-y)\right),
\end{split}
\end{equation}
where we have used the fact that $V\geq y$.  Note that, as $y\to\infty$, we have that
$Y_1^{(2)}$ converges in distribution to $Y_1$
and $\P(Y_1^{(2)}\geq l-y)$ goes to $1$.
Since $\Lambda_y\to 0$, as $y\to\infty$, we use the above display to obtain that, as $y\to\infty$,
\begin{align*}
\P(K>0)\P(Y_1^{(2)}<l)-\Lambda_y\P(Y_1^{(2)}\geq l-y) &=\big(1-e^{-\Lambda_y} \big)\P(Y_1^{(2)}<l)-\Lambda_y\P(Y_1^{(2)}\geq l-y)\\
&  =-\Lambda_y\P(Y_1\geq l)+o(\Lambda_y) \,.
\end{align*}
Therefore we have that
$\P(K>0)\P(Y_1^{(2)}<l)-\Lambda_y\P(Y_1^{(2)}\geq l-y)<0$
for $y$ large enough.
Going back to \eqref{2compare}, we end up with 
\[
\P(\zeta_1<y)-\P(\zeta_2<y \, | \, Y_1\geq l)<0
\]
for $y$ large enough, which proves \eqref{c}.
This completes the proof for \eqref{othersnormalSnlarge}. \qed

\subsection{Proof of Corollary \ref{cor:objnot=2local}}
\label{sec:cor1.7}
 
In the following, to simplify notation, we assume that $b_n=0$.
The proof that $\lim_{n\to\infty} \frac{x_n}{a_n} =\infty$ is a sufficient condition to get~\eqref{othersnormalSnlarge=x} follows the same line of proof as that of Theorem~\ref{thm:bigjumplocal}, see Section~\ref{sec:proofBJlocal} (following ideas of~\cite{AL09}), so we skip it.

\smallskip
We now show that $\lim_{n\to\infty} \frac{x_n}{a_n} =\infty$ is necessary. Assume that that~\eqref{xan} holds, \textit{i.e.}\ that $\lim_{n\to\infty} \frac{x_n}{a_n} = l \in [-\infty,\infty)$. We will distinguish two set of cases (i) and (ii). 

(i) Consider first the cases
\begin{equation}\label{easypart}
\begin{cases}
\alpha\in (0,1),\quad  p=1, \quad  l\in [-\infty,0], \\[2pt]
\alpha\in (0,1),\quad  p\in (0,1), \quad l=-\infty, \\[2pt]
\alpha\in[1,2), \quad p\in (0,1], \quad l=-\infty.\end{cases}
\end{equation}
If \eqref{othersnormalSnlarge=x} holds, due to \eqref{attract}, we have that
\[
\LL\Big(\frac{S_n-M_n-b_n}{a_n}\;\Big|\; S_n- b_n =x_n\Big)\stackrel{d}{\longrightarrow}\mathcal S_{\alpha} ,
\quad \text{ so } \quad  \LL\Big(\frac{M_n}{a_n}\;\Big|\; S_n- b_n =x_n\Big)\stackrel{d}{\longrightarrow} l-\mathcal S_{\alpha} \,.
\]
Note that $l-\mathcal S_{\alpha}<0$ for the first case in \eqref{easypart} (recall that $\mathcal S_{\alpha}$ is supported on $[0,\infty)$) and that $l-\mathcal S_{\alpha} =-\infty$ in the other cases. Notice also that $\frac{b_n}{n a_n}$ and $\frac{x_n}{na_n}$ both go to $0$ as $n\to\infty$ (indeed $b_n=0$ if $\alpha \in (0,1)$ and $b_n\leq (1+o(1)) n \E[|\xi|\ind_{|\xi|\leq a_n}]$ if $\alpha\in [1,2)$, with $\E[|\xi|\ind_{|\xi|\leq a_n}]$ slowly varying):
the above display therefore implies that 
\[
\lim_{n\to\infty}\P\Big(M_n < - \frac{b_n+x_n}{n}  \, \Big| \, S_n- b_n  =x_n \Big) =1 \,.
\]
But we clearly have that $\P(M_n< \frac{1}{n}(b_n+x_n), S_n =b_n+x_n) =0$, which brings a contradiction.
We conclude that \eqref{othersnormalSnlarge=x} is not true for the cases in~\eqref{easypart}.

(ii) Next we consider the cases in \eqref{hardwork}. Let $g_{\alpha}$ be the density function of $\mathcal S_{\alpha}$: by Stone's local limit theorem~\cite{S67}, we have
\[
\P(S_n- b_n  =x_n)=\frac{1}{a_n} \big( g_{\alpha}(l) +o(1)\big), \qquad \text{ as } n\to\infty.
\]
Here $g_{\alpha}(l)>0$ for any $l$ in the ranges given in \eqref{hardwork} in different cases. 
Choose any $\epsilon \in (0,1)$ and any $t\in \R$ such that $\P(\xi>t)>\epsilon$. Then $\P(M_n\leq t)<(1-\epsilon)^n$.
Combining with the above, we get that
\[
\P(M_n\leq t\,\, |\,\, S_n- b_n  =x_n) \leq \frac{\P(M_n\leq t)}{\P(S_n- b_n  =x_n)} \leq c (1-\epsilon)^n a_n \,,
\]
which goes to $0$ as $n$ goes to $\infty$, since $a_n$ is regularly varying.
We therefore obtain that 
$\lim_{n\to\infty}\P(M_n>t\, |\, S_n- b_n  =x_n)=1$,
which in turn implies that 
\[
\lim_{n\to\infty}\P\big(S_n-M_n- b_n <x_n-t\, \big|\,S_n- b_n  =x_n\big)=1.
\]
However if \eqref{othersnormalSnlarge=x} holds, given that $S_n-M_n$ is the sum of $R(\xi_1,\ldots,\xi_n)$, we would have 
\[
\lim_{n\to\infty}\P\big( S_n-M_n- b_n <x_n-t\,\big |\, S_n- b_n =x_n\big)=\P(\mathcal S_{\alpha}<l)\in(0,1),
\]
which is a contradiction. This proves that \eqref{othersnormalSnlarge=x} is not true in cases~\eqref{hardwork}.
This concludes the proof.
\qed

\section*{Acknowledgement}
The authors thank Jochen Blath, Denis Denisov, Sergey Foss, Dalia Terhesiu and Vitali Wachtel for their friendly comments and pointers to the literature.
We also thank an anonymous referee whose insightful comments helped us improve the presentation.

\printbibliography

\end{document}